\documentclass[11pt]{article}
\scrollmode

\usepackage{amsfonts}
\usepackage{amsmath}
\usepackage{amsthm}
\usepackage{amssymb}
\usepackage{enumerate}
\usepackage{bbm}
\usepackage{graphicx}

\allowdisplaybreaks

\title{Deviation inequalities for random walks}
\date{}
\author{P.~Mathieu \footnote{Aix Marseille Universit\'e, CNRS, Centrale Marseille, I2M, UMR 7373, 13453 Marseille France, 
pierre.mathieu@univ-amu.fr}\\
A.~Sisto \footnote{Department of Mathematics, ETH Zurich, 8092 Zurich, Switzerland,
sisto@math.ethz.ch}	}

\def\eps{\varepsilon}
\def\P{{\mathbb P}}
\def\esp{{\mathbb E}}
\def\var{{\mathbb V}} 
\def\Z{{\mathbb Z}}
\def\N{{\mathbb N}}

\def\R{{\mathbb R}}

\def\1{{\mathbf 1}}

\def\o{\omega}

\def\x{\Gamma}

\def\Nu{\bold{n}}

\def\p{{\cal P}}

\def\x1{X_1^{(1)}}

\def\tx0{\tilde X^0}

\def\tmu{\tilde{\mu}}

\def\={&=&}
\def\+{&+&}

\def\ra{\rightarrow}

\newcommand{\stack}[2]{\genfrac{}{}{0pt}{3}{#1}{#2}}

\newtheorem{theo}{Theorem}[section]
\newtheorem{prop}[theo]{Proposition}
\newtheorem{lm}[theo]{Lemma}
\newtheorem{cor}[theo]{Corollary}
\newtheorem{rmk}[theo]{Remark}
\newtheorem{df}[theo]{Definition}
\def\beq{\begin{equation}}
\def\eeq{\end{equation}}
\newcommand{\bei}{\begin{itemize}}
\newcommand{\eei}{\end{itemize}}
\newcommand{\ben}{\begin{enumerate}}
\newcommand{\een}{\end{enumerate}}
\newcommand{\beqn}{\begin{eqnarray}}
\newcommand{\beqnn}{\begin{eqnarray*}}
\newcommand{\eeqn}{\end{eqnarray}}
\newcommand{\eeqnn}{\end{eqnarray*}}
\newcommand{\brm}{\begin{rmk}}
\newcommand{\erm}{\end{rmk}}

\newcommand{\calG} {\ensuremath {\mathcal{G}}}

\newcommand{\calP} {\ensuremath {\mathcal{P}}}

\newcommand{\matP} {\ensuremath {\mathbb{P}}}
\newcommand{\matE} {\ensuremath {\mathbb{E}}}

\begin{document}\maketitle


\begin{abstract}
We study random walks on groups with the feature that, roughly speaking, successive positions of the walk tend to be ``aligned''. We formalize and quantify this property by means of the notion of deviation inequalities. 
We show that deviation inequalities have several consequences including Central Limit Theorems, the local Lipschitz continuity of the rate of escape and entropy, as well as linear upper and lower bounds on the variance of the distance of the position of the walk from its initial point. 
In a second part of the paper, we show that the (exponential) deviation inequality holds for measures with exponential tail on acylindrically hyperbolic groups. These include non-elementary (relatively) hyperbolic groups, Mapping Class Groups, many groups acting on CAT(0) spaces, and small cancellation groups. 
\end{abstract}


\smallskip




\section{Introduction} 

In this paper, we discuss fluctuations results for random walks on ``hyperbolic-like'' groups. 
In the sequel, $G$ denotes an infinite, countable and discrete group with neutral element $id$ equipped with a left-invariant metric $d$ and we let  $\mu$ denote a probability measure on $G$. We are interested in the behaviour of the random walk $(Z_n)_{n\geq 0}$ with driving measure $\mu$ and starting at $Z_0=id$. The sequence $(Z_n)_{n\geq 0}$ is obtained as successive products of random variables $(X_j)_{j\geq 1}$ that are independent and with law $\mu$. We let $\P^\mu$ be the law of the random walk and $\esp^\mu$ be the corresponding expectation. Precise definitions and some background on random walks are given in sections \ref{sec:motivation} and \ref{sec:rws}. 

We will work in the class of \emph{acylindrically hyperbolic groups} as defined in \cite{Os-acyl}. This is a large class of groups that includes non-elementary hyperbolic as well as relatively hyperbolic groups, Mapping Class Groups, groups acting on proper CAT(0) spaces with a rank-one element, and small cancellation groups. We refer to sections \ref{sec:acylin} and \ref{prelim} for definitions and references. 

Our main result is a Central Limit Theorem for the distance of the walk to its starting point $d(id,Z_n)$. 

The classical approach for the C.L.T., used in all the references we are aware of and whose most powerful version is in \cite{kn:bq}, assumes $G$ has a nice action on some compact space $X$ and derives the C.L.T. as a special case of a Central Limit Theorem for cocycles defined on $X$. When $G$ is hyperbolic, $X$ is the horo-functions (Busemann) boundary of $(G,d)$. In the case $G$ is a Mapping Class Group, $X$ is the boundary of Teichm\"uller space. 

In this paper, we take a radically different approach. Our strategy for proving the C.L.T. directly relies on deviation inequalities and completely avoids considering any boundary or compactification. Furthermore it only uses elementary geometric and probabilistic arguments.  

Deviation inequalities measure how much successive positions of the walk fail to be ``aligned''. Thus, for a random walk satisfying a deviation inequality, given two integers $n$ and $m$, with high probability, the distance $d(id,Z_{n+m})$ is close to the sum $d(id,Z_n)+d(Z_n,Z_{n+m})$. See Definition \ref{df:dev} and (\ref{eq:devlength}), (\ref{eq:devexplength}). Since the random variables $d(id,Z_n)$ and $d(Z_n,Z_{n+m})$ are independent, deviation inequalities can be used to compare the distance $d(id,Z_{n+m})$ with a sum of independent random variables. Eventually one thus proves that, in full generality, the second moment deviation inequality is sufficient to get a Central Limit Theorem (Theorem \ref{theo:clt}).

Deviation inequalities are well-adapted to deal with random walks on groups with hyperbolic features. We prove deviation inequalities if $G$ is hyperbolic and if $\mu$ satisfies some moment condition, Theorem \ref{devhyp}. We also obtain deviation inequalities for acylindrically hyperbolic groups when $\mu$ has a finite exponential moment in Corollary \ref{cor:main}. 

Combining the results mentioned above, we deduce in particular the following: 

\begin{theo} \label{theo:intro} 
Let $G$ be an acylindrically hyperbolic group and let $\mu$ be a probability measure on $G$ with a finite exponential moment and so that the semigroup generated by the support of $\mu$ is $G$. Let $d$ be a word metric on $G$. Then:\\ 
- (Deviation inequality) there exists a constant $\tau_0>0$ such that for all integers $n$ and $m$ and for all positive $c$, we have 
$$\P^\mu[d(id,Z_{n+m})-d(id,Z_n)-d(Z_n,Z_{n+m})\geq c]\leq \tau_0^{-1}e^{-\tau_0 c}\,,$$ 
and\\
- (Central limit theorem)  the law of 
$\frac 1{\sqrt{n}} (d(id,Z_n)-\esp^\mu[d(id,Z_n)])$ under $\P^\mu$ weakly converges to a Gaussian non-degenerate law. 
\end{theo} 

Theorem \ref{theo:intro} actually applies to more general measures $\mu$, as well as more general metrics $d$ than stated in Theorem \ref{theo:intro}, see Theorem \ref{theo:CLT_conclusions}. The description of all metrics we can deal with involves the notion of acylindrically intermediate spaces, see Definition \ref{def:acylinter}. 
For instance, they include the Teichm\"uller and Weil-Petersson metrics on Mapping Class Groups, see Proposition \ref{prop:examplesacylinter}.  Moreover, we recover the Central Limit Theorem for hyperbolic groups proven in \cite{kn:bq} as a combination of the deviation inequality for hyperbolic group (Theorem \ref{devhyp}) and Theorem \ref{theo:clt}.

Many results apply to other functionals than the distance to the initial point of the walk. We introduce the definition of ``defective adapted cocycles'', see Definition \ref{df:dac} and state a C.L.T. in this context, Theorem \ref{theo:clt} . It includes the case of quasimorphisms, first obtained in \cite{CLT_quasimorphisms}. 

Our theorems extend previously known results in two directions. On the one hand we deal with driving measures $\mu$ with an infinite support where most authors assume finite support or super-exponential moments. Most results are therefore new even for hyperbolic groups. 
On the other hand, we cover the case of acylindrically hyperbolic groups, which is vastly broader than that of hyperbolic groups.  
Such extensions require a new approach and motivated us to develop the theory of deviation inequalities. 

Deviation inequalities are also instrumental for other purposes than the C.L.T. 
As a first step in our proof of the C.L.T. we establish linear bounds on the variance. We also provide estimates of higher moments. We also use deviation inequalities to study fluctuations of the rate of escape and entropy in terms of the driving measure $\mu$. We prove they are Lipschitz continuous and differentiable and we identify the derivative. All statements are given in the Conclusions at the end of the paper (Section \ref{sec:statements}).

\medskip
{\bf Context} 

The history of Central Limit Theorems on groups with hyperbolic features can be traced back to early result of Furstenberg and Kesten in the sixties. We refer to \cite{kn:bq_linear} for historical background. The most recent contributions to this subject will be found in \cite{kn:bq}, where the C.L.T. is proved when $G$ is hyperbolic and $\mu$ has a finite second moment, and in \cite{Horbez:CLT}, where a similar C.L.T. is proved for random walks on Mapping Class Groups and $Out(F_n)$. 

The question of the regularity of the rate of escape or the entropy in terms of the driving measure was first addressed in \cite{kn:Er}. Gou\"ezel recently proved in \cite{Gouezel:analyticity} that, on a hyperbolic group $G$, the entropy and rate of escape are analytic functions on the set of probability measures $\mu$ with a given finite support. Once again the approach used in \cite{Gouezel:analyticity}, as well as in previous references like \cite{kn:led1} and \cite{kn:led2}, uses some boundary theory. In particular one needs to know the Martin boundary of the walk. (We recall that, when $G$ is hyperbolic and $\mu$ is finitely supported, then \cite{kn:anc} and \cite{kn:anc2} showed that the Martin boundary coincides with the Gromov boundary.)  

The situation is substantially more involved when considering random walks with a driving measure with a finite exponential moment. We recall that on any non-elementary hyperbolic group one can construct a random walk whose driving measure has a finite exponential moment but whose Martin boundary is not the Gromov boundary of the group, see \cite{Gou-Martin}.  

The situation is even more dramatic when $G$ is only assumed to be acylindrically hyperbolic. Then we have to deal with actions of $G$ on a hyperbolic spaces $X$ that need not be co-compact. It seems the best the Benoist-Quint approach could give would be a C.L.T. in $X$. As for the C.L.T. in $G$, one needs a different approach; we use deviation inequalities. 

Note that deviation inequalities were previously proved for hyperbolic groups and driving measures with finite support in \cite{kn:bhm2} as a consequence of quasi-conformal properties of the harmonic measure. We obtain them here in a completely different way. 

The linear progress property that we use to prove deviation inequalities,  was obtained by Maher and collaborators in \cite[Theorem 5.35]{CM-scl}, \cite[Theorem 1.2]{MaherTiozzo} for measures with bounded support (in the space being acted on). For our applications we need to deal with measures with exponential tail. We do not know whether the strategy of the aforementioned papers extends to this case.

The kind of control of the geometry of random walk paths we need to establish deviation inequalities for acylindrically hyperbolic groups is reminiscent of the estimates of tracking rates in \cite{Si-tracking}. 

There is a  connection between the Central Limit Theorem and the regularity of the rate of escape. This is a rather general fact that has little to do with hyperbolicity. It appears in \cite{kn:hype} in the context of random walks on hyperbolic groups with a finitely supported driving measure and we shall exploit it here too. 

Random walks on acylindrically hyperbolic group have been recently studied in \cite{MaherTiozzo} where it is shown that their Poisson boundary coincides with the Gromov boundary of $X$. This result was already known for hyperbolic groups, see \cite{kn:kaim}.

\medskip
{\bf Deviation inequalities}

Using deviation inequalities, we are able to completely avoid considering any compactification of our group or the space it acts on. This paper is 98 per cent self-contained. The main proofs rely on a combination of quite elementary probabilistic and geometric arguments (not much more than Markov's inequality and the triangle inequality indeed).

By definition, a random walk satisfies a deviation inequality in a given metric if the Gromov product between its initial point and two successive positions remains of order one in probability. Note that such a property cannot hold for almost all paths of the walk. (Typically one gets logarithmic divergences of the Gromov products.) We define different types of deviation inequalities depending on whether we require exponential or polynomial control over the tail of the law of the Gromov products, see Subsection \ref{subsec:dev} for this definition and Section \ref{sec:rws} for the definition in the case of defective adapted cocycles. 

Deviation inequalities imply some almost additivity properties: the law of large numbers can be made quantitative (Lemma \ref{lm:lln}). More is true: since the distance (or the cocycle) is almost additive along the trajectories of the walk, we deduce that the influence of a given increment is small. It then follows from an Efron-Stein type argument that the variance is sub-linear (Theorem \ref{theo:uppervariance}) and almost additive (Theorem \ref{theo:existvariance}).  The Central Limit Theorem itself follows from some strengthening of these additive properties: let us try to approximate the distance $d(id,Z_n)$ by a sum of independent random variables. The first layer of such an approximation would be to simply write the sum of the distances between successive positions of the walk. Doing so, we made an error. The second layer is then to take into account corrections due to the Gromov products of successive triplets of successive positions, avoiding overlaps to maintain some independence. This is still not perfect but we get better approximations by considering further corrections corresponding to Gromov products involving positions that correspond to times that differ by two (still avoiding overlaps). Eventually, running this process a finite but large number of times leads to a decomposition of the distance $d(id,Z_n)$ as a sum of i.i.d. random contributions plus some error, see Subsection \ref{ssec:cltfordac} . This error is expressed in terms of Gromov products and, under the assumption of a deviation inequality, we show it has a small variance. Thus the proof of the C.L.T. boils down to (and only relies on!) the C.L.T. for i.i.d. random variables.

We now describe our strategy for the proof of deviation inequalities for random walks on acylindrically hyperbolic groups. Most arguments start with a group $G$ acting on a (hyperbolic) metric space $X$ and are in two steps:

\begin{enumerate}
 \item {[Probabilistic step]} a sample path has with high probability a certain property $P$ (this step usually does not use specific geometric properties of $X$).
 \item {[Geometric step]} due to specific geometric properties of $X$ and of the action of $G$, any path with property $P$ is ``close to being a geodesic'' in the appropriate sense.
\end{enumerate}

For example, (part of) property $P$ can be that the distance between the $i$-th and $j$-th position, measured in the metric of $X$, is linear in $|i-j|$.

As a first step towards the proof of the deviation inequalities, we establish a linear progress property that says that, with overwhelming probability, the random walk tends to move away from its initial position at linear speed (Theorem \ref{linprog}). Note that this statement is proved in $X$ and not in the group itself. (Observe that the conclusion of Theorem \ref{linprog} is obviously true for random walks on non-amenable groups for a word metric. But since we need it in $X$, we have to face the fact that the distance in $X$ may not be proper. Then hyperbolicity and acylindricity are important.) 

We next deduce bounds on the probability that the distance between the position of the walk at some intermediate time $k$ and a quasi-geodesic from the identity to the position at time $n$ is larger than a given parameter (Theorem \ref{smalldevhier}). The argument is: due to the hyperbolicity of $X$ and due to the linear progress property, this distance can only be big if the walk performs very large jumps.  

Following such a strategy we obtain the deviation inequality for random walks on acylindrically hyperbolic groups with a driving measure with exponential tail in Corollary \ref{cor:main}. In the case $G$ itself is hyperbolic, we obtain a sharper control on the deviations in terms of the tail of $\mu$ in Theorem \ref{devhyp}.

\medskip
{\bf Organization of the paper}
The first part of the paper introduces deviation inequalities and investigates their consequences. It is titled ``Using deviation inequalities''. 

In Section \ref{sec:motivation} we recall the definitions of the rate of escape and the entropy and provide some general references about random walks on groups. 
In Section \ref{sec:rws}, we fix our notation, define defective adapted cocycles and deviation inequalities and state the law of large numbers (Theorem \ref{theo:lln} and Lemma \ref{lm:lln}). 
Section \ref{sec:clts} is about the Central Limit Theorem for defective adapted cocycles satisfying deviation inequalities (Theorem \ref{theo:clt}): we first estimate the variance (Theorem \ref{theo:uppervariance}) and prove it has a limit (Theorem \ref{theo:existvariance}) using a sub-additivity argument. The C.L.T. itself follows from an approximation of the distance $d(id,Z_n)$ by a sum of independent random variables, see Subsection \ref{ssec:cltfordac}. Subsection \ref{sec:highmoments} contains estimates of higher moments (Theorem \ref{theo:othermoments}). In Subsection \ref{sec:lowerbound} we obtain lower bounds for the variance (Theorem \ref{theo:downvariance}) that ensure that our C.L.T. is non-degenerate. 

In Section \ref{sec:fluctuatrescp}, we address the question of the regularity of the rate of escape in terms of the driving measure. We define a distance between probability measures with a fixed support in Subsection \ref{ssec:distance} and recall deviation inequalities. Subsection \ref{ssec:backrescp} contains some discussion of references. The main results are Theorem \ref{theo:rescp} about the Lipschitz regularity of the rate of escape and Theorem \ref{theo:diffrate} about its differentiability. Similar issues are discussed for the entropy in Section \ref{sec:entrop}. The main results are Theorem \ref{theo:entropsym} about the Lipschitz regularity of the entropy and Theorem \ref{theo:diffentrop} about its differentiability. The connection between the entropy and a rate of escape is done using Green metrics that are recalled in Subsection \ref{par:Green}. One needs some control on how the Green metric fluctuates with respect to $\mu$ and this is done in Proposition \ref{prop:flucgreen}. 

In the second part of the paper, titled ``Getting deviation inequalities'', we establish deviation inequalities for different classes of random walks. 

Section \ref{sec:acylin} starts with some references about acylindrically hyperbolic groups. The definition is given at the beginning of Section \ref{prelim} along with a first consequence (Lemma \ref{geomsep}). 
In Section \ref{sec:linearpr}, we prove that random walks on acylindrically hyperbolic groups tend to escape from their initial position, see Theorem \ref{linprog}. In Section \ref{sec:deviationqg}, we define acylindrically intermediate spaces (Definition \ref{def:acylinter}) and provide examples (Proposition \ref{prop:examplesacylinter}). The main result is Theorem \ref{smalldevhier} that gives a control on how much the trajectories of the random walk deviate from quasi-geodesics. An immediate consequence is the deviation inequality for random walks on acylindrically hyperbolic groups in Corollary \ref{cor:main}. Section \ref{sec:hypgroups} is dedicated to deviation inequalities in hyperbolic groups, and the main result there is Theorem \ref{devhyp}. In Section \ref{sec:greendev} we go back to acylindrically hyperbolic groups and study deviation inequalities in the Green metric (Theorem \ref{smalldevgen}).

Finally, in Section \ref{sec:statements} we collect all results about acylindrically hyperbolic groups that can be proven combining results in Part \ref{I:deviationinequalities} with results in Part \ref{part:geom}.

\subsection*{Acknowledgements} The authors would like to thank R. Marchand and O. Garet for pointing out the reference \cite{Hammersley:subadditive}. 
We also express our gratefulness to P. Ha\"\i ssinsky for interesting discussions on the subject of this paper and anonymous referees for insightful comments. 

\part{Using deviation inequalities}\label{part:prob} 

\section{Motivation}\label{sec:motivation}

Let $G$ be an infinite, discrete group with neutral element $id$;  
let $\mu$ be a probability measure on $\cal G$. Let $\mu^n$ denote the $n$-th convolution power of $\mu$. 

If the support of $\mu$ generates an infinite semi-group, then the sequence $\mu^n$ converges to $0$. 
One may measure the rate of decay of $\mu^n$ through the notion of entropy. 

Let $H(\mu):=\sum_{x\in G} (-\log \mu(x))\,\mu(x)\,$ be the entropy of $\mu$ and assume it is finite. 
Then the sequence $(H(\mu^n))_{n\in\N}$ is sub-additive and the following limit exists:  
\beqn\label{eq:0} 
h(\mu):=\lim_{n\ra\infty}\frac 1 n H(\mu^n)\,.\eeqn 
The quantity $h(\mu)$ is called the {\bf asymptotic entropy} of $\mu$. 

Another quantity of interest is the rate of escape: let $d$ be a left-invariant metric on $G$. 
Assume that $\mu$ has a finite first moment in the metric $d$, namely that  
$\sum_{x\in G} d(id,x)\mu(x)<\infty$. Then the sequence $(\sum_{x\in G} d(id,x)\mu^n(x))_{n\in\N}$ is sub-additive. Therefore the limit 
\beqn\label{eq:1} \ell(\mu;d):=\lim_{n\ra\infty}\frac 1 n \sum_{x\in G} d(id,x)\mu^n(x)\eeqn exists; it is called the {\bf rate of escape} of $\mu$ in the metric $d$. 
Thus the rate of escape gives the mean distance to the identity of a random element of $G$ sampled from the distribution $\mu^n$. 

We observed in \cite{kn:bhm1} that, when $\mu$ is symmetric, then  the entropy coincides with the rate of escape for a special choice of the distance $d$ called the Green metric. Details 
on the Green metric are given in Section \ref{sec:entrop}.

\medskip

The notion of asymptotic entropy was introduced by A. Avez in \cite{kn:avez} in relation with random walk theory. In \cite{kn:avez2}, Avez proved that, whenever $h(\mu)=0$, then $\mu$ satisfies the Liouville property: bounded, $\mu$-harmonic functions are constant. The converse was proved later, see \cite{kn:derrie} and \cite{kn:kaiver}. 
The Liouville property is equivalent to the triviality of the asymptotic $\sigma$-field of the random walk with driving measure $\mu$ (its so-called Poisson boundary), see \cite{kn:derrie} and \cite{kn:kaiver} again. In more general terms, the entropy plays a central role in the identification of the Poisson boundary of random walks in many examples. We refer in particular to \cite{kn:kaim} for groups with hyperbolic features. In this latter case, the asymptotic entropy is also related to the geometry of the harmonic measure through a ``dimension-rate of escape- entropy'' formula, see \cite{kn:bhm2} and the references quoted therein. 

The notion of rate of escape is also related to the potential theory on $G$. Assume that $d$ is proper and that the support of $\mu$ generates the whole group. 
One shows, see \cite{kn:KL}, that if the probability measure $\mu$ has a finite first moment and is such that $\ell(\mu;d)>0$ then at least one of the following two properties must hold: 
i) there exists a homomorphism from $G$ to $\R$, say $H$, such that the image of $\mu$ through $H$ has non zero mean; 
ii) the Poisson boundary is non trivial, i.e. there exist non constant bounded $\mu$-harmonic functions on $G$.

In this paper, we shall be mostly concerned with non-amenable groups and assume that the support of $\mu$ generates $G$. In that case the Poisson boundary is never trivial.

\medskip 

Both the rate of escape and the asymptotic entropy have simple interpretations in terms of random walks: let $(Z_n)_{n\ge 0}$ be a random walk with driving measure $\mu$. 
This means that the increments $X_n:=Z_{n-1}^{-1}Z_n$ are independent random variables with common law $\mu$. 
Kingman's sub-additive theorem implies that 
\beqn\label{eq:king}\ell(\mu;d)=\lim_{n\ra\infty}\frac 1 n d(id,Z_n)\,\hbox{\ (respectively\ }\, h(\mu)=\lim_{n\ra\infty}-\frac 1 n \log \mu^n(Z_n) \,\hbox{\ )\ }\,,\eeqn 
where both limits hold for almost any path of the random walk.

\medskip 

In this paper, we are mostly concerned with fluctuations in the ergodic limits in (\ref{eq:king}), where the word ``fluctuations''
can be understood as ``stochastic fluctuations with respect to the trajectories of the walk'' or ``fluctuations of the rate of escape or entropy'' 
with respect to $\mu$. Both types of fluctuations are actually closely related to each other, as we shall see.  

The rate of escape and the entropy can be interpreted as ``approximate cocycles'' on a product space. In the sequel we shall develop quite general tools for handling 
functionals of independent random variables that satisfy a cocycle relation up to some error. The error terms are dealt with through what we call ``deviation inequalities''.

\section{Random walks, cocycles, and deviation inequalities} \label{sec:rws} 

\subsection{Product spaces...} 

Let $G$ be a measurable space equipped with a $\sigma$-field $\cal G$. 
Let $\mu$ be a probability measure on $G$. 

Consider the product space $\Omega:=G^{\N^*}$ where $\N^*=\{1,..\}$. 
Let $(X_n)_{n\in\N^*}$ designate the coordinate maps from $\Omega$ to $G$: thus $X_n(\o)=\o_n$ 
for any sequence $\o=(\o_1,...)\in\Omega$ and $n\in\N^*$. 

Following the usual convention in probability theory we often omit to indicate that random functions, as  $X_n$ or $Q_n$ below, depend on $\o$.

We equip $\Omega$ with the product $\sigma$-field (i.e. the smallest $\sigma$-field for which all functions $X_n$ are measurable). 
We endow $\Omega$ with the product measure $\P^\mu:=\mu^{\N^*}$. 
We use the notation $\esp^\mu$ to denote the expectation with respect to $\P^\mu$ and  $\var^\mu$ to designate the variance with respect to $\P^\mu$: 
$\var^\mu[F]=\esp^\mu[F^2]-\esp[F]^2$. 

Let ${\cal F}_n$ denote the $\sigma$-field generated by the random variables $X_1,...,X_n$ (i.e. the smallest $\sigma$-field on $\Omega$ 
for which all functions $(X_j)_{j\le n}$ are measurable). 

Let $\theta$ be the canonical shift on $\Omega$ and $\theta_n:=\theta^n$; namely, $\theta_n((\o_1,\o_2,...))=(\o_{n+1},\o_{n+2},...)$.

We will refer to this framework as the ``general case''. 

\subsection {...and random walks} 

We specialize the above framework to the case of an infinite, countable, discrete group $G$.  

We now consider the sequence of $G$-valued functions $(Z_n)_{n\in\N}$ 
recursively defined by $Z_0(\o)=id$ and $Z_n(\o)=Z_{n-1}(\o)X_n(\o)$ for $n\ge 1$. 
We think of a sequence $(Z_n(\o))_{n\in\N}$ as describing a trajectory in the group $G$.  
Thus $Z_n(\o)$ gives the position of the trajectory at time $n$, while 
$X_n(\o)=(Z_{n-1}(\o))^{-1}Z_n(\o)$ gives its increment also at time $n$. 

Observe that the law of the sequence $(Z_n)_{n\in\N}$ under $\P^\mu$ is the law of a random walk driven by $\mu$ and 
starting at $id$:  its increments are independent and identically distributed random variables of law $\mu$. In particular, for any $n\in\N$,  
the law of $Z_n$ under $\P^\mu$ is the $n$-th convolution power of $\mu$,  $\mu^n$. 

Observe that $Z_m\circ \theta_n=Z_n^{-1}Z_{n+m}$. 

We shall refer to the framework we just described as the ``random walk case'' to distinguish it from the ``general case''. 
Unless otherwise specified, our framework is the ``general case''.

\subsection{Defective adapted cocycles} 

Although we are primarily interested in studying the distance of a random walk from its initial point at large times, 
our main results hold in the ``general case'' for a certain class of ``approximate cocycles'' that we define below. 

\begin{df} \label{df:dac}
A {\bf defective adapted cocycle (D.A.C.)} is a sequence of  real-valued maps on $\Omega$, say ${\cal Q}=(Q_n)_{n\in\N^*}$, where each 
map $Q_n$ is measurable with respect to the $\sigma$-field  ${\cal F}_n$. By convention we take $Q_0$ to be identically $0$. 
The {\bf defect} of $\cal Q$ is the collection of maps 
$\Psi=(\Psi_{n,m})_{(n,m)\in\N\times\N}$ defined by 
$$\Psi_{n,m}(\omega)=Q_{n+m}(\o)-Q_n(\o)-Q_m(\theta_n\o)\,.$$
\end{df}

{\bf Examples of defective adapted cocycles:} 

Here are the most important examples of D.A.C. to be considered in the sequel. 

1. Let $f$ be a measurable function from $G$ to $\R$ and let $M_n(\omega)=\sum_{j=1}^n f(X_j(\omega))$. Then the sequence ${\cal M}:=(M_n)_{n\in\N^*}$ 
is a defective adapted cocycle and its defect vanishes. 

2. Consider the ``random walk case''. A special class of D.A.C. are those for which there exists a map $q$ from $\Omega$ to $\R$ such that $Q_n(\o)=q(Z_n(\o))$. 
We call them {\bf end-point D.A.C.}. 

Define the differential $\partial q(g,h):=q(gh)-q(g)-q(h)$ and observe that 
$$\Psi_{n,m}=\partial q(Z_n,Z_n^{-1}Z_{n+m})\,.$$ 
(This follows from the identity $Z_m\circ \theta_n=Z_n^{-1}Z_{n+m}$.) 

When the function $\partial q$ is uniformly bounded with respect to $g$ and $h$, then the function $q$ is called a {\bf quasimorphism}. 

Let $d$ be a left-invariant metric on $G$ (e.g. a word metric). 
Then the maps $Q_n(\o)=d(id,Z_n(\o))$ define an end-point D.A.C. Its defect is 
$\Psi_{n,m}(\o)=-2(id,Z_{n+m}(\o))_{Z_n(\o)}$ where 
$$(x,y)_w:=\frac 12 (d(w,x)+d(w,y)-d(x,y))\,$$ 
is the {\bf Gromov product} of points $x,y\in G$ with respect to the reference point $w\in G$ in the metric $d$. 
We call this D.A.C. the {\bf length D.A.C.} (for the metric $d$). 

\subsection{Deviation inequalities} \label{I:deviationinequalities}

By ``deviation inequality'' we mean some control on how much a D.A.C. fails to be a true cocycle with respect to $\P^\mu$. 
In the case of the length D.A.C., it will give control on how much the successive positions of the random walk fail to follow a ``straight line''. 

\begin{df}\label{df:dev} 
Let $\mu$ be a probability measure on $G$. Let ${\cal Q}=(Q_n)_{n\in\N^*}$ be a defective adapted cocycle with defect $\Psi=(\Psi_{n,m})_{(n,m)\in\N\times\N}$. 
 
Let $p>0$. 
We say that $\cal Q$ satisfies the {\bf $p$-th-moment deviation inequality} (with respect to the measure $\mu$) if there exists a constant $\tau_p(\mu;{\cal Q})$ such that 
for all 
$n$ and $m$ in $\N$  then 
\beqn \label{eq:dev} \esp^{\mu}[\vert \Psi_{n,m}\vert^p]\leq \tau_p(\mu;{\cal Q})\,.\eeqn 
We say that $\cal Q$ satisfies the {\bf exponential-tail deviation inequality} (with respect to the measure $\mu$) if there exists a constant $\tau_0(\mu;{\cal Q})$ such that 
for all 
$n$ and $m$ in $\N$  and for all $c>0$, we have 
\beqn \label{eq:devexp} \P^{\mu}[\vert\Psi_{n,m}\vert\ge c]\leq \tau_0(\mu;{\cal Q})^{-1} e^{-\tau_0(\mu;{\cal Q})c}\,.\eeqn 
\end{df} 

Clearly the $p$-th-moment deviation inequality implies the $p'$-th-moment deviation inequality whenever $p\ge p'$ and the exponential-tail deviation inequality implies 
the $p$-th-moment deviation inequality for all $p>0$.

Let $p>0$. We say that $\cal Q$ has {\bf finite $p$-th moment} with respect to $\mu$ if 
$\esp^\mu[ \vert Q_1\vert^p]<\infty$. We say that $\cal Q$ has {\bf exponential tail} if there exists $\alpha>0$ such that 
$\esp^\mu[e^{\alpha \vert Q_1\vert}]<\infty$. 
We use the notation 
$$\chi_p(\mu;{\cal Q}):=\esp^\mu[ \vert Q_1\vert^p]\,.$$ 

{\bf Example} 

In the ''random walk case'', choose for $\mathcal Q$ the length D.A.C. for a certain left-invariant metric $d$. Recall the notation $(x,y)_w$ for the Gromov product. 

Thus we say that $\mu $ satisfies the {\bf $p$-th-moment deviation inequality} (in the metric $d$)  if there exists a constant $\tau_p(\mu;d)$ such that 
for all 
$n$ and $m$ in $\N$  then 
\beqn \label{eq:devlength} \esp^{\mu}[(id,Z_{n+m})_{Z_n}^p]\leq 2^{-p} \tau_p(\mu;d)\,.\eeqn 
We say that $\mu$ satisfies the {\bf exponential-tail deviation inequality} (in the metric $d$) if there exists a constant $\tau_0(\mu;d)$ such that 
for all 
$n$ and $m$ in $\N$  and for all $c>0$, then 
\beqn \label{eq:devexplength} \P^{\mu}[(id,Z_{n+m})_{Z_n}\ge c/2]\leq \tau_0(\mu;d)^{-1} e^{-\tau_0(\mu;d)c}\,.\eeqn 

Let $p>0$. We say that $\mu$ has {\bf finite $p$-th moment} with respect to $d$ if 
$\sum_{x\in G} d(id,x)^p\mu(x)<\infty$. We say that $\mu$ has {\bf exponential tail} if there exists $\alpha>0$ such that 
$\sum_{x\in G} e^{\alpha d(id,x)}\mu(x)<\infty$.
We use the notation 
$$\chi_p(\mu;d):=\sum_{x\in G} d(id,x)^p\mu(x)\,.$$

\subsection{Laws of large numbers for defective adapted cocycles} 
\label{ssec:lln}

The following result is a consequence of a more general ergodic theorem proved by Y. Derriennic \cite{derrie:ergodic}.

\begin{theo} \label{theo:lln} 
Let $\cal Q$ be a defective adapted cocycle with a finite first moment and satisfying the first moment deviation inequality with respect 
to the probability measure $\mu$. Then there exists 
a real number $\ell(\mu;{\cal Q})$ such that the sequence 
$\frac 1 n Q_n$ converges to $\ell(\mu;{\cal Q})$ in $L_1(\P^\mu)$. 
\end{theo} 

We call $\ell(\mu;{\cal Q})$ the {\bf rate of escape} of the D.A.C. $\cal Q$ with respect to $\mu$. 
\medskip

\begin{proof} 
We apply Derriennic's result (Th\'eor\`eme 1 in \cite{derrie:ergodic}). 

First we should check that all $Q_n$'s are integrable and that $\inf_n\frac 1 n\esp^\mu[Q_n]>-\infty$. 
  
We start from the identity 
$Q_{n+m}=Q_n+Q_m\circ\theta_n+\Psi_{n,m}$ and observe that $Q_m\circ\theta_n$ has the same law as $Q_m$. 
A simple induction argument based on the fact that $\esp^\mu[\vert\Psi_{n,m}\vert]<\infty$ for all $n$ and $m$ shows that $\esp^\mu[\vert Q_n\vert]<\infty$ for all $n$. 
It also proves that $\sup_n \frac 1 n \esp^\mu[\vert Q_n\vert]<\infty$.

The other condition to be checked is that $\esp^\mu[\vert\Psi_{n,m}\vert]\le c_m$ for a sequence $c_m$ such that $\frac 1 m c_m\ra 0$. 
With our assumptions, one may take $c_m=\tau_1(\mu;{\cal Q})$. 
\end{proof} 

\begin{lm}\label{lm:lln} 
Let $\cal Q$ be a defective adapted cocycle with a finite first moment and satisfying the first moment deviation inequality with respect 
to the probability measure $\mu$. Then, for all $n\ge 1$, 
\beqn\label{eq:lln} 
\vert\frac 1n \esp^\mu[Q_n]-\ell(\mu;{\cal Q})\vert\leq \frac 1 n \tau_1(\mu;{\cal Q})\,.
\eeqn 
\end{lm}

\begin{proof} 

Taking the expectation in the identity $Q_{n+m}=Q_n+Q_m\circ\theta_n+\Psi_{n,m}$, we get that 
$\vert\esp^\mu[ Q_{n+m}]-\esp^\mu[Q_n]-\esp^\mu[Q_m]\vert\le\tau_1(\mu;{\cal Q})$. Therefore the two sequences 
$a_n:=\tau_1(\mu;{\cal Q})-\esp^\mu[Q_n]$ and $b_n:=\tau_1(\mu;{\cal Q})+\esp^\mu[Q_n]$ are both subadditive. 

By Theorem \ref{theo:lln}, $\frac 1 n a_n$ converges to $-\ell(\mu;{\cal Q})$. The subadditivity implies that $\frac 1 n a_n\ge -\ell(\mu;{\cal Q})$ 
for all $n\ge 1$. Similarly, we get that $\frac 1 n b_n \ge \ell(\mu;{\cal Q})$. In other words, 
$\frac 1 n (\tau_1(\mu;{\cal Q})-\esp^\mu[Q_n])\ge  -\ell(\mu;{\cal Q})$ and $\frac 1 n (\tau_1(\mu;{\cal Q})+\esp^\mu[Q_n])\ge \ell(\mu;{\cal Q})$. 
These inequalities imply (\ref{eq:lln}). 
\end{proof}


\section{Central limit theorems for defective adapted cocycles} \label{sec:clts}

Recall from section \ref{ssec:lln}, that a D.A.C. with finite first-moment assumption and first-moment deviation inequality satisfies the law of large numbers with rate of escape 
$$\ell(\mu;{\cal Q})=\lim_{n\to\infty}\frac 1 n \esp^\mu[Q_n]\,.$$ 

The main results in this section are the following. In particular, the second one is a Central Limit Theorem for defective adapted cocycles.

\begin{theo}\label{theo:existvariance} 
Let ${\cal Q}$ be a defective adapted cocycle. Assume that $\cal Q$ has a finite second moment and satisfies the second-moment deviation inequality 
with respect to the probability measure $\mu$. 
Then the variance 
$\frac 1 n \var^\mu(Q_n)$ has a limit as $n$ tends to $\infty$. We denote it by 
$$\sigma^2(\mu;{\cal Q}):=\lim_{n\to\infty} \frac 1 n\var^\mu[Q_n]\,.$$
\end{theo}

\begin{theo} \label{theo:clt} 
Let ${\cal Q}$ be a defective adapted cocycle. Assume that $\cal Q$ has a finite second moment and satisfies the second-moment deviation inequality 
with respect to the probability measure $\mu$. 
Then the law of $\frac 1{\sqrt{n}} (Q_n-\ell(\mu;{\cal Q})n)$ under $\P^\mu$ weakly converges to the Gaussian law with zero mean and variance 
$\sigma^2(\mu;{\cal Q})$. 
\end{theo} 

\begin{rmk}
Observe that it might be the case that $\sigma^2(\mu;{\cal Q})=0$. Even in such a case, Theorems \ref{theo:existvariance} and \ref{theo:clt} remain true. The limit law in Theorem  \ref{theo:clt}  is then the Dirac mass at $0$. 

In Section \ref{sec:lowerbound} we give sufficient conditions for the asymptotic variance of a D.A.C. to be positive, see Theorem \ref{theo:downvariance}. 
\end{rmk} 

\subsection{Existence of the variance of D.A.C.: proof of Theorem \ref{theo:existvariance}} 
\label{ssec:existvariance} 

We start with an a priori estimate on the variance of $Q_n$. 

\begin{theo}\label{theo:uppervariance} 
Let ${\cal Q}$ be a defective adapted cocycle. Assume that $\cal Q$ has a finite second moment and satisfies the second-moment deviation inequality 
with respect to the probability measure $\mu$. Define $C^+(\mu;{\cal Q}):=4\chi_2(\mu;{\cal Q})+16 \tau_2(\mu;{\cal Q})$. 
Then 
\beqn\label{eq:upvar} \var^\mu(Q_n)\le C^+(\mu;{\cal Q}) n\,.\eeqn for all $n\ge 1$.
\end{theo} 

Below we give a proof of Theorem \ref{theo:uppervariance} using the Efron-Stein inequality. An alternative proof is given in Remark \ref{rmk:uppervariance}. 
The Efron-Stein inequality and some of its extensions are discussed in Section \ref{sec:highmoments} and Remark \ref{rmk:efronstein}. 
 
 \medskip 
 
The proof of Theorem \ref{theo:uppervariance} is based on the following {\bf replacement trick}. 
Let $k\ge 1$ and let $X'_k$ be a random variable with law $\mu$ and independent of the sequence $(X_j)_{j\ge 1}$. Let 
$X^{(k)}_j$ be the sequence obtained when replacing $X_k$ by $X'_k$ in the sequence $(X_j)_{j\ge 1}$. In other words 
$X^{(k)}_j=X_j$ for all $j\not=k$ and $X^{(k)}_k=X'_k$. We now denote with $\theta$ the shift operator operating on the two sequences $(X_j)_{j\ge 1}$ and $(X^{(k)}_j)_{j\ge 1}$. 

Recall that, for all $n\ge 1$, then $Q_n$ is measurable with respect to ${\cal F}_n$. Therefore $Q_n$ is of the form 
$Q_n=f(X_1,...,X_n)$ for some measurable function $f$. We define $Q^{(k)}_n:=f(X^{(k)}_1,...,X^{(k)}_n)$. 
Note that $Q^{(k)}_n=Q_n$ for $n<k$. Also note that $Q^{(k)}_n\circ\theta_m=Q_n\circ\theta_m$ for all $m\ge k$. 

Then ${\cal Q}^{(k)}:=(Q^{(k)}_n)_{n\ge 1}$ is a D.A.C. associated to the sequence of random variables $(X^{(k)}_j)_{j\ge 1}$. We let 
$\Psi^{(k)}_{n,m}:=Q^{(k)}_{n+m}-Q^{(k)}_n-Q^{(k)}_m\circ\theta_n$ be its defect. 

For $n\ge k$, we have 
\beqnn 
Q^{(k)}_n&=&Q^{(k)}_{k-1}+Q^{(k)}_{n-k+1}\circ\theta_{k-1}+\Psi^{(k)}_{k-1,n-k+1}\\
&=& Q^{(k)}_{k-1}+Q^{(k)}_{1}\circ\theta_{k-1}+Q^{(k)}_{n-k}\circ\theta_{k}+\Psi^{(k)}_{1,n-k}\circ\theta_{k-1}+\Psi^{(k)}_{k-1,n-k+1}\,,\eeqnn 
and similarly 
\beqnn 
Q_n=
Q_{k-1}+Q_{1}\circ\theta_{k-1}+Q_{n-k}\circ\theta_{k}+\Psi_{1,n-k}\circ\theta_{k-1}+\Psi_{k-1,n-k+1}\,.\eeqnn 
Since $Q^{(k)}_{k-1}=Q_{k-1}$ and $Q^{(k)}_{n-k}\circ\theta_{k}=Q_{n-k}\circ\theta_{k}$, we get that  

 $$Q^{(k)}_n-Q_n=$$
 \beqn\label{eq:difference} Q^{(k)}_{1}\circ\theta_{k-1}-Q_{1}\circ\theta_{k-1}
 +\Psi^{(k)}_{1,n-k}\circ\theta_{k-1}-\Psi_{1,n-k}\circ\theta_{k-1}
 +\Psi^{(k)}_{k-1,n-k+1}-\Psi_{k-1,n-k+1}\,.
 \eeqn 
 
 As a consequence 
  $$\vert Q^{(k)}_n-Q_n\vert\le$$
\beqn\label{eq:differencee}\vert Q^{(k)}_{1}\circ\theta_{k-1}\vert +\vert Q_{1}\circ\theta_{k-1}\vert 
 +\vert \Psi^{(k)}_{1,n-k}\circ\theta_{k-1}\vert +\vert \Psi_{1,n-k}\circ\theta_{k-1}\vert 
 +\vert \Psi^{(k)}_{k-1,n-k+1}\vert +\vert \Psi_{k-1,n-k+1}\vert \,.
 \eeqn 

Note that both terms $Q^{(k)}_{1}\circ\theta_{k-1}$ and $Q_{1}\circ\theta_{k-1}$ have the same law as $Q_1$. 
Similarly, $\Psi^{(k)}_{1,n-k}\circ\theta_{k-1}$ and $\Psi_{1,n-k}\circ\theta_{k-1}$  have the same law as $\Psi_{1,n-k}$, 
and $\Psi^{(k)}_{k-1,n-k+1}$ and $\Psi_{k-1,n-k+1}$ have the same law. 

\medskip 

{\it Proof of Theorem \ref{theo:uppervariance}.}
Taking the expectation of the square in (\ref{eq:differencee}), we deduce that 
\beqnn 
\esp^\mu[(Q^{(k)}_n-Q_n)^2]
\le \big(2\sqrt{\chi_2(\mu;{\cal Q})} +4\sqrt{\tau_2(\mu;{\cal Q})}   \,\big)^2
\le 8\chi_2(\mu;{\cal Q})+32\tau_2(\mu;{\cal Q})\,.\eeqnn 

By the Efron-Stein inequality, see \cite{kn:steele}, we have: 
\beqnn 
\var^\mu[Q_n]&\leq& \frac 12 \sum_{k=1}^n \esp^\mu[(Q^{(k)}_n-Q_n)^2]\\ 
&\leq& n (4\chi_2(\mu;{\cal Q})+16\tau_2(\mu;{\cal Q}))
\,.\eeqnn 
\qed

{\it Proof of Theorem \ref{theo:existvariance}.}

We will use the following fact.

\begin{lm}\cite{Hammersley:subadditive}
\label{lm:limsubadd}
 Let $(a_n)_{n=1,\dots}$ be a sequence of real numbers so that there exists $b\geq 0$ with the property that $a_{n+m}\leq a_n+a_m +b \sqrt{n+m}$ for each $m,n\geq 1$. Then $\frac{a_n}{n}$ converges to some $L<+\infty$.
\end{lm}

 The starting point is the identity 
 $$Q_{n+m}=Q_n+Q_m\circ\theta_n+\Psi_{n,m}\,.$$
 
The two terms $Q_n$ and $Q_m\circ\theta_n$ are independent and $Q_m\circ\theta_n$ has the same distribution as $Q_m$. 
Therefore 
$$\var^\mu[Q_n+Q_m\circ\theta_n]=\var^\mu[Q_n]+\var^\mu[Q_m]\,.$$

We can now apply the inequality
$$|\var^\mu(A+B)-\var^\mu(A)|\leq \var^\mu(B)+2\sqrt{\var^\mu(A)\var^\mu(B)}$$
with $A=Q_n+Q_m\circ\theta_n$ and $B=\Psi_{n,m}$. Theorem \ref{theo:uppervariance} yields 
\beqnn 
\vert \var^\mu[Q_{n+m}]-\var^\mu[Q_n]-\var^\mu[Q_m]\vert 
&\le& \var^\mu[\Psi_{n,m}]+2\sqrt{\var^\mu[Q_n]+\var^\mu[Q_m]}\sqrt{\var^\mu[\Psi_{n,m}]}\\
&\le& \tau_2(\mu;{\cal Q})+2\sqrt{C^+({\cal Q};\mu)(n+m)}\sqrt{\tau_2(\mu;{\cal Q})}\,,
\eeqnn 
since $\var^\mu[\Psi_{n,m}]\leq \esp^\mu[\Psi_{n,m}^2]$.

Lemma \ref{lm:limsubadd} implies that $\lim_{n\to\infty}\frac 1 n\var^\mu[Q_n]$ exists and is finite, as required.
\qed

\subsection{A C.L.T. for D.A.C.: proof of Theorem \ref{theo:clt}} \label{ssec:cltfordac}

In this section we prove our Central Limit Theorem, Theorem \ref{theo:clt}.

We split  time $n$ into successive blocks of length $2^M$ and express $Q_n$ 
as a sum of the contributions of the  different blocks, plus some cross terms. The cross terms are expressed in terms of defects and can be controlled by the deviation inequality. Thus 
we conclude that $Q_n$ can be approximated by a sum of i.i.d. random variables and therefore that its law is close to Gaussian. 

\begin{lm}\label{lm:sumsiid} 
Let ${\cal Q}$ be a defective adapted cocycle. Assume that $\cal Q$ has a finite second moment and satisfies the second-moment deviation inequality 
with respect to the probability measure $\mu$. Then 
\beqn\label{eq:bd4}
\lim_{M\ra\infty}\limsup_{n\ra\infty} \frac 1 n 
\var^\mu[Q_n-\sum_{j=0}^{\lceil n2^{-M} \rceil-1}Q_{2^M} \circ\theta_{j2^M}]=0\,.\eeqn 
\end{lm}

{\it Proof of (\ref{eq:bd4}).} 

First consider an integer $n$ of the form $2^k$. 
Iterating the identity 
$$Q_{2n}=Q_n+Q_n\circ\theta_n+\Psi_{n,n}\,,$$ we get that 
$$Q_{2^k}=\sum_{j=0}^{2^k-1} Q_1\circ\theta_j
+\sum_{i=1}^k\sum_{j=0}^{2^{k-i}-1}\Psi_{2^{i-1},2^{i-1}}\circ\theta_{j2^i}\,.$$ 

Let us now take $n$ of the form $n=2^k+2^l$, with $k< l$. 
We use the identity $Q_n=Q_{2^l}+Q_{2^k}\circ\theta_{2^l}+\Psi_{2^l,2^k}$ to get that 
$$Q_n=\Psi_{2^l,2^k}+\sum_{j=0}^{n-1} Q_1\circ\theta_j
+\sum_{i=1}^l\sum_{j=0}^{\lceil n2^{-i}\rceil-1}\Psi_{2^{i-1},2^{i-1}}\circ\theta_{j2^i}\,.$$ 

More generally, let $n\ge 1$ and write a dyadic decomposition of $n$ in the form 
$$n=\eps_0+\eps_1 2+\eps_2 2^2+...+\eps_m 2^m$$ 
where the $\eps_j$ are either $0$ or $1$ and $m$ is such that $\eps_m=1$. 
(In other words, $m$ is the integer part of $\log_2 n$.) 
Then 
\beqn\label{eq:decomposition} Q_n=\sum_{j=0}^m\gamma_j+\sum_{j=0}^{n-1} Q_1\circ\theta_j
+\sum_{i=1}^m\sum_{j=0}^{\lceil n2^{-i}\rceil-1}\Psi_{2^{i-1},2^{i-1}}\circ\theta_{j2^i}\,.\eeqn  
In this last expression, all the terms $\gamma_j$ are either $0$ or of the form $\Psi_{a,b}$ for some values of $a$ and $b$ that 
can be computed in terms of $n$ but whose value does not play any role in the sequel. 

We note that in the expression (\ref{eq:decomposition}), the terms $Q_1\circ\theta_j$ corresponding to different values of $j$ are independent and identically distributed. 
Likewise, the terms $\Psi_{2^{i-1},2^{i-1}}\circ\theta_{j2^i}$ corresponding to different values of $j$ are also independent and identically distributed. 

\medskip 

Choose $M\ge 1$ such that $M\le m$.

We group the different terms in the sum $\sum_{j=0}^{n-1} Q_1\circ\theta_j$ in packs of length $2^M$ to get that 
\beqn\label{eq:comp1}\sum_{j=0}^{n-1} Q_1\circ\theta_j=\sum_{j=0}^{\lceil n2^{-M}\rceil-1} \big(\sum_{t=0}^{2^M-1} Q_1\circ\theta_t\big)\circ\theta_{j2^M}+R^{(0)}_M\,,\eeqn 
where $R^{(0)}_M$ is a sum of at most $2^M-1$ terms of the form $Q_1\circ\theta_a$ for some index $a$. 
Likewise, for any index $i\le M$, let us write 
\beqn\label{eq:comp2}\sum_{j=0}^{n2^{-i}-1}\Psi_{2^{i-1},2^{i-1}}\circ\theta_{j2^i}
=\sum_{j=0}^{\lceil n2^{-M}\rceil-1}\big(\sum_{t=0}^{2^{M-i}-1}\Psi_{2^{i-1},2^{i-1}}\circ\theta_{t2^i}\big)\circ\theta_{j2^M}+R^{(i)}_M\,,\eeqn 
where $R^{(i)}_M$ is a sum of at most $2^M-1$ terms of the form $\Psi_{b,b}\circ\theta_a$ for some values of $a$ and $b$. 
Recall the expression 
$$Q_{2^M}=\sum_{t=0}^{2^M-1} Q_1\circ\theta_t
+\sum_{i=1}^M\sum_{t=0}^{2^{M-i}-1}\Psi_{2^{i-1},2^{i-1}}\circ\theta_{t2^i}\,.$$
Summing (\ref{eq:comp1}) and (\ref{eq:comp2}), we get  
\beqn\label{eq:decomposition2}\sum_{j=0}^{n-1} Q_1\circ\theta_j+\sum_{i=1}^M\sum_{j=0}^{\lceil n2^{-i}\rceil-1}\Psi_{2^{i-1},2^{i-1}}\circ\theta_{j2^i}
=\sum_{j=0}^{\lceil n2^{-M}\rceil-1} Q_{2^M} \circ\theta_{j2^M}+R_M\,,\eeqn 
where  $R_M$ is the sum of the $R^{(i)}_M$ for $i=0...M$. 

From (\ref{eq:decomposition}) and (\ref{eq:decomposition2}), we get the following decomposition 
\beqn\label{eq:decomposition3}Q_n=\sum_{j=0}^m\gamma_j+\sum_{j=0}^{\lceil n2^{-M}\rceil-1}Q_{2^M} \circ\theta_{j2^M}+R_M+\sum_{i=M+1}^m\sum_{j=0}^{\lceil n2^{-i}\rceil-1}\Psi_{2^{i-1},2^{i-1}}\circ\theta_{j2^i}\,.\eeqn

\medskip 

Let us now bound the variance of these terms. We use the notation 
$\tau_2=\tau_2(\mu;{\cal Q})$ and $\chi_2=\chi_2(\mu;{\cal Q})$. 

We have $m\le \log_2 n$. Since the variance of each $\gamma_j$ is bounded by $\tau_2$, we have 
\beqn\label{eq:bd1}
\var^\mu[\sum_{j=0}^m\gamma_j]\le (\log_2 n)^2\tau_2\,.\eeqn 

Since there are at most $2^M$  terms of the form $Q_1\circ\theta_a$ 
and at most $M2^M$  terms of the form $\Psi_{b,b}\circ\theta_a$  in $R_M$, we get that 
\beqn\label{eq:bd2} 
\var^\mu[R_M]\le 2^{2M+1}\big(\chi_2+M^2\tau_2\big)\,.\eeqn 

We observe that, for a fixed $i$, the random variables $\Psi_{2^{i-1},2^{i-1}}\circ\theta_{j2^i}$ are i.i.d. Therefore the variance of 
$\sum_{j=0}^{\lceil n2^{-i}\rceil-1}\Psi_{2^{i-1},2^{i-1}}\circ\theta_{j2^i}$ is bounded by $\lceil n2^{-i}\rceil\tau_2$ and 
\beqn\label{eq:bd3}
\var^\mu[\sum_{i=M+1}^m\sum_{j=0}^{\lceil n2^{-i}\rceil-1}\Psi_{2^{i-1},2^{i-1}}\circ\theta_{j2^i}]\le 
n\tau_2(\sum_{i=M+1}^m 2^{-i/2})^2\le  n\tau_2(\sum_{i=M+1}^\infty 2^{-i/2})^2   \,.\eeqn 

Using the inequality $\var[A+B+C]\le (\sqrt{\var[A]}+\sqrt{\var[B]}+\sqrt{\var[C]})^2$, 
we deduce from (\ref{eq:decomposition3}), (\ref{eq:bd1}), (\ref{eq:bd2}) and (\ref{eq:bd3}) that 
$$\frac 1 n 
\var^\mu[Q_n-\sum_{j=0}^{\lceil n2^{-M}\rceil-1}Q_{2^N} \circ\theta_{j2^M}] 
\leq \frac 1 n \big( (\log_2 n\sqrt{\tau_2}+2^{M+1}\sqrt{\chi_2+M^2\tau_2}+\sqrt{n\tau_2}(\sum_{i=M+1}^\infty 2^{-i/2})   \big)^2\,.$$ 
Therefore 
$$\limsup_n \frac 1 n 
\var^\mu[Q_n-\sum_{j=0}^{\lceil n2^{-M}\rceil-1}Q_{2^N} \circ\theta_{j2^M}] 
\leq \tau_2 (\sum_{i=M+1}^\infty 2^{-i/2})^2\,.$$

and we obtain (\ref{eq:bd4}) by letting $M$ tend to $\infty$.
\qed

We are now in a position to apply the following:

\begin{lm}\label{lm:gaussianconvergence} 
Let $(A_n)_{n\in\N}$ be a sequence of real-valued centered and square integrable random variables 
such that $\var[A_n]$ has a limit, say $\sigma^2$. 
We assume that for all $M$, there exist square integrable random variables $A^{(M)}_n$ and $B^{(M)}_n$ such that 
$A_n=A^{(M)}_n+B^{(M)}_n$ and\\ 
(i) $\lim_{M\ra\infty}\limsup_{n\ra\infty}\var[B^{(M)}_n]=0$, 
(ii) for each $M$, there exists $\sigma^2_M$ such that 
$\var[A^{(M)}_n]\ra\sigma^2_M$ and $A^{(M)}_n-\esp[A^{(M)}_n]$ converges in distribution towards the Gaussian 
law with mean $0$ and variance $\sigma^2_M$. 
Then $\lim_M\sigma^2_M=\sigma^2$ and $A_n$ converges in distribution as $n$ tends to $\infty$ towards the Gaussian law with mean $0$ and variance $\sigma^2$. 
\end{lm}

{\it Proof of Lemma \ref{lm:gaussianconvergence}.} 

We denote with ${\cal N}(0,\sigma^2)$ the Gaussian law with $0$ mean and variance $\sigma^2$. 

From the inequality 
$$\var[A+B]\le (\sqrt{\var[A]}+\sqrt{\var[B]})^2$$ 
one easily deduces that $\sigma^2_M$ converges to $\sigma^2$. 

Let $g$ be a smooth, compactly supported function on $\R$.

Since $A_n$ is centered, we must have $\esp[B^{(M)}_n]=-\esp[A^{(M)}_n]$. 
The function $g$ is Lipschitz continuous; it follows that there exists a constant $C$ such that 
\beqnn 
\vert \esp[g(A_n)]-\esp[g(A^{(M)}_n-\esp[A^{(M)}_n])]\vert &\le& C \esp[\vert A_n-A^{(M)}_n+\esp[A^{(M)}_n]\vert] \\
&=&C \esp[\vert B^{(M)}_n-\esp[B^{(M)}_n]\vert]\le C \sqrt{\var[B^{(M)}_n]}\,.\eeqnn
Applying conditions (i) and (ii), we get that 
$$\lim_{M\ra\infty}\limsup_{n\ra\infty}\vert \esp[g(A_n)]-\int g\,d {\cal N}(0,\sigma^2_M)\vert=0\,.$$
Since ${\cal N}(0,\sigma_M^2)$ weakly converges to ${\cal N}(0,\sigma^2)$, we are done.
\qed 

\par\smallskip

{\it End of the proof of Theorem \ref{theo:clt}.} 

We apply Lemma \ref{lm:gaussianconvergence} to the random variables 
$$A_n=\frac 1{\sqrt{n}}(Q_n-\esp^\mu[Q_n]) \hbox{\, and\, }  
A^{(M)}_n=\frac 1{\sqrt{n}}\sum_{j=0}^{n2^{-M}-1}Q_{2^N} \circ\theta_{j2^M}-\frac 1{\sqrt{n}}\esp^\mu[Q_n]\,.$$

The claim (\ref{eq:bd4}) gives condition (i) in the Lemma. 

Fix $M$. Then the random variables $(Q_{2^N} \circ\theta_{j2^M})_{j=0...n2^{-M}-1}$ are square integrable, independent and identically distributed. 
Therefore the convergence of the variance and the Central Limit Theorem for $A^{(M)}_n$ are nothing but the C.L.T. for sums of i.i.d. variables. 

We conclude that the distribution of $\frac 1{\sqrt{n}}(Q_n-\esp^\mu[Q_n])$ converges to a Gaussian law. Lemma \ref{lm:lln} allows to replace $ \esp^\mu[Q_n]$ by $\ell(\mu;{\cal Q})n$. 
\qed 

\begin{rmk}\label{rmk:uppervariance} 
One may obtain a linear upper bound on the variance of $Q_n$ (as in the statement of Theorem \ref{theo:uppervariance}) using the decomposition (\ref{eq:decomposition}). 

We already observed that, in the expression (\ref{eq:decomposition}), the terms $Q_1\circ\theta_j$ corresponding to different values of $j$ are independent and identically distributed. 
Likewise, the terms $\Psi_{2^{i-1},2^{i-1}}\circ\theta_{j2^i}$ corresponding to different values of $j$ are also independent and identically distributed. 
Therefore, taking the variance in (\ref{eq:decomposition}), we get that 

$$\sqrt{\var^\mu[Q_n]} 
\le m\sqrt{\tau_2(\mu;{\cal Q})}+\sqrt{n\chi_2(\mu;{\cal Q})}+\sqrt{\tau_2(\mu;{\cal Q})n}\sum_{i\ge 1} 2^{-i/2}\,.$$ 
Remember that $m\le\log_2(n)$. Thus the above inequality yields a linear upper bound on the variance $\var^\mu[Q_n]$ 
and we have an alternative proof of Theorem \ref{theo:uppervariance} that avoids using the Efron-Stein inequality. 

\end{rmk} 

\subsection{Higher moments} 
\label{sec:highmoments} 
In Theorem \ref{theo:uppervariance} we showed that a finite second moment and the second moment deviation inequality imply a linear 
upper bound on the variance of $Q_n$. The next Theorem exploits Burkholder's inequality to generalize this fact to other moments. 

\begin{theo}\label{theo:othermoments} 
For all $p>1$, there exists a constant $c(p)$ such that 
for any defective adapted cocycle ${\cal Q}$ that has a finite $p$-th moment and satisfies the $p$-th moment deviation inequality 
with respect to the probability measure $\mu$ then 
$$\esp^\mu[\vert Q_n-\esp^\mu[Q_n]\vert^p]\le c(p) (\chi_p(\mu;{\cal Q})+\tau_p(\mu;{\cal Q}))\,n^{p/2}\,.$$
\end{theo}

The proof uses Burkholder's inequality that we first recall: 
\begin{lm}\label{lm:burkholder} 
Let $p>1$. There exists a constant $c_B(p)$ such that for any $n\ge 1$ and any martingale difference sequence 
$(Y_j)_{j=1...n}$ with finite $p$-th moment then 
$$\esp[\vert \sum_{j=1}^n Y_j\vert^p]\le c_B(p) \esp[(\sum_{j=1}^n Y_j^2)^{p/2}]\,.$$ 
\end{lm}

\begin{rmk}\label{rmk:efronstein} 
The argument below in particular works with $p=2$. Then one may choose $c_B(2)=1$ and the inequality in Lemma \ref{lm:burkholder} becomes an equality. 
We thus recover Theorem \ref{theo:uppervariance} although with a different constant. 
\end{rmk}

{\it Proof of Theorem \ref{theo:othermoments}.} 

Recall that ${\cal F}_j$ is the $\sigma$-field generated by the variables $X_1,...,X_j$. 
Eventually, we shall apply Burkholder's inequality to the sequence of conditional expectations $Y_j=\esp^\mu[Q_n\vert{\cal F}_j]-\esp^\mu[Q_n\vert{\cal F}_{j-1}]$. 

Note that the sequence $(Y_j)$ is indeed a martingale difference sequence in $L_p$. And also note that 
$\sum_{j=1}^n Y_j=Q_n-\esp^\mu[Q_n]$. 

Let us play the same replacement trick as in the proof of Theorem \ref{theo:uppervariance}: 
we see that $Y_j=\esp^\mu[Q_n-Q^{(j)}_n\vert {\cal F}_j]$. (Keep in mind that $Q^{(j)}$ is obtained by replacing $X_j$ by an independent copy $X'_j$. 
Then $X'_j$ is in particular independent of ${\cal F}_j$; so that $\esp^\mu[Q^{(j)}_n\vert {\cal F}_j]=\esp^\mu[Q_n\vert {\cal F}_{j-1}]$.) 

It follows from H\"older's inequality that $\esp^\mu[\vert Y_j\vert^p]\le \esp^\mu[\vert Q_n-Q^{(j)}_n\vert^p]$. 
To estimate this last term, we take the $p$-th power in inequality (\ref{eq:differencee}) and conclude that 
\beqn\label{eq:burk1}\esp^\mu[\vert Y_j\vert^p]\le 6^{p-1}(2\chi_p(\mu;{\cal Q})+4\tau_p(\mu;{\cal Q}))\,.\eeqn 

By Burkholder's inequality, we have 
\beqnn 
 \esp^\mu[\vert Q_n-\esp^\mu[Q_n]\vert^p]&=&\esp^\mu[\vert \sum_{j=1}^n Y_j\vert^p]\\ 
 &\le& c_B(p) \esp^\mu[(\sum_{j=1}^n Y_j^2)^{p/2}]\\ 
 &\le& c_B(p) n^{p/2-1} \sum_{j=1}^n \esp^\mu[\vert Y_j\vert^p]\\
 &\le& c_B(p) n^{p/2} 6^{p-1}(2\chi_p(\mu;{\cal Q})+4\tau_p(\mu;{\cal Q}))\,.\eeqnn 
 \qed

\subsection{Lower bound on the variance of D.A.C.} 
\label{sec:lowerbound}

The C.L.T. in Theorem \ref{theo:clt} applies even in examples where the limiting variance $\sigma^2(\mu;{\cal Q})$ vanishes. 

We now give sufficient conditions for a linear \emph{lower} bound on the variance of $Q_n$. 

Recall that our framework is the ``general case''.  Let ${\cal Q}$ be a defective adapted cocycle. 
Choose two integers $k$ and $n$ with $1\leq k\leq n$. 
For a given sequence $(X_j)_{j\geq 1}$, let $(X^{((k))}_j)_{j\geq 1}$ be the modified sequence obtained by removing $X_k$ and translating the indices 
after $k$. More precisely $X^{((k))}_j=X_j$ for $j<k$ and $X^{((k))}_j=X_{j+1}$ for $j\geq k$. 

Recall that $Q_{n}$ is measurable with respect to ${\cal F}_{n}$. Therefore $Q_{n}$ is of the form 
$Q_{n}=g(X_1,...,X_{n-1})$ for some measurable function $g$. 

Let $A$ be a measurable set in $G$. We shall say that the D.A.C. ${\cal Q}$ is {\bf $A$-consistent} if 
\beqn\label{consistentdac} Q_n=g(X^{((k))}_1,...,X^{((k))}_{n-1})\eeqn 
on the set $\{X_k\in A\}$ and this relation holds for any choices of $k$ and $n$ 
and all sequences $(X_j)_{j\geq 1}$ in $\Omega$. 

{\bf Examples} 

1. Let $f$ be a measurable function from $G$ to $\R$ and let $M_n(\omega)=\sum_{j=1}^n f(X_j(\omega))$. Then the sequence ${\cal M}:=(M_n)_{n\in\N^*}$ 
is a defective adapted cocycle. Assume that $f$ vanishes on a measurable set $A$. Then $\cal M$ is $A$-consistent. 

2. Consider the ``random walk case''. Let $\cal Q$ be an end-point D.A.C.. Then $\cal Q$ is consistent for the set $A=\{id\}$.

\begin{theo}\label{theo:downvariance} 
Let ${\cal Q}$ be a defective adapted cocycle. Assume that $\cal Q$ has a finite second moment and satisfies the second-moment deviation inequality 
with respect to the probability measure $\mu$. Assume that $\cal Q$ is $A$-consistent for a measurable set $A$ such that $\mu(A)>0$.  
Assume that $\ell(\mu;{\cal Q})>0$. 
Then $\sigma^2(\mu; {\cal Q})>0$. 
\end{theo}

\begin{rmk}
In the special case of quasimorphisms, the authors of \cite{CLT_quasimorphisms} give a necessary and sufficient condition for the non-vanishing of the variance that is more precise than the one in Theorem \ref{theo:downvariance}.\end{rmk}

\begin{proof}[Proof of Theorem \ref{theo:downvariance}]
The existence of the limit of $\frac 1 n\var^\mu[Q_n]$ is guaranteed by Theorem \ref{theo:existvariance}, hence it suffices to show $\limsup\frac 1 n\var^\mu[Q_n]>0$. 

Choose a set $A$ so that $\cal Q$ is $A$-consistent and $\mu(A)>0$. 
Observe that if we had $\mu(A)=1$, then the consistency assumption would imply that $Q_n=0$ for all $n$, a contradiction with our assumption that $\ell({\cal Q};\mu)>0$. 
Therefore $\mu(A)<1$. Let $A^c$ be the complement of $A$ in $G$. 
Let $\tilde{\mu}(.)=\mu(A^c\cap .)/\mu(A^c)$ be the measure $\mu$ conditioned on $A^c$. 

Define $N_n$ to be the random variable $\#\{j\leq n:X_j\in A\}$ that counts the number of null increments of $\cal Q$ up to time $n$. 
Let $S_n=\inf\{m:m-N_m\ge n\}$ be the first time $m-N_m$ exceeds $n$. 

The idea of the proof is to exploit the fluctuations of $N_n$. 

We define a new D.A.C. $\tilde{\cal Q}$ obtained by removing 
all the increments in $A$ in the following way. First we define ${\tilde X}_n=X_{S_n}$. Then, considering as before $Q_n$ as a function of the random variables $(X_1,...,X_n)$, 
we define ${\tilde Q}_n=Q_{n}({\tilde X}_1,...,{\tilde X}_n)$. From the consistency assumption on $\cal Q$, it follows that ${\tilde Q}_n=Q_{S_n}$. 

Note that the sequence ${\tilde{\cal Q}}=({\tilde Q}_n)_{n\geq 0}$ is a D.A.C.  in the filtration generated by the sequence 
$({\tilde X}_j)_{j\geq 1}$. 

We claim that, under $\matP^\mu$, the  two sequences $({\tilde Q}_n)$ and $(N_n)$ are independent and, furthermore, the law of the sequence $({\tilde Q}_n)$ under $\matP^\mu$ is the same as 
the law of the sequence $(Q_n)$ under $\matP^{\tilde\mu}$. 

\emph{Proof.} 
First observe that $\tilde \mu$ is the law of $X_1$ conditioned on the event $(X_1\notin A)$. 

Let $M$ be an integer and $n_1,...,n_M$ be integers. Note that if the event ${\cal A}:=(N_j=n_j\,\forall j\le M)$ is not empty, then there is a unique set ${\tt N}\subset\{1,...,M\}$ such that, on $\cal A$ and for $j\le M$, $X_j\in A$ if and only if $j\in \tt N$. Conversely, once we know for which indices $j\le M$ we have $X_j\in A$, then we know the value of $N_j;j\le M$. Therefore conditioning on $\cal A$ is equivalent to conditioning on the event $(X_j\in A \hbox{\,iff\,} j\in{\tt N})$. 
Under the conditional law given $\cal A$, the random variables $(X_j:j\notin {\tt N})$ are i.i.d. with law $\tilde \mu$. 

Therefore, for any measurable set $F$, we have 
\begin{align*}
&\matP^\mu[({\tilde X}_1,...,{\tilde X}_{M-n_M})\in F\,;\, N_1=n_1,...,N_M=n_M]\\ 
=&\matP^{\tilde \mu}[(X_1,...,X_{M-n_M})\in F]\matP^\mu[N_1=n_1,...,N_M=n_M]\,
\end{align*}  

Using the relation ${\tilde Q}_n=Q_{n}({\tilde X}_1,...,{\tilde X}_n)$, 
we conclude that 
\begin{align*}
&\matP^\mu[{\tilde Q}_1=z_1,...,{\tilde Q}_{M-n_M}=z_{M-n_M}\,;\, N_1=n_1,...,N_M=n_M]\\ 
=&\matP^{\tilde \mu}[Q_1=z_1,...,Q_{M-n_M}=z_{M-n_M}]\matP^\mu[N_1=n_1,...,N_M=n_M]\,
\end{align*}  
for all choices of $M$, $n_1,..,n_M$ and $z_1,...,z_{M-n_M}$. 
Let us now choose $k$, $z_1,...,z_k$ and $n_1,...,n_k$. Choosing any $M\ge n_k+k$ in the previous equality, we get that   
\begin{align*}
&\matP^\mu[{\tilde Q}_1=z_1,...,{\tilde Q}_k=z_k\,;\, N_1=n_1,...,N_k=n_k]\\
=&\matP^{\tilde \mu}[Q_1=z_1,...,Q_k=z_k]\matP^\mu[N_1=n_1,...,N_k=n_k]\,.
\end{align*}  
It indeed shows that, under $\matP^\mu$, the sequence $({\tilde Q}_n)$ as the same law as  the sequence $(Q_n)$ under $\matP^{\tilde\mu}$ and that the two sequences $({\tilde Q}_n)$ and $(N_n)$ are independent. 
 \qed 

A consequence of the claim is that for all $k\le n$, the law of $Q_n$ given $N_n=k$ is the law of ${\tilde Q}_{n-k}$. 
To see this, note that, on the event $N_n=k$, we have $S_{n-k}=n$ and therefore ${\tilde Q}_{n-k}=Q_n$. 
Therefore 
$$\matP^\mu[Q_n=z;N_n=k]=\matP^\mu[{\tilde Q}_{n-k}=z;N_n=k]=\matP^\mu[{\tilde Q}_{n-k}=z]\matP^\mu[N_n=k]\,.$$ 
We used the independence of ${\tilde Q}_{n-k}$ and $N_n$. 

\medskip 

Let $\ell=\ell(\mu;{\cal Q})$. We will now show that there exists $\epsilon>0$ so that for each sufficiently large $n$ we have
$$\esp^{\mu}[(Q_n-\ell n)^2]+ \esp^{\mu}[(Q_{n+r(n)}-\ell n-\ell r(n))^2]>\epsilon n,$$
where $r(n)=\lceil \sqrt n \rceil$. 

The inequality above suffices to show that $\limsup\frac 1 n\var^\mu[Q_n]>0$ in view of Lemma \ref{lm:lln} which guarantees that $\ell n$ is a good approximation of $\esp^\mu[Q_n]$.

Observe that for each $m$ we have 
$$\esp^{\mu}[(Q_m-\ell m)^2]=$$
$$\sum_{k\in \mathbb N} \esp^\mu [(Q_m-\ell m)^2|N_m=k]\matP[N_{m}=k]=\sum_k \esp^\mu[({\tilde Q}_{m-k}-\ell m)^2]\matP[N_{m}=k].$$
We can now use the fact that $N_m$ has the same law as a sum of independent Bernoulli random variables with parameter $p=\mu(A)$. Then there exists $\epsilon_0>0$ so that for each large enough $n$ and integer $x$ satisfying $|x-pn|\leq 3\sqrt{n}$ we have $\matP^\mu[N_n=x]\geq \epsilon_0/\sqrt n$.

Hence, for each sufficiently large $n$ we have
\beqnn 
&& \esp^{\mu}[(Q_n-\ell n)^2]+ \esp^{\mu}[(Q_{n+r(n)}-\ell n-\ell r(n))^2]\\
 &\ge& \frac{\epsilon_0}{\sqrt n}\left( \sum_{|k-pn|\leq r(n)} \esp^{\mu}[({\tilde Q}_{n-k}-\ell n)^2]+\sum_{|j-pr(n)-pn|\leq 3r(n)} \esp^{\mu}[({\tilde Q}_{n+r(n)-j}-\ell n-\ell r(n))^2]\right)\\
 &\ge& \frac{\epsilon_0}{\sqrt n}\left( \sum_{|k-pn|\leq r(n)} \esp^{\mu}[({\tilde Q}_{n-k}-\ell n)^2]+\sum_{|k-pn|\leq r(n)} \esp^{\mu}[({\tilde Q}_{n-k}-\ell n-\ell r(n))^2]\right).
\eeqnn 

For any given $k$ and any $R\in\R$, we have $(R-\ell n)^2+(R-\ell n-\ell r(n))^2\geq \ell^2 r(n)^2/2$, so that we get 
$$\esp^{\mu}[(Q_n-\ell n)^2]+ \esp^{\mu}[(Q_{n+r(n)}-\ell n-\ell r(n))^2]\geq 
\frac{\epsilon_0}{\sqrt n} 2r(n)\frac{\ell^2 r(n)^2}{2}\geq \epsilon_0\ell^2 r(n)^2,$$
as required.
\end{proof}

\subsection{Applications to random walks} 
\label{sec:firstrws}

Let us now apply the results in the ``general case'' to the ``random walk case'': $G$ is an infinite, countable, discrete group, and we choose the length D.A.C. $\cal Q$ given by 
$Q_n=d(id, Z_n)$ where $d$ is a left-invariant metric on $G$ and $(Z_n)$ is the random walk with driving measure $\mu$. 
Recall the notation 
$$(x,y)_w=\frac 12 (d(w,x)+d(w,y)-d(x,y))$$ for the Gromov product in the metric $d$. 

Recall that the D.A.C. $\cal Q$ satisfies the $p$th-moment deviation inequality if there exists a constant $\tau_p(\mu;d)$ such that 
for all 
$n$ and $m$ in $\N$  then 
$$ 
 \esp^{\mu}[(id,Z_{n+m})_{Z_n}^p]\leq 2^{-p} \tau_p(\mu;d)\,.
 $$ 

And we say that $\mu$ has  finite $p$-th moment with respect to $d$ if 
$\chi_p(\mu;d):=\sum_{x\in G} d(id,x)^p\mu(x)<\infty$. 

Let $\ell(\mu;d)$ be the rate of escape in the metric $d$. 

The applications of Theorems \ref{theo:existvariance}, \ref{theo:clt} and \ref{theo:uppervariance}  are straightforward. 
As for Theorem \ref{theo:downvariance}, let us observe that the D.A.C. is $A$-consistent for the set $A=\{id\}$. Therefore the condition that $\mu(A)>0$ 
will be satisfied whenever $\mu(id)\not=0$. More generally, in the case that there is a convolution power $K$ such that $\mu^K(id)\neq 0$, we observe that $\lim\sup\left(\var^\mu[Q_n]/n\right)\geq \lim\sup\left(\var^{\mu^{\scriptscriptstyle K}}[Q_n]/(nK)\right)$. We thus conclude that: 

\begin{theo}\label{theo:cltrws}
Let $G,\mu,d,(Z_n)$ be as above.
Suppose that $\mu$ has a finite second moment with respect to $d$ and the second-moment deviation inequality holds true. 
Then the variance 
$\frac 1 n \var^\mu(d(id,Z_n))$ has a limit as $n$ tends to $\infty$. We denote it by $\sigma^2(\mu;d)$. It satisfies the bound 
$$\sigma^2(\mu;d)\leq 4\chi_2(\mu;d)+16 \tau_2(\mu;d).$$
Furthermore,
the law of $\frac 1{\sqrt{n}} (d(id,Z_n)-\ell(\mu;d)n)$ under $\P^\mu$ weakly converges to the Gaussian law with zero mean and variance 
$\sigma^2(\mu;d)$. 

Assume furthermore that $\ell(\mu;d)>0$ and that there exists an integer $K$ such that $\mu^K(id)\not=0$. Then $\sigma^2(\mu;d)>0$. 
\end{theo}

Finally observe that Theorem \ref{theo:othermoments} also applies and provides some useful bounds on other moments of the distance $d(id,Z_n)$.

\section{Fluctuations of the rate of escape} \label{sec:fluctuatrescp} 

Let $d$ be a left-invariant metric on the infinite, countable, discrete group $G$. 
In this section of the paper, we discuss regularity properties of the rate of escape $\ell(\mu;d)$ (as defined in (\ref{eq:1}) and (\ref{eq:king})), considered as a function of the driving measure $\mu$. 
In Theorem \ref{theo:rescp}, we give sufficient conditions that imply the Lipschitz continuity of $\ell$; 
Theorems \ref{theo:diffrate} and \ref{theo:c1rate} are about the differentiability of $\ell$. Both use deviation inequalities as assumptions. 

The question of the regularity of the rate of escape and the entropy as a function of $\mu$ was raised by A. Erschler and V. Kaimanovich in \cite{kn:EK}. We refer to \cite{kn:GL} for a review on the subject. Although the question is simple enough to state, very little is known for general (non-hyperbolic) groups. Only in a handful of examples, can we explicitly compute $\ell(\mu;d)$. In \cite{kn:EK}, it is proved that, for non-elementary hyperbolic groups and under a first moment assumption, the asymptotic entropy  is continuous for the weak topology on measures - a fact that fails to be true in all groups, see \cite{kn:Er}. If we restrict ourselves to measures $\mu$ with fixed finite support (and still assume that $G$ is non-elementary hyperbolic), F. Ledrappier proved in \cite{kn:led2} that $h$ and $\ell$ are Lipschitz continuous. This result was upgraded to analyticity by S. Gouezel in  \cite{Gouezel:analyticity}. 


Different techniques were used to prove these results. In \cite{kn:EK}, the authors use a version of Kaimanovich's ray criteria. 
The results of \cite{kn:led1}, \cite{kn:led2} and \cite{Gouezel:analyticity} are based on properties of the dynamics induced by a random walk on the boundary of $G$. 
\cite{kn:HMM} proved the analyticity of the rate of escape for random walks in Fuchsian groups using their automatic structure and regeneration times.  

In \cite{kn:hype}, we introduced a martingale approach to  clarify the connection between the differentiability of $\ell$ and $h$ and the Central Limit Theorem. The proofs of both Theorem \ref{theo:rescp} and \ref{theo:diffrate} follow a similar martingale approach.

As for the Central Limit Theorem, our theorems are expressed for a defective adapted cocycle on a general product space before being applied to random walks on groups. 
We use the notation and definitions from Section \ref{sec:rws}.

\subsection{Distances between measures} \label{ssec:distance} 

Let $\p({\cal G})$ be the set of probability measures on the space $(G,{\cal G})$. 

We define a distance on $\p({\cal G})$ in the following way: let $\mu$ and $\tmu$ be in $\p({\cal G})$. 
First assume that $\mu$ and $\tmu$ are absolutely continuous with respect to each other. Let 
$f=d\tmu/d\mu$ and $\tilde {f}=d\mu/d\tmu$ be the Radon-Nikodym derivatives and define 
$$\Nu(\tmu,\mu)=\max(\hbox{$\sup^\mu$}(\vert f-1\vert ) ; \hbox{$\sup ^{\tmu}$}(\vert \tilde{f}-1\vert ),$$ 
where $\sup^\mu$ denotes the essential $\sup$ with respect to $\mu$ and similarly for $\sup^{\tmu}$. 
Observe that $$\hbox{$\sup^{\tmu}$}(\vert \tilde{f}-1\vert )=\hbox{$\sup^\mu$}(\vert \frac 1 f-1\vert).$$ Therefore 
$$\Nu(\tmu,\mu)=\max(\hbox{$\sup^\mu$}(\vert f-1\vert ) ; \hbox{$\sup^\mu$}(\vert \frac 1 f-1\vert )).$$ 
If $\mu$ and $\tmu$ are not  absolutely continuous with respect to each other, we then let $\Nu(\mu,\tmu)=\infty$. 
Then $\Nu$ is a distance on $\p({\cal G})$  although possibly taking the value $+\infty$.  

In the sequel, we shall call ''neighborhood'' of a probability measure $\mu_0$ a subset of $\p({\cal G})$ of the form 
${\cal N}=\{\mu\in\p({\cal G})\,;\, \Nu(\mu,\mu_0)\leq K\}$ for some $K>0$. 

In the case where $G$ is a countable and discrete set, the expression of $\Nu(\tmu,\mu)$ becomes 
$$\Nu(\tmu,\mu)=\sup_{a\in G}\max( \frac{\tmu(a)}{\mu(a)}  ;  \frac{\mu(a)}{\tmu(a)})-1.$$ 
Observe that $\Nu(\tmu,\mu)$ cannot be finite unless $\tmu$ and $\mu$ have the same support. 
For a subset $B\subset G$, we use the notation $\p(B)$ for the set of probability measures with support equal to $B$. 

If $B$ is finite, then we may identify $\p(B)$ with a subset of $\mathbb{R}^d$ with $d=\# B$. Then $\Nu$ is locally 
equivalent to the Euclidean distance. 

%
%
%
%
%
%

\subsection{Uniform deviation inequalities}\label{subsec:dev} 

Let ${\cal Q}=(Q_n)_{n\in\N^*}$ be a defective adapted cocycle with defect $\Psi=(\Psi_{n,m})_{(n,m)\in\N\times\N}$ as in Definition \ref{df:dac}. 
Recall the notions of finite moments and deviation inequalities from Section \ref{sec:rws}. 

We shall also use uniform versions of the deviation inequalities. Namely: 
let $\cal B$ be a subset of $\p({\cal G}))$. 
In the sequel we use notation 
$\tau_p({\cal B};{\cal Q}):=\sup_{\mu\in{\cal B}} \tau_p(\mu;{\cal Q})$ and $\chi_p({\cal B};{\cal Q}):=\sup_{\mu\in{\cal B}}\chi_p(\mu;{\cal Q})$. 
If $\tau_p({\cal B};{\cal Q})<\infty$ (resp. $\chi_p({\cal B};{\cal Q})<\infty$), we say that 
the {\bf $p$-th moment deviation inequality} (resp.  the {\bf finite $p$-th moment assumption}) holds {\bf uniformly} on $\cal B$. 

Given $\mu_0$ in $\p({\cal G})$, we also say that the  $p$-th moment deviation inequality (resp.  the  finite $p$-th moment assumption) holds {\bf locally uniformly} 
around $\mu_0$ if there exists a neighborhood of $\mu_0$, say $\cal N$, such that the $p$-th moment deviation inequality (resp.  the  finite $p$-th moment assumption) holds 
uniformly on $\cal N$. 

We shall need the following simple observation: 

\begin{lm} \label{lm:simple} 
Let $\mu\in\p({\cal G})$ and $\tmu\in\p({\cal G})$  and assume $\mu$ has a finite first moment. Then 
$$\chi_1(\tmu;{\cal Q})\le (1+\Nu(\tmu,\mu))\chi_1(\mu;{\cal Q})\,.$$
\end{lm} 
\begin{proof}
By definition of $\Nu$, we have $f(a)\le 1+\Nu(\tmu,\mu)$ for all $a$. The inequality in the Lemma follows.
\end{proof}

\subsection{Lipschitz continuity and differentiability of the rate of escape} \label{ssec:backrescp}

Recall that we deal with the ``general case''. Let $\cal B$ be a subset of $\p({\cal G})$. In the sequel, we will say that a function $F$ is Lipschitz continuous on $\cal B$ if it satisfies 
$\vert F(\mu_1)-F(\mu_0)\vert\le C \Nu(\mu_0,\mu_1)$ for some constant $C$ and all $\mu_0$ and $\mu_1$ in $\cal B$. 

The main results in this section, stated below, are about Lipschitz continuity and differentiability of the rate of escape.

\begin{theo}
\label{theo:rescp} Let $\cal B$ be a convex subset of $\p({\cal G})$. Assume $\cal Q$ satisfies the finite first-moment assumption 
and first-moment deviation inequalities uniformly  in $\cal B$. 
Then the function $\mu\ra \ell(\mu;{\cal Q})$ is Lipschitz continuous on $\cal B$. \end{theo}

\begin{theo} \label{theo:diffrate} Let $\mu$ be a probability measure in $\p({\cal G})$ and 
let $\cal Q$ be a defective adapted cocycle with a finite second moment and satisfying the second moment deviation inequality with respect to $\mu$. 

Then the function $\tmu\ra \ell(\tmu;{\cal Q})$ is differentiable at $\tmu=\mu$ in the following sense: 
Let $(\mu_t, t\in[0,1])$ be a curve in $\p({\cal G})$ such that $\mu_0=\mu$ and such that $\mu_t$ is absolutely continuous with respect to $\mu$ for all $t\in[0,1]$. 
Further assume that \\
(i) we can choose the Radon-Nikodym derivatives $f_t=d\mu_t/d\mu$ such that, for all $a\in G$, the function 
$t\ra \log f_t(a)$ has a derivative at $t=0$, say $\nu(a)$, the function  $\nu$ is bounded on $G$ and 
$\sup_{t\in[0,1]}\sup_{a\in G} \vert \frac 1 t \log f_t(a)-\nu(a)\vert<\infty$. \\
(ii) $\cal Q$ satisfies the uniform first moment deviation inequality on the family $(\mu_t, t\in [0,1])$.\\ 
Then the limit of $\frac 1 t (\ell(\mu_t;{\cal Q})-\ell(\mu;{\cal Q}))$ as $t$ tends to $0$ exists. 
Furthermore, this limit coincides with the covariance
\beqn\label{eq:sigmanu}\sigma(\nu,\mu;{\cal Q}):=\lim_n\frac 1 n \esp^\mu[Q_n\sum_{j=1}^n \nu(X_j)]\,.\eeqn
 \end{theo}

Observe that $\sigma(\nu,\mu;{\cal Q})$  is linear w.r.t. $\nu$.

\begin{theo} \label{theo:c1rate} 
The function $\mu\ra \ell(\mu;{\cal Q})$ is $C^1$ in the following sense: 
Let  $(\mu_t, t\in[-1,1])$ be a curve in $\p({\cal G})$ such that $\mu_t$ is absolutely continuous with respect to $\mu_0$ for all $t\in[-1,1]$. 
Further assume that:\\ 
(i) we can choose the Radon-Nikodym derivatives $f_t=d\mu_t/d\mu_0$ such that, for all $a\in G$, the function 
$t\ra \log f_t(a)$ is $C^1$.  Let $\nu_t(a)$ be its derivative and assume that 
$\sup_{t\in[-1,1]}\sup_{a\in G}\vert\nu_t(a)\vert<\infty$.  \\ 
(ii) there exists $p>1$ such that $\cal Q$ satisfies the uniform $p$-th moment deviation inequality on the family $(\mu_t, t\in [-1,1])$.\\ 
(iii) $\mu_0$ has a finite first moment.\\ 
Then $\cal Q$ satisfies  the uniform first moment inequality on the family $(\mu_t, t\in [-1,1])$ 
and  the function 
$t\ra \ell(\mu_t;d)$ is continuously differentiable.
\end{theo}

The main ingredient in the proofs of these three Theorems is the following {\bf Girsanov formula}. 
Let $\mu$ and $\tmu$ be two probability measures in $\p({\cal G})$. Assume that $\tmu$ is absolutely 
continuous with respect to $\mu$ with Radon-Nikodym derivative $f$. Then 
the restriction of $\P^{\tmu}$ to the $\sigma$-field ${\cal F}_n$  
is absolutely continuous with respect to 
the restriction of $\P^\mu$  with Radon-Nikodym derivative equal to 
$\Pi_{j=1}^n f(X_j)$. In other words,  for any non-negative measurable function $F:G^n\ra\R_+$: 
\beqn\label{eq:girgir} 
\esp^{\tmu}[F(X_1,...,X_n)]=\esp^\mu[F(X_1,...,X_n) \Pi_{j=1}^n f(X_j)]\,.
\eeqn 

Formula (\ref{eq:girgir}) also holds if the random variable $F(X_1,...,X_n)$ is integrable with respect to $\P^{\tmu}$. 

Theorem \ref{theo:rescp} directly follows from the Girsanov formula, the replacement trick from Section \ref{ssec:existvariance} and the deviation inequality. 

The strategy to obtain Theorem \ref{theo:diffrate} is the same as in \cite{kn:hype}. It very much relies on the expression of the derivative of the variance at a fixed time 
as a correlation as in (\ref{eq:gir1}), and the Central Limit Theorem \ref{theo:clt}.  

The proof of Theorem \ref{theo:c1rate} exploits the Girsanov formula in a more direct way in combination with the representation formula (\ref{eq:magic}) for the rate of escape. 
It does not use the C.L.T.


\subsection{Proof of Theorem \ref{theo:rescp}}
\label{sec:proofrescp}

We shall in fact obtain a stronger result than stated in Theorem \ref{theo:rescp} with an explicit control on the Lipschitz constant, see Proposition \ref{prop:lip} below.  

Let $\mu_0$ and $\mu_1$ be probability measures in $\p({\cal G})$. 

For $t\in[0,1]$, we define the measure $\mu_t=\mu_0+t(\mu_1-\mu_0)$. Then $\mu_t$ belongs to $\p({\cal G})$. 

Assume that $\mu_1$ is absolutely continuous with respect to $\mu_0$ and let $f_1=d\mu_1/d\mu_0$ be a Radon-Nikodym derivative. 
Then $\mu_t$ is also absolutely continuous with respect to $\mu_0$ and 
$$f_t:=\frac{d\mu_t}{d\mu_0}=1+t(f_1-1)$$ 
is a Radon-Nikodym derivative of $\mu_t$ with respect to $\mu_0$. 

Define $\nu_t:=(f_1-1)/f_t$. Note that for any $a\in G$, then $\sup_{t\in[0,1]}\vert \nu_t\vert=\max(\vert f_1(a)-1\vert; \vert 1/f_1(a)-1\vert)$. 
It then follows that 
$$\Nu(\mu_1,\mu_0)=\hbox{$\sup^{\mu_0}$}\sup_{t\in[0,1]}\vert \nu_t\vert.$$ 

The next Proposition is more precise than Theorem \ref{theo:rescp}. 

\begin{prop}\label{prop:lip} 
Let $\mu_0$ and $\mu_1$ be probability measures in $\p({\cal G})$ and $\cal Q$ be a D.A.C. with a finite first moment with respect to $\mu_0$ and $\mu_1$. 
Then, for all $n\geq 1$, we have  
\beqn \label{eq:lipn} \frac 1n\esp^{\mu_1}[Q_n]-\frac 1n\esp^{\mu_0}[Q_n]\leq \,  \Nu(\mu_1,\mu_0)
\big( 2\chi_1(\mu_0;{\cal Q})+2\chi_1(\mu_1;{\cal Q})+4\sup_{t\in[0,1]}\tau_1(\mu_t;{\cal Q})\big)\,.\eeqn 
\end{prop}

\begin{proof}
We now assume that $\mu_0$ and $\mu_1$ are such that $\chi_1(\mu_0;{\cal Q})$, $\chi_1(\mu_1;{\cal Q})$, $\Nu(\mu_1,\mu_0)$,  and 
$\sup_{t\in[0,1]}\tau_1(\mu_t;{\cal Q})$ are all finite. (Otherwise there is nothing to be proved.) 

We use the shorthand notation $\esp^t$ instead of $\esp^{\mu_t}$. 

Applying (\ref{eq:girgir}) to $\mu_t$, we get   that for any integrable non-negative measurable function $F:G^n\ra\R_+$: 
\beqn\label{eq:gir} 
\esp^t[F(X_1,...,X_n)]=\esp^0[F(X_1,...,X_n) \Pi_{j=1}^n f_t(X_j)]\,.
\eeqn 

Since $Q_n$ is integrable, we deduce that 
\beqn\label{eq:gir0} 
\esp^t[Q_n]=\esp^0[Q_n \Pi_{j=1}^n f_t(X_j)]\,.
\eeqn 

Let us take the derivative in $t$ in equation (\ref{eq:gir0}). This is justified since the expectation w.r.t. $\esp^0$ in (\ref{eq:gir0}) is in fact a polynomial in $t$ and also 
because $\nu_t$ is bounded. 
We get that 
\beqnn 
\frac d{dt} \esp^t[Q_n]&=&\sum_{k=1}^n \esp^0[Q_n\frac {f_1(X_k)-1}{f_t(X_k)}\Pi_{j=1}^n f_t(X_j)]\\
&=&\sum_{k=1}^n \esp^0[Q_n \nu_t(X_k)\Pi_{j=1}^n f_t(X_j)]\,.
\eeqnn
Using the Girsanov formula again (but in the other direction!), we deduce that 
\beqn\label{eq:gir1} 
\frac d{dt} \esp^t[Q_n]=\sum_{k=1}^n \esp^t[Q_n \nu_t(X_k)]\,.
\eeqn

We now use the same replacement trick as in Section \ref{ssec:existvariance} and, as in Section \ref{ssec:existvariance}, 
we let $Q^{(k)}_n$ be the D.A.C. obtained when replacing $X_k$ by an independent copy in the definition of $Q_n$. 
Then $Q^{(k)}_n$ is independent of $X_k$ and therefore 
$$\esp^t[Q^{(k)}_n \nu_t(X_k)]=\esp^t[Q^{(k)}_n]\,\esp^t[ \nu_t(X_k)]=0. $$ 
Therefore  
$$\frac d{dt} \esp^t[Q_n]=\sum_{k=1}^n \esp^t[(Q_n-Q^{(k)}_n)\nu_t(X_k)]$$ 
and 
\beqn \label{eq:ef1}  \vert \frac d{dt} \esp^t[Q_n]\vert \leq \Nu(\mu_1,\mu_0) \sum_{k=1}^n \esp^t[\vert Q_n-Q^{(k)}_n\vert].\eeqn

We next want to bound $\esp^t[\vert Q_n-Q^{(k)}_n\vert]$. Taking the expectation in Formula (\ref{eq:difference}), we deduce that  
$$\esp^t[\vert Q_n-Q^{(k)}_n\vert]
\leq 2\chi_1(\mu_t;{\cal Q})+4\tau_1(\mu_t;{\cal Q}).$$ 

Inserting this bound into (\ref{eq:ef1}) and taking the $\sup$ over $t$, we get that 

$$ \frac 1 n\vert \frac d{dt} \esp^t[Q_n]\vert \leq \Nu(\mu_1,\mu_0) \big(2\sup_{t\in[0,1]}\chi_1(\mu_t;{\cal Q})+4\sup_{t\in[0,1]}\tau_1(\mu_t;{\cal Q})\big).$$ 

Observe that $\sup_{t\in[0,1]}\chi_1(\mu_t;{\cal Q})\leq \chi_1(\mu_0;{\cal Q})+\chi_1(\mu_1;{\cal Q})$. 
The Proposition follows at once.
\end{proof}

\begin{rmk} 
We just used the obvious bound $\sup_{t\in[0,1]}\chi_1(\mu_t;{\cal Q})\leq \chi_1(\mu_0;{\cal Q})+\chi_1(\mu_1;{\cal Q})$. One might ask whether a similar bound holds 
for $\tau_1$ instead of $\chi_1$: is there a way to bound $\tau_1(\mu_t;{\cal Q})$ in terms of $\tau_1(\mu_0;{\cal Q})$ and $\tau_1(\mu_1;{\cal Q})$? We do not even know 
whether $\tau_1(\mu_t;{\cal Q})$ is finite whenever $\tau_1(\mu_0;{\cal Q})$ and $\tau_1(\mu_1;{\cal Q})$ are. 

Consequently it may be sometimes difficult to figure out which curves satisfy the assumptions of Theorems \ref{theo:diffrate} and \ref{theo:c1rate}. Fortunately, in Part II, we shall deal 
with groups (hyperbolic and acylindrically hyperbolic groups) for which rather explicit bounds on $\tau_1$ can be proved for large classes of measures. 

\end{rmk}

\subsection{Proof of Theorem \ref{theo:diffrate}.} 

Let us first observe that, since we assume $\sup_{t\in[0,1]}\sup_{a\in G} \vert \frac 1 t \log f_t(a)-\nu(a)\vert<\infty$ and $\nu$ is bounded, 
then it follows that $\sup_{t\in[0,1]}\sup_{a\in G}  f_t(a) <\infty$ and $\sup_{t\in[0,1]}\sup_{a\in G} 1/ f_t(a) <\infty$. \\ 
Therefore $\sup_{t\in [0,1]}\Nu(\mu_t,\mu_0) <\infty$. Then Lemma \ref{lm:simple} implies that $\cal Q$ satisfies the uniform first-moment inequality 
on the family $(\mu_t, t\in [0,1])$. Combined with the uniform first-moment deviation inequality, it implies the existence of the rate of escape $\ell(\mu_t;{\cal Q})$ 
for all $t\in[0,1]$.

In the first part of the proof, we argue that Theorem \ref{theo:existvariance} implies that the limit defining $\sigma(\nu,\mu;d)$ in (\ref{eq:sigmanu}) exists.  

The rest of the proof of Theorem \ref{theo:diffrate} follows the same strategy as in \cite{kn:hype}. 
As in  (\ref{eq:gir1}), one may use the Girsanov formula to write the derivative of $\esp^t[d(id,Z_n)]$ as a covariance, divide by $n$ and let $n$ tend to $\infty$. This leads to a correct guess 
for the identification of the derivative but it remains to explain how to exchange the two limits as $n$ tend to $\infty$ and $\lambda$ tends to $0$. The justification  we find in \cite{kn:hype} uses Gaussian integration by parts and the full strength of the C.L.T. (not just the existence of the variance).

For $n\ge 1$, let us define $$M_n:=\sum_{j=1}^n \nu(X_j)\,;\,M_0=0.$$ 
Recall that $\nu(a)$ is the derivative at $t=0$ of the function $f_t(a)$. Since all measures $\mu_t$ are probability measures, we must have 
$\int \nu(a)d\mu(a)=0$. Therefore the sequence of random variables $(M_n)_{n\in\N}$ is a sum of i.i.d. bounded and centered random variables under $\esp^\mu$.  In particular it satisfies the C.L.T. 

Let us now check that the limit defining $\sigma(\nu,\mu;d)$ in (\ref{eq:sigmanu}) exists. 
First observe that, since $M_n$ is centered, then $\esp^\mu[Q_nM_n]=\esp^\mu[(Q_n-\esp^\mu[Q_n])M_n]$ is the covariance of $Q_n$ and $M_n$. 

Let $a\in\R$. We define a D.A.C. ${\cal Q}^a:=(Q^a_n:=Q_n+aM_n)_{n\in\N^*}$. Note that since $\nu$ is bounded and since 
we are assuming that $\cal Q$ has a finite second-moment, then ${\cal Q}^a$ has a finite second moment. 
The defect of ${\cal Q}^a$ coincides with the defect of $\cal Q$. Since we are assuming the second-moment deviation inequality for $\cal Q$, 
then ${\cal Q}^a$ also satisfies the second-moment deviation inequality in the sense of Definition \ref{df:dac}.

A first application of Theorem \ref{theo:existvariance} to the D.A.C. ${\cal Q}$ yields the existence of the limit 
$\lim_n\frac 1 n \var^\mu[Q_n]$. It is clear that the limit $\lim_n\frac 1 n \var^\mu[M_n]$ also exists. 
Applying now  Theorem \ref{theo:existvariance} to ${\cal Q}^{1/2}$, we deduce the existence of the limit 
$ \lim_n\frac 1 n \esp^\mu[Q_nM_n] 
=\lim_n \frac 1 n\big(\var^\mu[Q^{1/2}_n]-\frac 1 4 \var^\mu[M_n]-\var^\mu[Q_n]\big)$.

\medskip

The rest of the proof of Theorem \ref{theo:diffrate} follows the same strategy as in \cite{kn:hype}. 

The Central Limit Theorem \ref{theo:clt} under $\P^\mu$ applied to the family of D.A.C. ${\cal Q}^a$ implies the joint Central Limit Theorem under $\P^\mu$ for the vector 
$\frac 1{\sqrt{n}} (Q_n-n\ell(\mu;{\cal Q}),M_n)$ as in Proposition 3.2 (i) of \cite{kn:hype}; see Lemma \ref{lm:Levy} below. 
 
Note that the variance upper bound needed to apply Theorem 2.3 (assumption (ii)) in \cite{kn:hype} follows from Theorem \ref{theo:uppervariance} here.  

Lemma \ref{lm:lln} applied to the D.A.C. ${\cal Q}$ with the assumption of uniform first-moment deviation inequality imply Lemma 3.1 in \cite{kn:hype}. 

Also note that Theorem 2.3 in \cite{kn:hype} was written for the length D.A.C. and a measure $\mu$ with finite support. The details to adapt it to our setup are straightforward. 

\begin{lm}\label{lm:Levy}
 Let $(A_n),(B_n)$ be sequences of real-valued random variables. Then the random vectors $(A_n,B_n)$ converge in distribution to the random vector $(A,B)$ if and only if $(B_n)$ converges in distribution to $B$ and for each $a\in\mathbb R$ the random variables $A_n+aB_n$ converge in distribution to $A+aB$.
\end{lm}

\begin{proof}
 By L\'evy's Theorem, $(A_n,B_n)$ converge in distribution to $(A,B)$ if and only if for every $a,b\in\mathbb R$ we have $\esp^{\mu}[e^{i(bA_n+aB_n)}]\to \esp^{\mu}[e^{i(bA+aB)}]$. By (the other direction of) L\'evy's Theorem this happens if and only if, for every $a,b\in\mathbb R$, $bA_n+aB_n$ converges in distribution to $bA+aB$, and the conclusion easily follows.
\end{proof}

\subsection{Proof of Theorem \ref{theo:c1rate}.}

We start by establishing a convenient representation formula for the rate of escape: 

\begin{prop}\label{prop:magic} 
Let $\mu$ be a probability measure and $\cal Q$ a D.A.C. with a finite first moment and satisfying the first moment deviation inequality.  
Define the random variable 
$$Z:=\sum_{k\geq 0} 2^{-k-1} \Psi_{2^k,2^k}\,.$$ 
Then 
\begin{equation}\label{eq:magic} 
\ell(\mu;{\cal Q})=\esp^\mu[Q_1]+\esp^\mu[Z]\,.
\end{equation}
\end{prop} 

Observe that the first moment deviation inequality implies that 
$\esp^\mu[\sum_{k\geq 0} 2^{-k-1} \vert\Psi_{2^k,2^k}\vert]\leq \tau_1(\mu;{\cal Q})<\infty$. 
Therefore the series defining $Z$ converges almost surely and in $L_1$ under $\P^\mu$. 

\begin{proof} 

By definition of the defect $\Psi$, we have 
$$2^{-k-1} \Psi_{2^k,2^k}=2^{-k-1} Q_{2^{k+1}}-2^{-k-1} Q_{2^k}-2^{-k-1} Q_{2^k}\circ \theta_{2^k}.$$
Recall that $Q_{2^k}\circ \theta_{2^k}$ has the same law as $Q_{2^k}$. 
Taking the expectation, we then get that 
$$2^{-k-1} \esp^\mu[\Psi_{2^k,2^k}]=2^{-k-1} \esp^\mu[Q_{2^{k+1}}]-2^{-k} \esp^\mu[Q_{2^k}].$$

Therefore, for all $N$, we obtain that 
$$\sum_{k=0}^N 2^{-k-1}\esp^\mu[ \Psi_{2^k,2^k} ]
=2^{-(N+1)}\esp^\mu[Q_{2^{N+1}}]-\esp^\mu[Q_1]\,.
$$

Applying the first-moment deviation inequality, we may pass to the limit as $N$ tends to $\infty$. We deduce that 
$\esp^\mu[Z]=\ell(\mu; {\cal Q})-\esp^\mu[Q_1]$.

 \end{proof}  

{\it Proof of Theorem \ref{theo:c1rate}} 

As in the  proof of Theorem \ref{theo:diffrate}, let us first observe that,  
since we assume $\sup_{t\in[0,1]}\sup_{a\in G} \vert \nu_t(a)\vert<\infty$, 
then it follows that $\sup_{t\in[0,1]}\sup_{a\in G}  f_t(a) <\infty$ and $\sup_{t\in[0,1]}\sup_{a\in G} 1/ f_t(a) <\infty$. 
Therefore $\sup_{t\in [0,1]}\Nu(\mu_t,\mu_0)<\infty$. Lemma \ref{lm:simple} implies that $\cal Q$ satisfies the uniform first-moment inequality 
on the family $(\mu_t, t\in [0,1])$. Combined with the uniform first-moment deviation inequality, it implies the existence of the rate of escape $\ell(\mu_t;{\cal Q})$ 
for all $t\in[0,1]$. 

We use the shorthand notation $\esp^t$ instead of $\esp^{\mu_t}$. 

Let  
$$C:= \sup_{t\in[-1,1]}\sup_{a\in B}\vert\nu_t(a)\vert\,,$$ 
and recall that we are assuming that $C<\infty$. 

The proof of the Theorem consists in taking the derivatives of the summands in formula (\ref{eq:magic}) 
using the Girsanov formula. 

 Let $n\geq 1$ and $m\geq 1$. As in the proof of (\ref{eq:gir1}), we have that 
 
 \begin{equation}\label{eq:girc1} \esp^t[\Psi_{n,m}]=\esp^0[\Psi_{n,m}\Pi_{i=1}^{n+m}\frac{\mu_t(X_i)}{\mu_0(X_i)}]\,.
 \end{equation} 
 
 The derivative of the expression $\Pi_{i=1}^{n+m}\frac{\mu_t(X_i)}{\mu_0(X_i)}$ equals 
 $(\sum_{j=1}^{n+m}\nu_t(X_j))\Pi_{i=1}^{n+m}\frac{\mu_t(X_i)}{\mu_0(X_i)}$. 
 This last random variable is bounded by $C(n+m)\exp({C(n+m)})$. Therefore we may exchange the expectation and the derivative 
 when differentiating in formula (\ref{eq:girc1}) and thus we obtain that 
 
$$\frac{d}{dt} \esp^t[\Psi_{n,m}]
  =\esp^0[\Psi_{n,m}(\sum_{j=1}^{n+m}\nu_t(X_j))\Pi_{i=1}^{n+m}\frac{\mu_t(X_i)}{\mu_0(X_i)}]$$
   \begin{equation}\label{eq:deri} =\esp^t[\Psi_{n,m}\sum_{j=1}^{n+m}\nu_t(X_j)]\,.\end{equation} 
 
 We use H\"older's  inequality to bound this last expression: choose $p$ as in assumption (ii) of the theorem and let $q=p/(p-1)$. 
 The term $\esp^t[\vert \Psi_{n,m}\vert^p]$ is bounded by $\tau_p(\mu_t;{\cal Q})$. Therefore 
 $$ \vert \esp^t[\Psi_{n,m}\sum_{j=1}^{n+m}\nu_t(X_j)]\vert\leq \tau_p(\mu_t;{\cal Q})^{1/p}\, \esp^t[\vert \sum_{j=1}^{n+m}\nu_t(X_j) \vert^q]^{1/q}\,.$$
 We use Burkholder's inequality as stated in Lemma \ref{lm:burkholder} to deduce that 
 $$ \esp^t[\vert \sum_{j=1}^{n+m}\nu_t(X_j) \vert^q]\leq c_B(q)\, \esp^t[(\sum_{j=1}^{n+m}\vert \nu_t(X_j) \vert^2)^{q/2}]$$
 $$\leq c_B(q)\, C^q\, (n+m)^{q/2}\,.$$

Thus we conclude that 
$$ \vert \frac{d}{dt} \esp^t[\Psi_{n,m}]\vert 
\leq C\, c_B(q)^{1/q}\, \tau_p(\mu_t;{\cal Q})^{1/p}   \sqrt{n+m}\,.
$$
In particular 
\begin{equation}\label{eq:girc2} 
 \vert \frac{d}{dt} \esp^t[\Psi_{2^k,2^k}]\vert 
 = \vert \esp^t[\Psi_{2^k,2^k}\sum_{j=1}^{2^{k+1}}\nu_t(X_j)]\vert 
\leq C\, c_B(q)^{1/q}\, \tau_p(\mu_t;{\cal Q})^{1/p}  \, 2^{\frac{k+1}2}\,.
\end{equation} 
Remember that we are assuming that $\sup_{t\in[-1,1]}\tau_p(\mu_t;{\cal Q})$ is finite. So that the sum 
$$\sum_k 2^{-k-1} \sup_{t\in[-1,1]} \tau_p(\mu_t;{\cal Q})^{1/p}  \, 2^{\frac{k+1}2}$$ converges. 

Recall from Proposition \ref{prop:magic} that 
$$\esp^t[Z]=\sum_{k\geq 0} 2^{-k-1} \esp^t[\Psi_{2^k,2^k}]\,.$$ 
We use the bound (\ref{eq:girc2}) to compute the derivative of $\esp^t[Z]$ by differentiating this last expression term by term. We get that 

\begin{equation}\label{eq:whattodo} \frac{d}{dt}\esp^t[Z]
=\sum_{k\geq 0} 2^{-k-1}\esp^t[\Psi_{2^k,2^k}\sum_{j=1}^{2^{k+1}}\nu_t(X_j)]\,.
\end{equation}

We also deduce from (\ref{eq:girc2}) that the function $t\ra \esp^t[Z]$ is $C^1$. 

A very similar argument shows that 
the function $t\ra \esp^t[Q_1]$ is $C^1$ and we conclude using Proposition \ref{prop:magic} 
that the function $t\ra \ell(\mu_t;{\cal Q})$ is $C^1$. \qed

\section{Lipschitz regularity of the entropy}
\label{sec:entrop}

In this section we only deal with the ``random walk case''.

One may deduce the Lipschitz continuity of the entropy from Theorem \ref{theo:rescp} using the identification of the entropy as a rate of escape in the so-called Green metric, 
see paragraph \ref{par:Green} below. 
This argument is reminiscent of the proof in part 4 of \cite{kn:hype}. Because the Green metric is a true distance (i.e. symmetric) only when $\mu$ is itself symmetric, we have to  
restrict ourselves to symmetric measures. 

Let $B$ be a subset of $G$. In the sequel, $\p_s(B)$ will denote the set of symmetric probability measures with support $B$. 

We shall assume that $G$ is finitely generated and further impose that $G$ is non-amenable.

\begin{theo}\label{theo:entropsym} 
Assume that $G$ is finitely generated and non-amenable. 
Let $B$ be a (finite or infinite) symmetric generating subset of $G$ and choose a symmetric measure $\mu\in\p_s(B)$. 
Assume that there exists a neighborhood of $\mu$ in $\p_s(B)$, say ${\cal N}_0$, such that the first-moment deviation inequality (\ref{eq:devlength}) 
holds uniformly for $\mu\in{\cal N}_0$ and also uniformly with respect to all the Green metrics $d_{\cal G}^{\mu'}$ associated with a measure $\mu'$ in ${\cal N}_0$. 
Assume that $\mu$ has a finite first moment. 

Then there exists a neighborhood of $\mu$ in $\p_s(B)$, say $\cal N$, such that the function $\mu\to h(\mu)$ is Lipschitz continuous on $\cal N$. 
\end{theo}

\begin{theo} \label{theo:diffentrop} Assume that $G$ is finitely generated and non-amenable. 
Let $B$ be a (finite or infinite) symmetric generating subset of $G$ and choose a symmetric measure $\mu\in\p_s(B)$. 
Assume that $\mu$ has a finite second moment, satisfies the second-moment deviation inequality and the locally uniform first-moment deviation inequality 
in the Green metric $d_{\cal G}^{\mu}$. 

Then the function $\mu_0\ra h(\mu_0)$ is differentiable at $\mu_0=\mu$ in the following sense: 
Let $(\mu_t, t\in[0,1])$ be a curve in $\p(B)$ such that $\mu_0=\mu$ and, for all $a\in B$, the function 
$t\ra \log \mu_t(a)$ has a derivative at $t=0$, say $\nu(a)$. We assume that $\nu$ is bounded on $B$ and also that 
$\sup_{t\in[0,1]}\sup_{a\in B} \vert \frac 1 t \log\frac{\mu_t(a)}{\mu_0(a)} -\nu(a)\vert<\infty$. Then the limit of $\frac 1 t (h\mu_t)-h(\mu))$ as $t$ tends to $0$ exists. 
Furthermore, this limit coincides with the covariance
\beqn\sigma_{\cal G}(\nu,\mu):=\lim_n\frac 1 n \esp^\mu[d_{\cal G}^\mu(id,Z_n)\big(\sum_{j=1}^n \nu(X_j)\big)]\,.\eeqn
 \end{theo}

\subsection{Green metrics} \label{par:Green} 

Let us first recall some useful facts about the Green metric. 

Let $G$ be a non-amenable group. Let $\mu$ be a symmetric probability measure on $G$ whose support generates the whole group. 

We recall that there exists a constant, $\rho_\mu<1$ - the {\bf spectral radius} - such that 
\beqn\label{eq:specrad} \mu^n(z)\le (\rho_\mu)^n\,,\eeqn 
for all $n\ge 0$ and $z\in G$, see \cite{kn:woess}. 

The Green function is defined by 
$$G^\mu(x):=\sum_{n=0}^\infty \mu^n(x)\,.$$ 
Because of (\ref{eq:specrad}), the series defining $G^\mu$ does converge. 
The Green distance between points $x$ and $y$ in $G$ is then 
$$d^\mu_{\cal G}(x,y):=\log G^\mu(id)-\log G^\mu(x^{-1}y)\,.$$ 
This defines a proper left-invariant distance on $G$. Moreover, it follows from (\ref{eq:specrad}) that $d_{\cal G}^\mu$ is bi-Lipschitz equivalent to word metrics. See \cite{kn:blbr} and \cite{kn:bhm2} for the details. Observe that $d_{\cal G}^\mu$ need not be geodesic. 

We may equivalently express $d_{\cal G}^\mu$ in terms of the hitting probabilities of the random walk: 
for a given trajectory $\o\in\Omega$ and $z\in G$, let 
$$T_z(\o)=\inf\{n\geq0\,;\, Z_n(\o)=z\}\,$$ 
be the hitting time of $z$ by $\o$. Observe that $T_z(\o)$ may be infinite. 

Define $F^\mu(z):=\P^\mu[T_z<\infty]$. Then 
$$d_{\cal G}^\mu(id,z)=-\log F^\mu(z)\,,$$ 
as can be easily checked using the Markov property. 


In \cite{kn:bhm1} (see also \cite{kn:bepe}), we proved that 
\beqn\label{eq:ident} h(\mu)=\ell(\mu;d^\mu_{\cal G})\,.\eeqn

It makes sense to define the Green metric through the Green function as soon as the random walk is transient. 
The identification (\ref{eq:ident}) is also valid in this extended framework but we shall not need it here. 

\subsection{Fluctuations of the Green metric}

Our first aim is to control the fluctuations between two Green metric, say $d^{\mu_0}_{\cal G}$ and $d^{\mu_1}_{\cal G}$.

We use the same notation as in the beginning of Part \ref{sec:proofrescp}: 
Let $\mu_0$ and $\mu_1$ belong to $\p_s(B)$. 
For $t\in[0,1]$ and $a\in B$, we define $\mu_t(a):=\mu_0(a)+t(\mu_1(a)-\mu_0(a))$ and 
$\nu_t(a)=(\mu_1(a)-\mu_0(a))/\mu_t(a)$.  
Then $\Nu(\mu_0,\mu_1)=\sup_{a\in B}\sup_{t\in[0,1]}\vert\nu_t(a)\vert$. 

\begin{prop}\label{prop:flucgreen} 
Let $G$ be a finitely generated non-amenable group equipped with a word metric denoted with $d$. Let $B$ be a symmetric generating subset of $G$.\\ 
For any $\mu$ in $\p_s(B)$, there exist $\eps_\mu>0$ and $k_\mu$ such that for any two symmetric measures $\mu_0$ and $\mu_1$ 
in $\p_s(B)$ satisfying
\begin{equation} \label{eq:assmpt} \Nu(\mu,\mu_0)+\Nu(\mu,\mu_1)\le\eps_\mu\,,\end{equation}  
then 
\begin{equation} \label{eq:concl} \vert d^{\mu_1}_{\cal G}(id,z)-d^{\mu_0}_{\cal G}(id,z)\vert\le k_\mu\, \Nu(\mu_0,\mu_1)d(id,z)\,,\end{equation} 
for all $z\in G$. 
\end{prop} 

We use the shorthand notation $\esp^t$ (resp. $\P^t$) instead of $\esp^{\mu_t}$ (resp. $\P^{\mu_t}$) 
and $d^t_{\cal G}$ instead of $d^{\mu_t}_{\cal G}$ and $F^t$ instead of $F^{\mu_t}$. 

The proof of Proposition \ref{prop:flucgreen} is based on the following Lemma: 

\begin{lm} \label{lem:flucgreen} 
In the context of Proposition \ref{prop:flucgreen} and with the same notation, if the conditional expectation of $T_z$ is finite, then it satisfies  
\beqn\label{eq:flucgreen} 
\esp^t[T_z\bigr\vert T_z<\infty]\le k_\mu\, d(id,z)\,,\eeqn 
for all $t\in[0,1]$ and $z\in G$ for some constant $k_\mu$. 
\end{lm} 

\begin{proof}
 Let $\mu\in\p_s(B)$. Recall that $\rho_\mu<1$. 

Let $\rho':=\frac 12(1+\rho)$. 
Choose $\eps_\mu$ so small that measures satisfying (\ref{eq:assmpt}) are such that 
$\rho_{\mu_t}\le\rho'$ for all $t\in[0,1]$. 

Also assume that $\eps_\mu$ is such that there exists $\gamma>0$ such that, for all $t\in[0,1]$ and for all $z\in G$ then 
\beqn\label{eq:hitt}\P^t[T_z<\infty]\ge \gamma^{d(id,z)}\,.\eeqn

Both these conditions are ensured by the following: since $B$ generates $G$ and since $G$ is finitely generated, then there exists a finite sub-set of $B$, say $\tilde B$, 
that generates $G$. The uniform upper bound on the spectral radius as well as the uniform lower bound on the probability of hitting a point $z$ in $G$ are both obtained 
once we choose $\eps_\mu$ such that all measures $\mu_t$ are uniformly bounded from below on $\tilde B$. 

By (\ref{eq:specrad}), we have 
$$\P^t[T_z=k]\le \P^t[Z_k=z]\le(\rho')^k\,.$$ 
Therefore, for any $c>0$, 
$$\esp^t[T_z\big\vert T_z<\infty]\le c\, d(id,z)+\gamma^{-d(id,z)}\sum_{k\ge c\, d(id,z)} k\, (\rho')^k\,.$$ 

It only remains to choose $c$ large enough so that $\gamma^{-d(id,z)}\sum_{k\ge c\, d(id,z)} k\, (\rho')^k\le 1$. 
\end{proof}

{\it Proof of Proposition \ref{prop:flucgreen}.} 
Let $N$ be an integer. 

The Girsanov formula (\ref{eq:gir}) implies that 
$$\P^t[T_z\le N]=\esp^0[\1_{T_z\le N} \Pi_{j=1}^N\frac{\mu_t(X_j)}{\mu_0(X_j)}]\,.$$ 
Taking the derivative with respect to $t$, we get that 
$$\frac{d}{dt} \P^t[T_z\le N]=\esp^t[\1_{T_z\le N} \sum_{j=1}^N\nu_t(X_j)]\,.$$ 
The martingale property implies that 
$$\esp^t[\1_{T_z\le N} \sum_{j=1}^N\nu_t(X_j)]=\esp^t[\1_{T_z\le N} \sum_{j=1}^{T_z}\nu_t(X_j)]\,,$$ 
so that 
$$\frac{d}{dt} \P^t[T_z\le N]=\esp^t[\1_{T_z\le N} \sum_{j=1}^{T_z}\nu_t(X_j)]\,,$$ 
and $$\frac{d}{dt} \P^t[T_z\le N]\le\Nu(\mu_0,\mu_1)\, \esp^t[T_z\1_{T_z\le N}]\,.$$

Choose $N$ large enough so that $\P^t[T_z\le N]\not=0$ for all $t$, and use Lemma \ref{lem:flucgreen} to get that 
$$\frac 1{ \P^t[T_z\le N]}\frac{d}{dt} \P^t[T_z\le N]\le\Nu(\mu_0,\mu_1)\, k_\mu\, d(id,z) \frac{\P^t[T_z<\infty]}{\P^t[T_z\le N]}\,,$$ 
and therefore 
$$ \log \P^1[T_z\le N]-\log \P^0[T_z\le N]
\le \Nu(\mu_0,\mu_1)k_\mu\, d(id,z) \int_0^1 \frac{\P^t[T_z<\infty]}{\P^t[T_z\le N]}\,dt\,.$$ 
We now let $N$ tend to $+\infty$. Observe that there exist $N_0$ and $\eps$ such that, for all $t$, then 
$\P^t[T_z\le N]\ge \eps$ for all $N\ge N_0$. Thus we may apply the dominated convergence Lemma to deduce that 
$$ \log \P^1[T_z<\infty]-\log \P^0[T_z<\infty]
\le \Nu(\mu_0,\mu_1)k_\mu\, d(id,z)\,.$$ 
Exchanging the roles of $\mu_0$ and $\mu_1$ leads to 
$$ \big\vert\log \P^1[T_z<\infty]-\log \P^0[T_z<\infty]\big\vert 
\le \Nu(\mu_0,\mu_1)k_\mu\, d(id,z)\,.$$ 
\qed 

{\it Proof of Theorem \ref{theo:entropsym}.}
Write
\beqnn &&h(\mu_1)-h(\mu_0)= 
\ell(\mu_1;d^1_{\cal G})-\ell(\mu_0;d^0_{\cal G}) \\ 
&=&\big(\ell(\mu_1;d^1_{\cal G})-\ell(\mu_0;d^1_{\cal G})\big)+\big(\ell(\mu_0;d^1_{\cal G})-\ell(\mu_0;d^0_{\cal G})\big):=I+II\,.\eeqnn

We argue that both terms I and II are bounded by $C\Nu(\mu_0,\mu_1)$ for some $C$, uniformly in a small enough neighborhood of $\mu$ in $\p_s(B)$. 

To handle Term I, first argue as in the beginning of the proof of Lemma \ref{lem:flucgreen} to choose a neighborhood of $\mu$, say $\cal N$,  so that (\ref{eq:hitt}) holds for some constant 
$\gamma$ uniformly in $\cal N$. Then we have $d^1_{\cal G}\leq (\log \frac 1\gamma)d$ for all $\mu_1\in{\cal N}$ and therefore 
$\chi_1(\mu;d^1_{\cal G})\leq (\log \frac 1\gamma) \chi_1(\mu;d)$. 

From Lemma \ref{lm:simple}, we deduce that 
$\chi_1(\mu_1;d^1_{\cal G})\leq (1+\Nu(\mu_1,\mu))(\log \frac 1\gamma) \chi_1(\mu;d)$ 
and applying now Proposition \ref{prop:lip}, we obtain that, for $\mu_1$ and $\mu_0$ in $\cal N$, 
$$\ell(\mu_1;d^1_{\cal G})-\ell(\mu_0;d^1_{\cal G})\leq 
\Nu(\mu_0,\mu_1)\big((4+2\Nu(\mu_0,\mu)+2\Nu(\mu_1,\mu))(\log \frac 1\gamma) \chi_1(\mu;d)+4\sup_{t\in[0,1]} \tau_1(\mu_0+t(\mu_1-\mu_0);d^1_{\cal G})  \big).$$

By assumption, we may also impose that $\sup_{t\in[0,1]} \tau_1(\mu_0+t(\mu_1-\mu_0);d^1_{\cal G})$ is uniformly bounded for all  $\mu_1$ and $\mu_0$ in $\cal N$. 
(Note that when $\mu_1$ and $\mu_0$ are in $\cal N$, then $\mu_0+t(\mu_1-\mu_0)$ also belongs to $\cal N$.)Thus we indeed obtain a bound 
of the form $\ell(\mu_1;d^1_{\cal G})-\ell(\mu_0;d^1_{\cal G})\leq C \Nu(\mu_0,\mu_1)$ for all $\mu_1$ and $\mu_0$ in $\cal N$.


For Term II, we rely on Proposition \ref{prop:flucgreen}. 
We have 
$$\vert d^1_{\cal G}(id,Z_n)-d^0_{\cal G}(id,Z_n)\vert\le k_\mu\, \Nu(\mu_0,\mu_1) d(id,Z_n)\,.$$ 
Taking the expectation with respect to $\esp^0$, dividing by $n$ and letting $n$ tend to $\infty$, we get that 
$$\vert\ell(\mu_0;d^1_{\cal G})-\ell(\mu_0;d^0_{\cal G})\vert\le k_\mu\, \Nu(\mu_0,\mu_1)\ell(\mu_0;d)\,.$$ 
It then suffices to note that $\ell(\mu_0;d)\le \esp^0[d(id,X_1)]$ is bounded on a neighborhood of $\mu$, see Lemma \ref{lm:simple}. 
\qed 

{\it Proof of Theorem \ref{theo:diffentrop}.} 

Recall that $h(\mu)=\ell(\mu;d^\mu_{\cal G})$ and write that 
\beqnn \frac 1 t\big(h(\mu_t)-h(\mu_0)\big)=\frac 1 t \big(h(\mu_t)-\ell(\mu_t;d^\mu_{\cal G})\big)+\frac 1 t \big(\ell(\mu_t;d^\mu_{\cal G})-\ell(\mu;d^\mu_{\cal G})\big)\,.\eeqnn 
We apply Theorem \ref{theo:diffrate} in the Green metric $d_{\cal G}^\mu$ and deduce that the second term converges to $\sigma_{\cal G}(\nu,\mu)$. 
The first term goes to $0$ by Proposition 4.1 in \cite{kn:hype}.\qed


\part{Getting deviation inequalities}\label{part:geom}

\section{Acylindrically hyperbolic groups} \label{sec:acylin} 
Recall that a geodesic metric space is (Gromov-)hyperbolic if there exists $\delta\geq 0$ so that for any geodesic triangle $[p,q]\cup [q,r] \cup [r,p]$ and each $x\in [p,q]$ there exists $y\in [q,r] \cup [r,p]$ so that $d_X(x,y)\leq \delta$. A finitely generated group $G$ is hyperbolic if some (equivalently, any) Cayley graph of $G$ is hyperbolic.

The groups for which we can prove deviation inequalities are the so-called \emph{acylindrically hyperbolic groups}. Such class of groups vastly generalises the class of hyperbolic groups and includes non-elementary relatively hyperbolic groups, Mapping Class Groups, $Out(F_n)$, many groups acting on CAT(0) cube complexes (e.g. right-angled Artin groups that do not split as a direct product), possibly infinitely presented small cancellation groups, and many subgroups of the above. In particular, we will refine and generalise the results in \cite{Si-tracking}, some of whose techniques are used here as well.

Acylindrical hyperbolicity is defined in terms of an ``interesting enough'' action on some hyperbolic space. ``Interesting enough'' means in this case acylindrical and non-elementary, see Section \ref{prelim} for the definitions. Roughly speaking, acylindricity says that the coarse stabiliser of any two far away points is finite, and being non-elementary is a non-triviality-type condition. Acylindrically hyperbolic groups have been defined by Osin who showed in \cite{Os-acyl} that several approaches to groups that exhibit rank one behaviour \cite{BeFu-wpd,Ha-isomhyp, DGO, Si-contr} are all equivalent. Acylindrical hyperbolicity has strong consequences: Every acylindrically hyperbolic group is SQ-universal (in particular it has uncountably many pairwise non-isomorphic quotients), it contains free normal subgroups \cite{DGO}, it contains Morse elements and hence all its asymptotic cones have cut-points \cite{Si-hypembmorse}, and its bounded cohomology is infinite dimensional in degrees 2 \cite{HullOsin} and 3 \cite{FPS-H3b-acylhyp}. Moreover, if an acylindrically hyperbolic group does not contain finite normal subgroups, then its reduced $C^*$-algebra is simple \cite{DGO} and every commensurating endomorphism is an inner automorphism \cite{AMS-pwise-inner}.

We will in fact not only consider word metrics on acylindrically hyperbolic groups but more generally metrics coming from actions with certain properties, see Definition \ref{def:acylinter}. In particular, this covers the metric that an acylindrically hyperbolic group inherits from the hyperbolic space it acts on. Perhaps even more interestingly, it covers the metric coming from the action of a Mapping Class Group on the corresponding Teichm\"uller space endowed with either the Teichm\"uller of the Weil-Petersson metric, and the metric inherited by the fundamental group of a finite-volume hyperbolic $n$-manifold (an example of relatively hyperbolic group) on $\mathbb H^n$.


\section{Preliminaries}\label{prelim}
A discrete path is an ordered sequence of points $\alpha=(w_i)_{k_1\leq i\leq k_2}$ in a metric space $Y$. Its length $l_Y(\alpha)$ is defined as $\sum d(w_i,w_{i+1})$. The notions of Lipschitz and quasi-geodesic discrete paths are defined regarding a discrete path as a map from an interval in $\Z$ to $Y$ (notice that a discrete path is $L$-Lipschitz if the ``jumps'' have size at most $L$). We will often omit the adjective discrete.


\subsection{Acylindrical actions}
Let the group $G$ act by isometries on the metric space $X$. The action is called {\bf acylindrical} if for every $r\geq 0$ there exist $R,N\geq 0$ so that whenever $x,y\in X$ satisfy $d_X(x,y)\geq R$ there are at most $N$ elements $g\in G$ so that $d_X(x,gx),d_X(y,gy)\leq r$. Also, we will say that the action is {\bf non-elementary} if orbits are unbounded and $G$ is not virtually cyclic (cfr. \cite[Theorem 1.1]{Os-acyl}).

When an action of a group $G$ on the metric space $X$ and a word metric $d_G$ on $G$ have been fixed, we denote by $diam^{\ast}$ the diameter, by $B^{\ast}(\cdot, R)$ a ball of radius $R$ and by $N^{\ast}_t$ a neighborhood of radius $t$, where $\ast$ can be either $G$ or $X$ depending on which metric we are using to define the given notion.

We will need the following lemma about acylindrical actions (a similar lemma is exploited in \cite{Si-hypembmorse}).

\begin{lm}\label{geomsep}
 Let the finitely generated group $G$ act acylindrically on the hyperbolic geodesic metric space $X$. Let $\pi:G\to X$ be an orbit map with basepoint $x_0$, and endow $G$ with a word metric $d_G$.
 
  Then there exists $L$ and a non-decreasing function $f$ so that for each $l_1,l_2\geq 0$, each $t\geq 0$ and whenever $x,y\in Gx_0$ satisfy $d_X(x,y)\geq L+l_1+l_2$, we have
 $$diam^G(\ \pi^{-1}(B^X(x,l_1))\cap N^G_t (\,\pi^{-1}(B^X(y,l_2))\,)\ )\leq f(t).$$
\end{lm}

\begin{proof}
First of all, we choose the constants. Let $\delta$ be the hyperbolicity constant of $X$. Let $R,N$ be as in the definition of acylindrical action with $r=4\delta+1$. Finally, let $L=R+6\delta+1$.
  
Up to applying an element of $G$, we can and will assume $x=x_0$ throughout the proof. Fix some $t$ from now on. Let $\{h_i\}_{i=1,\dots,k}$ be the finitely many elements of $B^G(id,t)$ and for each $i$ let $\gamma_i$ be a geodesic in $X$ from $x_0$ to $h_ix_0$.
 
 \par\medskip
 
 {\bf Claim 1:} There exist only finitely many $g\in G$ so that $diam^X(N_{2\delta}^X(\gamma_i)\cap g\gamma_j)\geq R$ for some $i,j$. In particular, the diameter of the set of such $g$'s is bounded by, say, $f(t)$.
 
  \par\medskip
 
\emph{Proof of Claim 1.} Since there are only finitely many $h_i$'s, we can fix $i,j$. Suppose by contradiction that there exist infinitely many distinct elements $g_k$, $k=0,1,\dots$, so that $diam^X(N_{2\delta}^X(\gamma_i)\cap g_k\gamma_j)\geq R$. Hence, for every $k$ there exist points $p_k,q_k\in\gamma_j$ and $p'_k,q'_k\in\gamma_i$ so that $d_X(g_kp_k,p'_k),d_X(g_kq_k,q'_k)\leq 2\delta$ and $d_X(p_k,q_k)\geq R$. By compactness of $\gamma_i$ and $\gamma_j$, we can pass to a subsequence of $\{g_k\}$ and assume that for every $k_1,k_2$ we have that each of $d_X(p_{k_1},p_{k_2}),d_X(q_{k_1},q_{k_2}),d_X(p'_{k_1},p'_{k_2}),d_X(q'_{k_1},q'_{k_2})$ is at most $1/2$. In particular, for every $k$ we have $d_X(p_0,g_0^{-1}g_kp_0)=d_X(g_0p_0,g_kp_0)\leq 4\delta+1$ and similarly $d_X(q_0,g_0^{-1}g_kq_0)\leq 4\delta+1$. But this is a contradiction because acylindricity implies that there can be only finitely many $a$ so that $d_X(p_0,ap_0),d_X(q_0,aq_0)\leq 4\delta+1$ (since $d_X(p_0,q_0)\geq R$). This completes the proof of the claim.\qed

\begin{figure}[h]
\begin{minipage}[b]{0.4\textwidth}
\centering
 \includegraphics[scale=0.75]{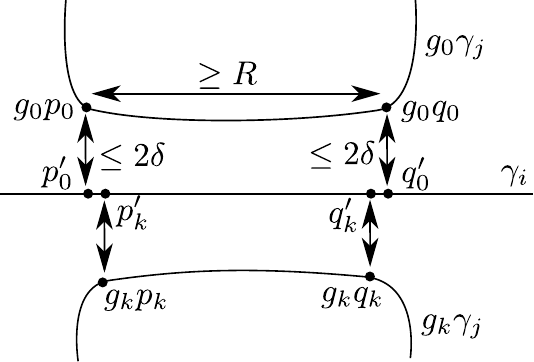}
 \caption{Claim 1. Translates of $\gamma_j$ stay close to the same subgeodesic of $\gamma_i$.}
\end{minipage}
\hspace{0.2cm}
\begin{minipage}[b]{0.55\textwidth}
\centering
 \includegraphics[scale=0.75]{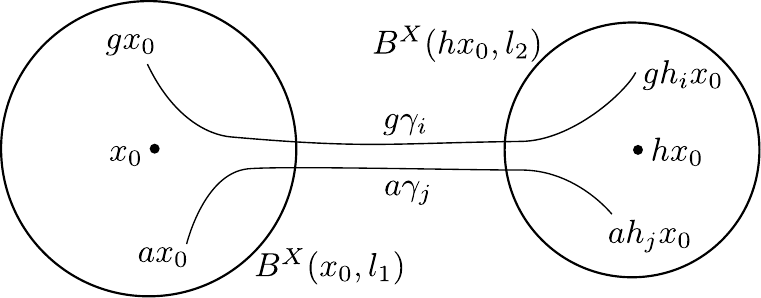}
 \caption{Claim 2. $g\gamma_i$ and $h\gamma_j$ fellow-travel between $B^X(x_0,l_1)$ and $B^X(hx_0,l_2)$.}
\end{minipage}
\end{figure}

For $h\in G$ and $l_1,l_2\geq 0$, denote $A_{l_1,l_2}(h)=\pi^{-1}(B^X(x_0,l_1))\cap N^G_t (\pi^{-1}(B^X(hx_0,l_2)))$. Let $h\in G$ be so that $d_X(x_0,hx_0)\geq L+l_1+l_2$. Assume that $A_{l_1,l_2}(h)$ is non-empty, for otherwise there is nothing to prove, and let $g\in A_{l_1,l_2}(h)$. Notice that there exists $i$ so that $gh_ix_0\in B^X(hx_0,l_2)$. Similarly, for each $a\in A_{l_1,l_2}(h)$ there exists $j=j(a)$ so that $ah_jx_0\in B^X(hx_0,l_2)$.

 \par\medskip

{\bf Claim 2:} For any $a\in A_{l_1,l_2}(h)$ we have $diam^X(N_{2\delta}^X(g\gamma_i)\cap a\gamma_j)\geq R$, where $j=j(a)$.

 \par\medskip

\emph{Proof of Claim 2.} Consider a geodesic quadrangle in $X$ containing $g\gamma_i$ and $a\gamma_j$ with vertices $gx_0, ax_0, gh_ix_0,ah_jx_0$. Notice that the first two vertices belong to $B_1=B^X(x_0,l_1)$ while the other vertices belong to $B_2=B^X(hx_0,l_2)$, and by our hypothesis on $h$ the distance between $B_1$ and $B_2$ is at least $L$. Also, considering a triangle with vertices $x_0,gx_0,ax_0$ it is readily seen that the geodesic $[gx_0, ax_0]$ is contained in the $\delta$-neighborhood of $B_1$, and that, similarly, $[gh_ix_0,ah_jx_0]$ is contained in the $\delta$-neighborhood of $B_2$. Consider now a subgeodesic $\gamma$ of $a\gamma_j$ of length at least $L-6\delta-1$ that does not intersect the $3\delta$-neighborhood of $B_1\cup B_2$. By $2\delta$-thinness of geodesic quadrangles, any point of $\gamma$ is $2\delta$ close to $g\gamma_i$, since it cannot be $2\delta$-close to either $[gx_0, ax_0]$ or $[gh_ix_0,ah_jx_0]$. Hence, $diam^X(N_{2\delta}^X(g\gamma_i)\cap a\gamma_j)\geq L-6\delta-1=R$, as required.\qed

In view of Claim 2 we get that, for each $a\in A_{l_1,l_2}(h)$, $g^{-1}a$ belongs to the finite set given by Claim 1. In particular $diam^G(A_{l_1,l_2}(h))\leq f(t)$, as required.
\end{proof}

\section{Linear progress with exponential decay}\label{sec:linearpr} 

When a group $G$ acting on a hyperbolic space $X$ is fixed, we will implicitly make a choice of basepoint $x_0\in X$ and, to simplify the notation, write $d_X(g,h)$ instead of $d_X(gx_0,hx_0)$ when $g,h\in G$.

In the above setting, we say that a semigroup is \emph{non-elementary} if it contains two loxodromic elements that freely generate a free group. When $\mu$ is a measure on the group $G$, the semigroup generated by the support of $\mu$ is the smallest sub-semigroup of $G$ containing the support of $\mu$.

Recall that we defined a distance, whence a topology, on measures on a given group in Subsection \ref{ssec:distance}.

\begin{theo}\label{linprog}
 Let the finitely generated group $G$ act acylindrically on the geodesic hyperbolic space $X$. Then any measure $\mu_0$ with exponential tail whose support generates a non-elementary semigroup has a neighborhood, say $\mathcal N$, so that there exists $C$ with  the following property. For any $\mu\in \mathcal N$, any positive integer $n$ and any $g_0\in G$ we have
$$\matP^\mu[d_X(id,g_0Z_n)-d_X(id,g_0)\leq n/C]\leq Ce^{-n/C}.$$
In particular, the rate of escape measured in the metric $d_X$ of the random walk driven by $\mu$ is strictly positive.
\end{theo}

\begin{rmk}
 The condition on the support of $\mu_0$ will also appear in Theorems \ref{smalldevhier}, \ref{smalldevgen}. Notice that if $G$ acts non-elementarily on $X$ such condition is weaker than requiring that the semigroup generated by the support is $G$.
\end{rmk}

First of all, we remark that it is enough to show the theorem for the measure $\mu_0$.

\begin{lm}\label{lem:uniform_in_mu}
Let $G,X,x_0,\mu_0$ be as in Theorem \ref{linprog} and suppose that there exists $C$ so that for any integer $n$ we have $\matP^{\mu_0}[d_X(id,Z_n)\leq n/C]\leq Ce^{-n/C}$. Then there exists a neighborhood $\mathcal N$ of $\mu_0$ so that for any $\mu\in \mathcal N$ and any positive integer $n$, we have $\matP^\mu[d_X(id,Z_n)\leq n/C]\leq Ce^{-n/2C}.$
\end{lm}

\begin{proof}
 Let $\mu$ be so that $\Nu(\mu,\mu_0)\leq\epsilon$, where $\epsilon$ will be chosen later.
Recall that by the Girsanov formula for any non-negative measurable function $F:G^n\to\R_+$, we have
$$\esp^{\mu}[F(X_1,\dots,X_n)]=\esp^{\mu_0}\left[F(X_1,\dots,X_n)\Pi_{j=1}^n \frac{\mu(X_j)}{\mu_0(X_j)}\right].$$
Since $\mu(a)/\mu_0(a)\le 1+\epsilon$ for each $a\in G$, we have
$$\esp^{\mu}[F(X_1,\dots,X_n)]\le (1+\epsilon)^n \esp^{\mu_0}[F(X_1,\dots,X_n)].$$
Using this inequality with $F=\mathbbm{1}_A$ where $A$ is the event ``$d_X(id,Z_n)\le n/C$'' yields:
$$\matP^{\mu}[d_X(id,Z_n)\le n/C]\le C(1+\epsilon)^n e^{-n/C}.$$
It is then enough to choose $\epsilon$ small enough so that $(1+\epsilon)^ne^{-n/C}\le e^{-n/2C}$. 
\end{proof}

Fix the notation of the theorem from now on. In view of the lemma, we can fix $\mu=\mu_0$. We will write $\matP$ instead of $\matP^\mu$. All Gromov products are taken with respect to $d_X$, meaning that $(g,h)_k$ denotes the Gromov product $(gx_0,hx_0)_{kx_0}$ taken in $X$.

\begin{prop}\label{further}
 There exist $C,k>0$ with the following properties. For every $g\in G$ we have
\begin{enumerate}
 \item For every $h\in G$
$$\matP[(g,gZ_kh)_{id}\leq d_X(id,g)-C]\leq 1/10.$$
\item
$$\lim_{m\to\infty}\matE[d_X(id,Z_m)]=+\infty.$$
\end{enumerate}
\end{prop}

\begin{proof}
1) For convenience we will assume $d_X\leq d_G$, which can be arranged by rescaling the metric on $X$. The notation $[gx_0,hx_0]$ will denote any choice of a geodesic in $X$ from $gx_0$ to $hx_0$. Let $\delta$ be a hyperbolicity constant for $X$. We will use the following deterministic lemma.

\begin{lm}
\label{varys}
 There exist $C_0,D$ with the following property. For each $g,h\in G$ and integer $k\geq 1$ the set $A(g,h,k)$ of elements $s\in  B^G(id,k)$ so that $diam^X([x_0,gx_0]\cap N^X_{2\delta}([gsx_0,gshx_0]))\geq D$ has cardinality at most $C_0k^2$.
\end{lm}

\begin{proof}
By acylindricity, there exist $R,N$ so that whenever $x,y\in X$ satisfy $d_X(x,y)\geq R$, there are at most $N$ elements $s$ so that $d_X(x,sx),d_X(y,sy)\leq 100\delta+100$. We now argue that whenever $x,y,x',y'\in X$ are so that $d_X(x,y)\geq R$, there are at most $N$ elements $s$ so that $d_X(x',sx),d_X(y',sy)\leq 50\delta+50$. Fix $x,y,x',y'$ as before, and let $\mathcal S$ be the set of all $s$ so that $d_X(x',sx),d_X(y',sy)\leq 50\delta+50$. If $\mathcal S$ is empty, we are done, otherwise let $s_0\in S$. Notice that for any $s\in \mathcal S$ we have
$$d_X(x,s^{-1}s_0 x)=d_X(sx,s_0 x)\leq d_X(sx,x')+d_X(x',s_0 x)\leq 100\delta+100,$$
and similarly $d_X(y,s^{-1}s_0 y)\leq 100\delta+100$. Hence, there are at most $N$ choices for $s$ by the choice of $R,N$, as required.

We choose $D=R+100\delta+100$. Suppose $diam^X([x_0,gx_0]\cap N^X_{2\delta}([gsx_0,gshx_0]))\geq D$. This means that there exist $p'_1,p'_2\in [x_0,gx_0]$ and $q'_1,q'_2\in [x_0,hx_0]$ so that $d_X(p'_i,gsq'_i)\leq 2\delta$ and $d_X(p'_1,p'_2)\geq D$. We will choose the indexing so that $p'_1$ is closer to $gx_0$ than $p'_2$. We now want to, roughly speaking, impose more conditions on $p'_1,p'_2,q'_1,q'_2$ to ensure that there are only about $k^2$ choices for the quadruple.

\begin{figure}[h]
\centering
 \includegraphics[scale=0.9]{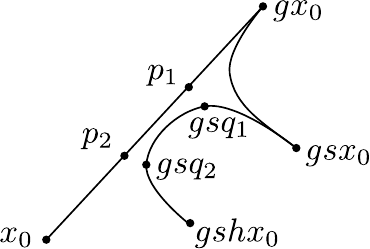}
\end{figure}

First, we replace $p'_1$ with a point at distance at most $k$ from $gx_0$, and along the way we also construct a point $q''_1\in[x_0,hx_0]$ for later purposes. If $d_X(p'_1,gx_0)\leq k$, set $p_1=p'_1$, and $q''_1=q'_1$. Otherwise, if $d_X(p'_1,gx_0)>k$,  let $p_1$ on $[x_0,gx_0]$ lie at distance $k$ from $gx_0$. Consider a quadrangle with vertices $gx_0$, $p'_1$, $gsq'_1$, $gsx_0$. We claim that $p_1$ lies within distance $4\delta$ of some points $gsq''_1$, for some $q''_1\in [x_0,hx_0]$. In fact, by $2\delta$-thinness of the quadrangle, we have $d_X(p_1,x)\leq 2\delta$ for some $x$ lying on one of the other 3 sides. If the side is $[gsx_0,gsq'_1]$, we are done. If the side is $[p'_1,gsq'_1]$, then $x$ lies within $2\delta$ of $gsq'_1$, and we can set $q''_1=q'_1$. Finally , if the side is $[gx_0,gsx_0]$, then we can set $q''_1=x_0$: we have $d_X(x,gx_0)\geq k-2\delta$ and $d_X(gx_0,gsx_0)\leq k$, implying $d_X(x,gsx_0)\leq 2\delta$. In all cases, we have $d_X(gsq''_1,p_1)\leq 4\delta$, and $d_X(p_1,p'_2)\geq D$.

Notice that $d_X(q''_1,x_0)\leq 2k+4\delta$, as $d_X(q''_1,x_0)= d_X(gsq''_1,gsx_0)\leq d_X(gsq''_1,p_1)+d_X(p_1,gx_0)+d_X(gx_0,gsx_0)$. In particular, there is $q_1\in[x_0,hx_0]$ so that $d_X(x_0,q_1)\leq 2k$ is an integer and $d_X(gsq_1,p_1)\leq 8\delta+1$.

Finally, we replace $p'_2,q'_2$ with points at a specified distance from $p_1,q_1$. Consider the quadrangle with vertices $p_1,p'_2,q'_2,gsq_1,gsq'_2$. Since $d_X(p_1,q_1),d_X(p'_2,q'_2)\leq 8\delta+1$, any point on $[p_1,p'_1]$ lies within distance $10\delta+1$ of $[gsq_1,gsq'_2]$. In particular, if we let $p_2$ be the point on $[p_1,p'_1]$ at distance $D-50\delta-50$ from $p_1$, we have that there exists $q''_2\in[x_0,hx_0]$ so that $d_X(p_2,gsq''_2)\leq 10\delta+1$. Since $|d_X(q_1,q''_2)-d_X(p_1,p_2)|\leq d_X(p_1,q_1)+d_X(q''_2,p_2)\leq 18\delta+2$, the point $q_2\in[x_0,hx_0]$ at distance $D-50\delta-50$ from $q_1$ (which exists since $d_X(q_1,q''_2)\geq d_X(p_1,p'_2)-18\delta-2$) has the property that $d_X(p_2,gsq_2)\leq 28\delta+3$.

To sum up, we found $p_1,p_2\in [x_0,gx_0]$ and $q_1,q_2\in [x_0,hx_0]$ with the following properties:

\begin{enumerate}
 \item $d_X(p_1,p_2),d_X(q_1,q_2)=D-50\delta-50$,
 \item $d_X(p_i,gsq_i)\leq 28\delta\delta+3$,
 \item $d_X(p_1,gx_0)\leq k$ and $d_X(q_1,x_0)\leq 2k$
 \item $d_X(p_1,gx_0)$ and $d_X(q_1,x_0)$ are integers.
\end{enumerate}

In particular, there are at most $(k+1)(2k+1)$ choices for the quadruple $(p_1,p_2,q_1,q_2)$, by properties 1,3, and 4. Moreover, given a quadruple $(p_1,p_2,q_1,q_2)$ satisfying property 1 (in fact, the inequality ``$\geq$'' suffices), there are at most $N$ elements $s$ satisfying 2. Hence, given $g,h$, there can be only at most $N(k+1)(2k+1)$ elements $s$ so that $diam^X([id,g]\cap N^X_{2\delta}([gs,gsh]))\geq D$, as required.
\end{proof}

The deterministic lemma will be combined with the probabilistic lemma below.

\begin{lm}
\label{littleweight}
For every $L$ there exists $k$ so that for every $A\subseteq G$ of cardinality at most $C_0(Lk)^2$ we have $\matP[Z_k\in A]\leq 1/20$.
\end{lm}

\begin{proof}
 The hypotheses on $\mu$ imply that the group generated by its support contains a non-abelian free group (see \cite[Theorem 1.1]{Os-acyl}) and is therefore non-amenable. Hence, we have, for each $g\in G$, $\matP[Z_k=g]\leq \rho^k$ for some $\rho<1$ \cite{kn:woess}, so the required result follows from summing over $A$ once we choose $k$ so that $C_0(Lk)^2\rho^k\leq 1/20$.
\end{proof}

Let us now fix some constants. Let $L$ be so that $\matP[d_X(id,Z_n)>Ln]\leq 1/20$ for every $n$, which exists because we are assuming that the measures we deal with have exponential tails. Let $k$ be as in Lemma \ref{littleweight} for the given $L$. Finally let $C$ be so that $\matP[d_X(id,Z_k)\geq C]\leq 1/4$ and $C\geq Lk+D+100\delta$, for $\delta$ a hyperbolicity constant for $X$.

We are ready to prove the required inequality, i.e. that for any $h\in G$ we have
$$\matP[(g,gZ_k h)_{id}\leq d_X(id,g)-C]\leq 1/10.\ \ \ \ (*)$$

Fix any $h\in G$. We observe that if $(g,gsh)_{id}\leq d_X(id,g)-C$, for some $s$ with $d_X(id,s)\leq Lk$, then we have $diam(N^X_{2\delta}([x_0,gx_0])\cap [gsx_0,gshx_0])\geq D$.

Hence, letting $A=A(g,h,Lk)$ be as in Lemma \ref{varys} (in particular $\# A\leq C_0(Lk)^2$) we have
$$\matP[(g,gZ_k h)_{id}\leq d_X(id,g)-C]\leq \matP[d_X(id,Z_k)> Lk]+\matP[Z_k\in A].$$
The first term is bounded by $1/20$ by the choice of $L$, while the second one is at most $1/20$ by Lemma \ref{littleweight}, so the claim is proved.

2) Fix any $K\geq 2C$, for $C$ as in the first part of the proposition. The argument below shows $\matE[d_X(id,Z_n)]\geq 4(K-2C)/5$ for each large enough $n$. Let us define $M=\inf\{m; d_X(id,Z_m)\ge K\}$.

The first claim is that $M<\infty$ almost surely.
In fact, there exists $m_0$ such that $\matP[d_X(id, Z_{m_0}) \ge 2K]=\eps>0$. Therefore, eventually one of the increments
$Z_{jm_0}^{-1}Z_{(j+1)m_0}$ will exceed $2K$. If this happens for the first time for a given $j$, then either $d_X(id,Z_{jm_0})\geq K$ or $d_X(id,Z_{(j+1)m_0})\geq K$.

Let $A_{g,m}$ be the event ``$M=m$ and $Z_m=g$''.
Then for any $n\ge m$ the events ``$A_{g,m}$ and $d_X(id,Z_n)\ge  d_X(id,g) -2C$'' and ``$A_{g,m}$ and $d_X(id,g Z_m^{-1}Z_n)\ge  d_X(id,g) -2C$'' coincide.
Notice that $A_{g,m}$ and ``$d_X(id,g Z_m^{-1}Z_n)\ge  d_X(id,g) -2C$'' are independent. Furthermore, if $n\ge m+k$, for $k$ as in the first part of the proposition, then we have
$$\matP[d_X(id,g Z_m^{-1}Z_n)\ge d_X(id,g) -2C]\ge \frac{9}{10}.$$
In fact, $(g,gZ_kh)_{id}\geq d_X(id,g)-C$ can be rewritten as $d_X(id,gZ_kh)-d_X(g,gZ_kh)\geq d_X(id,g)-2C$, so that
$$\matP[d_X(id,g Z_{n-m})\ge d_X(id,g) -2C]\geq \sum_{h\in G}\matP[(g,gZ_kh)_{id}\geq d_X(id,g)-C]\matP[Z_k^{-1}Z_{n-m}=h].$$
By independence, we have
$$\matP[A_{g,m}\textrm{\ and\ } d_X(id,Z_n)\ge d_X(id,g) -2C]\ge \frac{9}{10} \matP[A_{g,m}],$$
which implies $\matE[d_X(id, Z_n)| A_{g,m}]\geq 9(K-2C)/10$ whenever $g$ satisfies $d_X(id,g)\geq K$ (notice that $A_{g,m}=\emptyset$ if $d_X(id,g)< K$).
We can now bound
$$\matE[d_X(id,Z_n)]\ge \sum_{m\le n-k, d_X(id,g)\ge K} \matE[d_X(id, Z_n)| A_{g,m}]\matP[A_{g,m}]$$
$$\ge  \frac{9}{10} (K-2C) \sum_{m\le n-k, d_X(id,g)\ge K} \matP[A_{g,m}]$$
$$=   \frac{9}{10} (K-2C) \matP[M\le n-k].$$
But, since $M<\infty$ almost surely, we have $\matP[M\le n-k]\to 1$ as $n$ tends to $\infty$, and the proof is complete.
\end{proof}

%
%

\emph{Proof of Theorem \ref{linprog} (for $\mu=\mu_0$).}
Throughout the proof we denote by $A=A(g,m)$ the event ``$d_X(id,gZ_m)-d_X(id,g)\geq d_X(id,Z_m)-2C$''. Simple algebraic manipulations show that this is the same event as ``$(g,gZ_m)_{id}\geq d_X(id,g)-C$''.

Let us start with the following claim

{\bf Claim:} There exist $\lambda,\epsilon>0$ and $m$ so that for each $g\in G$ we have
$$\matE\left[e^{-\lambda(d_X(id,gZ_m)-d_X(id,g))}\right]\leq 1-\epsilon.$$

\emph{Proof of Claim.} Fix $k\leq m$. On $A$ we have
$$d_X(id,gZ_m)-d_X(id,g)\geq d_X(id,Z_k^{-1}Z_m)-d_X(id,Z_k)-2C,$$
while on the complement $A^c$ we have
$$d_X(id,gZ_m)-d_X(id,g)\geq -d_X(id,Z_m)\geq -d_X(id,Z_k^{-1}Z_m)-d_X(id,Z_k).$$

So, for any $h\in G$, $m$ and $\lambda>0$, we have
\beqnn 
&&\matE\left[e^{-\lambda(d_X(id,gZ_m)-d_X(id,g))}|Z_k^{-1}Z_m=h\right] \\
& \leq&\matE\left[e^{\lambda2C}e^{\lambda d_X(id,Z_k)}e^{-\lambda d_X(id,h)}\mathbbm{1}_A|Z_k^{-1}Z_m=h\right]
 +\matE\left[e^{\lambda d_X(id,Z_k)}e^{\lambda d_X(id,h)}\mathbbm{1}_{A^c}|Z_k^{-1}Z_m=h\right]\\ 
 &\leq&
 e^{2C\lambda}\Big( e^{-\lambda d_X(id,h)}\matE\left[e^{\lambda d_X(id,Z_k)}|Z_k^{-1}Z_m=h\right]+e^{\lambda d_X(id,h)}\matE\left[e^{\lambda d_X(id,Z_k)}\mathbbm{1}_{A^c}|Z_k^{-1}Z_m=h\right]\\ 
 &&-\,\, e^{-\lambda d_X(id,h)}\matE\left[e^{\lambda d_X(id,Z_k)}\mathbbm{1}_{A^c}|Z_k^{-1}Z_m=h\right]\Big)\\ 
 &= &e^{2C\lambda}\Big( e^{-\lambda d_X(id,h)}\matE\left[e^{\lambda d_X(id,Z_k)}\right]+ \matE\left[e^{\lambda d_X(id,Z_k)}\mathbbm{1}_{A^c}|Z_k^{-1}Z_m=h\right] (e^{\lambda d_X(id,h)}-e^{-\lambda d_X(id,h)})\Big).
\eeqnn

Using Cauchy-Schwarz and Proposition \ref{further}-(1) we get
\beqnn
 &\matE\left[e^{\lambda d_X(id,Z_k)}\mathbbm{1}_{A^c}|Z_k^{-1}Z_m=h\right]\leq&
 \matE\left[e^{2\lambda d_X(id,Z_k)}\right]^{1/2}\ \matP\left[A^c|Z_k^{-1}Z_m=h\right]^{1/2}\\
&&\leq \sqrt{1/10}\ \matE\left[e^{2\lambda d_X(id,Z_k)}\right]^{1/2}.
\eeqnn 
Using this and integrating with respect to $h$ we get
\beqnn 
 \matE\left[e^{-\lambda(d_X(id,gZ_m)-d_X(id,g))}\right]  &\leq&
 e^{2C\lambda}\matE\left[e^{-\lambda d_X(id,Z_{m-k})}\right]\matE\left[e^{\lambda d_X(id,Z_k)}\right]\\ 
 &+ &e^{2C\lambda}\matE\left[e^{\lambda d_X(id,Z_{m-k})}-e^{-\lambda d_X(id,Z_{m-k})}\right]\sqrt{1/10}\ \matE\left[e^{2\lambda d_X(id,Z_k)}\right]^{1/2}:=\phi(\lambda).
\eeqnn
Notice that $\phi$ does not depend on $g$ and $\phi(0)=1$. Also,
$$\phi'(0)=2C-\matE[d_X(id,Z_{m-k})]+\matE[d_X(id,Z_k)]+2\sqrt{1/10}\ \matE[d_X(id,Z_{m-k})].$$
Hence, in view of Proposition \ref{further}-(2), we can choose $m$ so that $\phi'(0)<0$, and the Claim follows.\qed

Let us now fix $\lambda,\epsilon, m$ as in the Claim, and any $g_0\in G$. For a positive integer $j$, write
$$d_X(id,g_0Z_{jm+m})-d_X(id,g_0)=(d_X(id,g_0Z_{jm+m})-d_X(id,g_0Z_{jm}))+d_X(id,g_0Z_{jm})-d_X(id,g_0).$$
By the Claim we have, for each $g\in G$,
$$\matE[e^{-\lambda(d_X(id,g_0Z_{jm+m})-d_X(id,g_0Z_{jm}))}|Z_{jm}=g]\leq (1-\epsilon).$$
So, summing over all possible $g$,
$$\matE[e^{-\lambda (d_X(id,Z_{jm+m})-d_X(id,g_0))}]\leq (1-\epsilon)\matE[e^{-\lambda (d_X(id,g_0Z_{jm})-d_X(id,g_0))}],$$
and inductively we get
$$\matE[e^{-\lambda (d_X(id,g_0Z_{jm})-d_X(id,g_0))}]\leq (1-\epsilon)^j.$$
Using Markov's inequality, for any $c>0$ we can make the estimate
$$\matP[d_X(id,g_0Z_{jm})-d_X(id,g_0)<cjm]=\matP[e^{-\lambda (d_X(id,Z_{jm})-d_X(id,g_0))}>e^{-\lambda cjm}]\leq e^{\lambda cjm}(1-\epsilon)^j.$$

Choosing $c$ small enough, we see that there exists $C_0\geq 1$ so that
$$\matP[d_X(id,g_0Z_{jm})-d_X(id,g_0)<jm/C_0]\leq e^{-jm/C_0}.\ \ \ (*)$$

If $n$ is now any positive integer, we can write $n=jm+r$, with $0\leq r<m$.

Since $d_X(id,g_0Z_n)-d_X(id,g_0)\geq d_X(id,g_0Z_{jm})-d_X(g_0Z_{jm},g_0Z_n)-d_X(id,g_0)$, we can make the estimate
\begin{multline*}
 \matP[d_X(id,g_0Z_{n})-d_X(id,g_0)< n/(2C_0)]\leq\\
 \matP[d_X(id,g_0Z_{jm})-d_X(id,g_0)<jm/C_0]+ \max_{i=0,\dots,m-1}\matP[d_X(id,Z_{i})\geq (jm-i)/(2C_0)].
\end{multline*}
The first term decays exponentially in $j$, whence in $n$, because of $(*)$, while the exponential decay of the second term follows from the exponential tail of $\mu_0$.

This concludes the proof.
\qed

\section{Deviation from quasi-geodesics} \label{sec:deviationqg} 

\begin{df}
\label{def:acylinter}
 Let $G$ be a finitely generated group acting acylindrically on the geodesic hyperbolic space $X$. The geodesic metric space $Y$ endowed with an isometric group action of $G$ is \emph{acylindrically intermediate for} $(G,X)$ if there exists a coarsely Lipschitz $G$-equivariant map $\pi:Y\to X$ so that Lemma \ref{geomsep} holds for $\pi$, namely there exist $x_0\in X$, $L\geq 0$ and a non-decreasing function $f$ so that for each $l_1,l_2\geq 0$, each $t\geq 0$ and whenever $x,y\in Gx_0$ satisfy $d_X(x,y)\geq L+l_1+l_2$, we have
 $$diam^Y(\ \pi^{-1}(B^X(x,l_1))\cap N^Y_t (\,\pi^{-1}(B^X(y,l_2))\,)\ )\leq f(t),$$
where $diam^Y$ and $N^Y_t$ denote the diameter and neighborhood taken with respect to the metric of $Y$.
 \end{df}
 
 \begin{rmk}
 Since $\pi$ is coarsely Lipschitz, once we fix $x_0\in X$ and $y_0\in Y$ with $\pi(y_0)=x_0$, we can rescale the metric of $X$ to ensure that $d_X(gx_0,hx_0)\leq d_Y(gy_0,hy_0)$ for each $g,h\in G$. This will be convenient a few times below.
 \end{rmk}

\begin{prop}\label{prop:examplesacylinter}
 The following are examples of groups $G$ and metric spaces $X,Y$ so that $Y$ is acylindrically intermediate for $(G,X)$.
 \begin{enumerate}
  \item If $G$ is a finitely generated group acting acylindrically on the geodesic hyperbolic space $X$ then $Y=X$ and $Y=Cay(G,S)$, for $S$ a finite generating set of $G$, are both acylindrically intermediate for $(G,X)$.
  \item If $G$ is relatively hyperbolic, $X$ is its coned-off graph and $Y$ is its Bowditch space\footnote{By Bowditch space we mean the space obtained attaching combinatorial horoballs to parabolic subgroups as in \cite{Bow-99-rel-hyp, GrMa-perfill}}, then $Y$ is acylindrically intermediate for $(G,X)$. Similarly, if $G$ is the fundamental group of a finite-volume hyperbolic $n$-manifold, then we can take $Y=\mathbb H^n$.
  \item If $G$ is the mapping class group of a connected oriented hyperbolic surface $S$ of finite type, $X$ is the curve complex of $S$ and $Y$ is Teichm\"uller space of the surface endowed with either the Teichm\"uller or the Weil-Petersson metric, then $Y$ is acylindrically intermediate for $(G,X)$.  
 \end{enumerate}
\end{prop}

\begin{proof}
 Item 1) For $Y=X$, we can take $\pi$ to be the identity map, $L=10\delta$, where $\delta\geq 1$ is so that $X$ is $\delta$-hyperbolic, and $f(t)=2t$. Fix $l_1,l_2,t,x,y$ as in Definition \ref{def:acylinter}.
 
Fix a geodesic $[x,y]$ in $X$ from $x$ to $y$ and let $p\in [x,y]$ be so that $d_X(x,p)=l_1+5\delta$, and hence $d_X(y,p)\geq l_2+5\delta$. Notice that  $N^X_t (B^X(y,l_2))=B^X(y,l_2+t)$. We will show the inclusion $B^X(x,l_1)\cap B^X(y,l_2+t)\subseteq B^X(p,t)$, which implies $diam^X(B^X(x,l_1)\cap B^X(y,l_2+t))\leq 2t$, as required.
 
 Let $w\in B^X(x,l_1)\cap B^X(y,l_2+t)$, that is, $d_X(x,w)\leq l_1$ and $d_X(y,w)\leq l_2+t$. Consider geodesics $[x,w],[w,y]$. By hyperbolicity, $p$ lies within $\delta$ of some $q\in [x,w]\cup [w,y]$. However, if we had $q\in [x,w]$ then we would have $d_X(x,p)\leq d_X(x,q)+\delta\leq d_X(x,w)+\delta\leq l_1+\delta$, contradicting the defining property of $p$. Hence, we have $q\in [w,y]$, and moreover $d_X(y,q)\geq d_X(y,p)-\delta\geq l_2+4\delta$. Since $q$ lies on a geodesic from $w$ to $y$ we have $d_X(q,w)=d_X(w,y)-d_X(y,q)$, and therefore
 $$d_X(p,w)\leq (d_X(w,y)-d_X(y,q))+\delta\leq (l_2+t-(l_2+4\delta))+\delta\leq t,$$
 as we wanted.
 
 For $Y=Cay(G,S)$, the claim follows from Lemma \ref{geomsep}.
 
 2) The natural map $\pi$ from $Y$ to $X$ maps geodesics within bounded Hausdorff distance of quasi-geodesics, see e.g \cite[Lemma 7.3]{Bow-99-rel-hyp}. In particular, since $Y$ is hyperbolic, there exists $C$ so that for any ball $B$ in $X$ there exists a $C$-quasiconvex set $Q(B)$ in $Y$ (the union of all geodesics connecting points in $\pi^{-1}(B)$) so that $\pi^{-1}(B)\subseteq Q(B)\subseteq \pi^{-1}(N^X_C(B))$. If $B_1$ and $B_2$ are far away balls in $X$, then $Q(B_1)$ and $Q(B_2)$ are far away in $Y$, and hence in this case $Q(B_1)\cap N^Y_t(Q(B_2))$ can be bounded in terms of $t,C$ and the hyperbolicity constant of $Y$.
 
 If $G$ is the fundamental group of a finite-volume hyperbolic $n$-manifold, then $\mathbb H^n$ is equivariantly quasi-isometric to the Bowditch space for $\pi_1(M)$ (with respect to the standard relatively hyperbolic structure).
 
 3) The acylindricity of the action of $G$ on $X$ was proven in \cite{Bow-acylcc}. The statement about Teichm\"uller space is a consequence of the distance formula and related machinery for the Weil-Petersson \cite{MM1,MM2,Br-pants} and Teichm\"uller metric \cite{Rafi:combinatorial_teichmuller,Durham:augmented}. The proofs for the two metrics are pretty much identical. The proof below gives a rather rough bound, but we tried to keep it simple as possible.
 
 For every (isotopy class of essential) subsurface $Z$ of $S$, there is an associated hyperbolic metric space $C_Z$ and a map $\pi_Z:Y\to C_Z$. (More specifically, if $Z$ not an annulus, then $C_Z$ is the curve complex of $Z$. If $Z$ is an annulus and we are considering the Teichm\"uller metric, $C_Z$ is quasi-isometric to a horoball in $\mathbb H^2$, while if we are considering the Weil-Petersson metric we can take $C_Z$ to be a point.)
 
 The distance formula says the following. For $A,R$ real numbers, denote $[A]_L=A$ if $A\geq L$ and $0$ otherwise. Also, write $A\approx_K B$ if $A/K -K \leq B \leq KA+K$, and similarly for $\lesssim$. Then for each large enough $L$ there exists $K$ so that for each $x,y\in Y$ we have
 $$d_Y(x,y)\approx_K \sum_Z [d_{C_Z}(\pi_Z(x),\pi_Z(y)]_L,$$
 where the sum is taken over all (isotopy classes of essential) subsurfaces $Z$ of $S$.
 
 We will also need the consequence of the bounded geodesic image theorem that says that there exists $C$ so that, for $x,y,x',y'\in Y$, if the geodesics from $\pi_S(x')$ to $\pi_S(y')$ are further than $C$ from geodesics from $\pi_S(x)$ to $\pi_S(y)$, then whenever $Z$ is a subsurface for which $d_{C_Z}(\pi_Z(x),\pi_Z(y))\geq C$, we have $d_{C_Z}(\pi_Z(x'),\pi_Z(y'))\leq C$.
 
 Now, suppose that the balls $B_1,B_2$ in $X=C_S$ are far enough apart and fix $t\geq 0$. Let $x,y\in \pi_S^{-1}(B_1)$ and let $x',y'\in \pi_S^{-1}(B_2)$ with $d_Y(x,x'),d_Y(y,y')\leq t$ (that is, $x,y\in \pi_S^{-1}(B_1)\cap N^Y_t(\pi_S^{-1}(B_2))$). We wish to bound $d_Y(x,y)$ in terms of $t$. Let $L>C$ be large enough. In the estimate below we write $\approx$ instead of $\approx_K$ and it is understood that $K$ depends only on $L$:
 $$d_Y(x,y)\approx \sum_Z [d_{C_Z}(\pi_Z(x),\pi_Z(y)) ]_{3L} \lesssim \sum_Z [d_{C_Z}(\pi_Z(x),\pi_Z(x'))+C+d_{C_Z}(\pi_Z(y'),\pi_Z(y))]_{3L} $$
 $$\lesssim 3\left( \sum_Z [d_{C_Z}(\pi_Z(x),\pi_Z(x'))]_L+\sum_Z [d_{C_Z}(\pi_Z(y'),\pi_Z(y))]_L \right) \lesssim t,$$
 as required. The third inequality follows from the fact that if $d_{C_Z}(\pi_Z(x),\pi_Z(x'))+C+d_{C_Z}(\pi_Z(y'),\pi_Z(y))\geq 3L$ then $\max\{d_{C_Z}(\pi_Z(x),\pi_Z(x')),d_{C_Z}(\pi_Z(y'),\pi_Z(y))\}\geq L$.
\end{proof}

\subsection{Superlinear divergence}\label{ssec:superlineardiv} 

{\bf Convention.} To save notation, when the group $G$ acts on $X$ and $Y$ is acylindrically intermediate for $(G,X)$, we automatically fix basepoints $x_0\in X$, $y_0\in Y$ so that $x_0=\pi(y_0)$, where $\pi$ is as in Definition \ref{def:acylinter}. Also, we identify $G$ with the orbit in $Y$ of $y_0$, and, moreover, for $g,h\in G$ we set $d_X(g,h)=d_X(gx_0,hx_0)$. Furthermore, we rescale the metric of $X$ to ensure that $d_X(g,h)\leq d_Y(g,h)$ for each $g,h\in G$. Finally, recall that we often omit the adjective ``discrete'' when discussing discrete paths.

\begin{prop}
\label{superlinproj}
Let $G$ act acylindrically on the hyperbolic space $X$ and let $Y$ be acylindrically intermediate for $(G,X)$. Then for any $L$ there exists a constant $C$ and a diverging function $\rho:\R^+\to \R^+$ so that the following holds. Let $\alpha_1,\alpha_2$ be $L$-Lipschitz (discrete) paths with respect to the metric of $Y$, where $\alpha_i$ connects $g_i$ to $h_i$.
Then 
$$\max\{l_Y(\alpha_1),l_Y(\alpha_2)\}\geq (d_X(g_1,h_1)-d_X(g_1,g_2)-d_X(h_1,h_2)-C)\cdot \rho(d_Y(\alpha_1,\alpha_2)).$$
\end{prop}

\begin{proof}
Suppose that $X$ is $\delta$-hyperbolic, with $\delta\geq 1$. We denote by $C_i$ suitable constants depending on $G,X,Y,L,\delta$ only, and for convenience we take $C_{i+1}\geq C_{i}$.
 
 The first lemma ensures that we can assume that the paths $\alpha_i$ stay close to $d_X$-geodesics, for otherwise they would be long due to the geometry of $X$.
 
 \begin{lm}\label{avoidballs}
 Let $\alpha$ be a discrete path in $G$ which is $L$-Lipschitz for the metric of $Y$ and which connects $g$ to $h$. Consider $n$ points $p_i$ in $X$ within distance $2\delta$ from a geodesic $\gamma$ from $gx_0$ to $hx_0$.
 
 Suppose that, for some $r\geq C_0$, we have that $d(p_i,p_{i+1})\geq 3r+100\delta$ and $d_X(\alpha,p_i)\geq r$ for each $i$. Then $l_Y(\alpha)\geq n r^2/C_0$.
 \end{lm}
 
 We could replace $r^2$ with an exponential function (with a different proof, cfr. \cite[Lemma 2.6, Claim 2]{HS-coarse-emb}), but we do not need this fact.
 
\begin{proof}
Let $\beta$ be the path obtained projecting $\alpha$ to $X$, which is again an $L$-Lipschitz path (because of how we rescaled the metric of $X$).

Let $\mathfrak p:X\to\gamma$ be a closest-point projection map, meaning that $d_X(x,\mathfrak p(x))=d_X(x,\gamma)$ for each $x\in X$. We will make use of the following well-known properties of closest-point projections in hyperbolic spaces:

\begin{itemize}
 \item For each $x,y\in X$ we have $d_X(\mathfrak p(x),\mathfrak p(y))\leq d_X(x,y)+20\delta$.
 \item For each $x\in X$, setting $r_x=d(x,\gamma)$, we have $diam^X(\mathfrak p(B^X(x,r_x)))\leq 20\delta$.
\end{itemize}

Let $p'_i=\mathfrak p(p_i)$. We now consider the map $\mathfrak p$ along $\beta$, and observe that $\beta$ has disjoint subpaths that each project on a subgeodesic of size roughly $r$ around $p'_i$, provided that $r$ is sufficiently large. More precisely, if $r\geq 10^6(L+\delta)$, then $\beta$ has disjoint subpaths $\beta_i$, with endpoints $x_i,y_i$ with the following properties:
\begin{itemize}
 \item $d_X(\mathfrak p(x_i),\mathfrak p(y_i))\geq r/2$,
 \item for every $x$ on $\beta_i$ we have $d_X(\mathfrak p(x),p'_i)\leq r/2$.
\end{itemize}

Notice that the second property implies that $d_X(x,\gamma)\geq r/3$, for otherwise we would have
$$d_X(x,p_i)\leq d_X(x,\mathfrak p(x))+d_X(\mathfrak p(x),p'_i)+d_X(p'_i,p_i)\leq r/3+r/2+2\delta<r.$$
In particular, any subpath of $\beta_i$ of length at most $r/3$ projects to a set of diameter at $20\delta$, by the second property of closest-point projections. We can subdivide $\beta_i$ into at most $\frac{l(\beta_i)}{r/4}$ paths of length $\leq r/3$ as follows. Consider the longest initial subpath $\beta^1_i$ of $\beta_i$ of length $\leq r/3$. In fact, such path has length $\geq r/3-L$ since $\beta$ is $L$-Lipschitz, and since $d_X(\mathfrak p(x_i),\mathfrak p(y_i))>20\delta$, which rules out $\beta^1_i=\beta_i$. Next, consider the longest subpath $\beta^2_i$ of $\beta_i$ of length $\leq r/3$ that starts at the final point of $\beta^1_i$.
Proceeding inductively, one constructs $m-1$ paths $\beta^k_i$ of length between $r/3-L$ and $r/3$, and a final one $\beta^m_i$ of length $\leq r/3$. We have
$$l(\beta_i)=\sum_{k=1,\dots,m-1}l(\beta^k_i)+l(\beta^m_i)\geq (m-1)(r/3-L).$$

Since $d_X(\mathfrak p(x_i),\mathfrak p(y_i))\geq r/2$, and each $\beta^k_i$ has projection of diameter at most $20\delta$, we have $m\geq \frac{r/2}{20\delta}$, and hence
$$l(\beta_i)\geq \left( \frac{r}{40\delta}-1\right)(r/3-L)\geq r^2/(200\delta).$$
Since the $\beta_i$ are disjoint, we have $l(\beta)\geq \sum l(\beta_i)$, and we are done.
%
\end{proof}

 \begin{lm}\label{lm:farpairs}
 There exists a diverging non-decreasing function $\rho_0:\R^+\to \R^+$ with the following properties. Let $a_i,b_i\in G$, for $i=1,2$ be so that
 \begin{itemize}
  \item $d_X(a_i,b_i) \geq d_X(a_1,a_2)+d_X(b_1,b_2) +C_1$.
 \end{itemize}
Then, denoting $s=\min\{d_Y(a_1,a_2),d_Y(b_1,b_2)\}$, we have
$$\max\{d_Y(a_i,b_i)\}\geq \rho_0(s).$$
\end{lm}

\begin{proof}
 Recall that we denote by $diam^{\ast}$ the diameter, by $B^{\ast}(\cdot, R)$ a ball of radius $R$ and by $N^{\ast}_t$ a neighborhood of radius $t$, where $\ast$ can be either $Y$ or $X$ depending on which metric we are using to define the given notion. Recall that we are assuming the following: There exist $C_1$ and a non-decreasing function $f$ so that for each $t$ and whenever $g,h\in G$ satisfy $d_X(g,h)\geq r_1+r_2+C_1$, we have $diam^Y(B^X(g,r_1)\cap N^Y_t (B^X(h,r_2)))\leq f(t)$.
 
 Let $\rho_0$ be a non-decreasing diverging function so that $f(\rho_0(t))<t$ for each $t$.
 
 If we had $d_Y(a_i,b_i)< \rho_0(s)$ for $i=1,2$, then we would have
 $$f(\rho_0(s))\geq diam^Y(B^X(a_1,d_X(a_1,a_2))\cap N^Y_{\rho_0(s)} (B^X(b_1,d_X(b_1,b_2))))\geq d_Y(a_1,a_2)\geq s> f(\rho_0(s)),$$
 a contradiction.
\end{proof}

Let $r:\R^+\to \R^+$ be a diverging function so that $\rho_0(t)/r(t)\to\infty$ as $t\to\infty$.
 
Let $K= d_X(g_1,h_1)-d_X(g_1,g_2)-d_X(h_1,h_2)-C_1$. We can assume $K\geq 0$. We let $r=r(d_Y(\alpha_1,\alpha_2))$. We can and will assume $r> 2C_0$.

If we have $K\leq 6r+2C_1$ then we can make the following estimate. By Lemma \ref{lm:farpairs} above, there exists $i$ so that $d_Y(g_i,h_i)\geq \rho_0(d_Y(\alpha_1,\alpha_2))$. Hence,
$$l_Y(\alpha_i)\geq  K \frac{\rho_0(d_Y(\alpha_1,\alpha_2))}{6r+2C_1},$$
and we are done.

Suppose now $K\geq 6r+2C_1$. Consider any geodesic $\gamma_i$ in $X$ from $g_ix_0$ to $h_ix_0$ and fix a sequence of points $\{p_k\}_{k=1,\dots,n}$ appearing in the given order along $\gamma_1$ so that
\begin{enumerate}
 \item $n\geq K/(6r+2C_1)$,
 \item $d_X(p_k,p_{k+1})\geq 3r+C_1$,
 \item $d_X(p_k,\gamma_2)\leq 2\delta$.
\end{enumerate}
 
 If for some $i$, $\alpha_i$ avoids $B^X_{r/2}(p_k)$ for at least $n/10$ values of $k$, then, in view of Lemma \ref{avoidballs}, we have
 $$l_Y(\alpha_i)\geq \frac{nr^2}{40C_0} \geq K \frac{r^2}{40C_0(6r+2C_1)},$$
 and we are done.
 
 Otherwise, there are more than $4n/5$ values of $k$ so that $\alpha_i$ contains a point $q^i_k$ in $B^X_{r/2}(p_k)$ for $i=1,2$. We also assume that the points $q^i_k$ appear in the given order along $\alpha_i$ and we set $q^i_0=g_i$, $q^i_{n+1}=h_i$. By Lemma \ref{lm:farpairs} above, there exists $i$ so that for more than $2n/5$ values of $k$ we have $d_Y(q^i_k,q^i_{k+1})\geq \rho_0(d_Y(\alpha_1,\alpha_2))$. But then
$$l_Y(\alpha_i)\geq \sum d_Y(q^i_k,q^i_{k+1})\geq \frac{2n}{5} \rho_0(d_Y(\alpha_1,\alpha_2))\geq K \frac{2\rho_0(d_Y(\alpha_1,\alpha_2))}{5(6r+2C_1)},$$
as required.
\end{proof}

\subsection{Main argument}

\begin{theo}
\label{smalldevhier}
Let $G$ be a finitely generated group acting acylindrically on the geodesic hyperbolic space $X$, and let $Y$ be acylindrically intermediate for $(G,X)$. Let $\mu_0$ be a measure on $G$ with exponential tail whose support generates a non-elementary semigroup. Then $\mu_0$ has a neighborhood $\mathcal N$ so that for every $D$ there exists $C$ with the following property. 
For each $\mu\in \mathcal N$, $l,n\geq 1$ and $k<n$ we have
 $$\matP^\mu\left[\sup_{\alpha\in QG_D(id,Z_n)} d_Y(Z_k,\alpha)\geq l\right]\leq Ce^{-l/C},$$
 where $QG_D(a,b)$ denotes the set of all $(D,D)$-quasi-geodesics (with respect to $d_Y$) from $a$ to $b$.
\end{theo}

First of all, we point out a corollary of the theorem. Given a metric space $X$, a $K$-{\bf quasi-ruler} in $X$ is a map $\gamma:I\to X$, where $I\subseteq \R$ is an interval, so that for all $s\leq t\leq u$ in $I$ the Gromov product satisfies $(\gamma(s),\gamma(u))_{\gamma(t)}\leq K$. 

Following \cite{kn:bhm2}, we will say that a metric is {\bf quasi-ruled} if there exists $K$ so that any two points can be joined by a $(K,K)$-quasi-geodesic $K$-quasi-ruler.

\begin{cor}\label{cor:main} 
 Under the hypotheses of Theorem \ref{smalldevhier}, let $d$ be a quasi-ruled metric on $G$ (e.g. a geodesic metric) quasi-isometric to $d_Y$. Then $\mu_0$ satisfies the locally uniform exponential-tail deviation inequality with respect to $d$.
\end{cor}

\begin{proof}
 Let $D$ be so that $d$ is $(D,D)$-quasi-isometric to $d_Y$, and any two points $x,y$ of $Y$ can be joined by a $(D,D)$-quasi-geodesic $D$-quasi-ruler $r(x,y)$. Then for each $x,y\in G$, we have $(id,y)_x\leq D d_Y(x,r(id,y))+2D$ (the Gromov product is measured with respect to $d$). In fact, for any $p\in r(id,y)$ so that $d_Y(x,p)=d_Y(x,r(id,y))$, we have
$$(id,y)_x\leq (id,y)_p+ d(x,p)\leq D+(D d_Y(x,r(id,y))+D).$$
Hence, the following holds for $C$ as in Theorem \ref{smalldevhier}. For all $n\geq k\geq 1$, $\mu\in\mathcal N$ and $l>3D$ we have
$$\matP^{\mu}[(id,Z_n)_{Z_k}\geq l]\leq \matP^{\mu}\left[d_Y(Z_k,r(id,Z_n))\geq \frac{l}{D}-2\right] \leq Ce^{-l/(CD)-2/C},$$
as required.
\end{proof}

Fix the notation of Theorem \ref{smalldevhier}. Recall that we rescaled the metric of $X$ to ensure that $d_X(g,h)\leq d_Y(g,h)$ for all $g,h\in G$. When we write an inequality involving $\matP$ without explicit reference to the measure we mean that the statement holds for every $\mu\in \mathcal N$ and that the constants involved can be chosen uniformly for all $\mu\in\mathcal N$, where $\mathcal N$ is a small enough neighborhood of $\mu_0$. Up to increasing $D$, we can replace $QG_D(\cdot,\cdot)$ in the statement with the family $QG'(\cdot,\cdot)$ of discrete $D$-Lipschitz $(D,D)$-quasi-geodesics with given endpoints. Also, we will denote by $\gamma(g,h)$ any element of $QG'(g,h)$. In particular:

\begin{rmk}
\label{lvsd}
 $l_Y(\gamma(g,h))\leq D^2 d_Y(g,h)+D^3$ for each $g,h\in G$.
\end{rmk}

\emph{Proof of Theorem \ref{smalldevhier}.}
 We denote by $C_i\geq 1$ suitable constants that do not depend on $k,n$.

The fact that $\mu_0$ has exponential tail implies that
$$\matP[l_Y((Z_i)_{i\leq n})\geq C_0n]\leq C_0 e^{-n/C_0} \ \ \ (*)$$
for a suitable $C_0$.

Recall that the following holds.

\begin{theo}(Theorem \ref{linprog})
\label{linprogpreview}
 $(Z_n)$ makes linear progress with exponential decay in the $d_X$-metric, i.e. we have
$$\matP[d_X(id,Z_n)< n/C_1]\leq C_1 e^{-n/C_1}.$$
\end{theo}

We say that a path $(w_i)_{i\leq n}$ is \emph{tight around $w_k$ at scale $l$} if it satisfies the following conditions for any $k_1\leq k\leq k_2$ with $k_2-k\geq l$.

\begin{enumerate}
 \item $d_X(w_{k_1},w_{k_2})\geq (k_2-k_1)/C_1$,
 \item $l_Y((w_i)_{k_1\leq i\leq k_2})\leq C_0(k_2-k_1)$,
 \item $d_Y(w_{k'},w_{k'+1})\leq \max\{l,|k-k'|/(100C_1)\}$ for each $k'$.
\end{enumerate}

The third item says that the geodesic connecting the endpoints of the jump at step $k'$ has length at most $l$ if $|k'-k|$ is small and at most $|k'-k|/(100C_1)$ if $|k'-k|$ is large (recall that $k$ is fixed and $l$ is a parameter).


\begin{lm}\label{tightpath}
 There exists $C_5$ so that for all $n$, all $k\leq n$, and all $l\geq 1$ we have 
 $$\P[(Z_i)_{i\leq n} \hbox{\, is tight around $Z_k$ at scale $l$}]\geq 1-C_5 e^{-l/C_5}\,.$$ 
\end{lm}

\begin{proof}
The probability that $1)$ does not hold for given $k_1,k_2$ can be estimated using Theorem \ref{linprogpreview}. In fact, we have 
$$\matP[d_X(Z_{k_1},Z_{k_2})< (k_2-k_1)/C_1]=$$
$$\matP[d_X(id,Z_{k_2-k_1})< (k_2-k_1)/C_1]\leq C_1 e^{-(k_2-k_1)/C_1}$$
since the law of $Z_{k_1}^{-1}Z_{k_2}$ is the same as the law of $Z_{k_2-k_1}$.

So, for a given $k_1$ we get
$$\matP[\exists k_2\geq k: k_2-k_1\geq l, d_X(Z_{k_1},Z_{k_2})< (k_2-k_1)/C_1]\leq$$
$$\sum_{k_2-k_1=j\geq \max\{l,k-k_1\}} C_1e^{-j/C_1}\leq C_3 e^{-\max\{l,k-k_1\}/C_1}.$$

Summing again over all possible $k_1$ we get:
$$\matP[\exists k_1\leq k\leq k_2: k_2-k_1\geq l, d_X(Z_{k_1},Z_{k_2})< (k_2-k_1)/C_1]\leq$$
$$\sum_{\substack{k_1\leq k\\ k-k_1\leq l}} C_3 e^{-l/C_1}+\sum_{k-k_1=j> l}C_3e^{-j/C_1}\leq$$
$$C_3l e^{-l/C_1}+C_4e^{-l/C_1}\leq C_5 e^{-l/C_5},$$
what we wanted.

Items $2)$ and $3)$ can be obtained using the same summing procedure as item $1)$, we will not spell out the details. In the case of item $2)$ one uses $(*)$, while in the case of item $3)$ one uses that $\matP[d_Y(id,Z_1)\geq l]$ decays exponentially in $l$.
\end{proof}

We now reduced the proof of Theorem \ref{smalldevhier} to the following entirely geometric lemma.

\begin{lm}
\label{tightrack}
 Let $(w_i)_{0\leq i\leq n}$ be tight around $w_k$ at scale $l$. Then $d_Y(w_k,\gamma(w_0,w_n))\leq C_7 l$.
\end{lm}

\begin{proof}
For convenience, set $\gamma=\gamma(w_0,w_n)$. Recall that we have $d_X(g,h)\leq d_Y(g,h)$ for all $g,h\in G$.

We now choose some constants. Let $\rho$ be as in Proposition \ref{superlinproj}, with $L=D$, and fix $C_6$ so that $\rho(t)>2C_0C_1$ for each $t\geq C_6$. Up to increasing $C_6$, we can also require $C_6\geq C$, where $C$ is as in Proposition \ref{superlinproj}.

Suppose $d_Y(w_k,\gamma)\geq C_6l$, for otherwise we are done. Let $k_1< k$ be maximal (resp. $k_2\geq k$ be minimal) so that a geodesic $[w_{k_1-1},w_{k_1}]$ (resp. $[w_{k_2},w_{k_2+1}]$) intersects the neighborhood $N^Y_{C_6l}(\gamma)$ in $Y$.

Let $\alpha$ be the concatenation of geodesics $[w_i,w_{i+1}]$ for $k_1\leq i\leq k_2-1$. In particular, $d_Y(\alpha,\gamma)\geq C_6l$, and, denoting $[w_{k_i},w_{k_i\pm1}]$ any geodesic in $Y$ from $w_{k_i}$ to $w_{k_i\pm1}$, we have 
\begin{align*}
 d_Y(w_{k_i},\gamma) & \leq   d_Y(w_{k_i},w_{k_i\pm1})+d_Y([w_{k_i},w_{k_i\pm1}],\gamma) \\
 & \leq\max\{l, (k_2-k_1)/(100C_1)\}+C_6l.
\end{align*}
Also, by property $2)$ from the definition of tightness, we have $l_Y(\alpha)\leq C_0(k_2-k_1)$. 

We analyse 2 cases, with the aim of showing that only the first one can hold.

The first case is if $k_2-k_1\leq 100C_6C_1l$. Then
$$d_Y(w_{k},\gamma)\leq d_Y(w_k,w_{k_1})+d_Y(w_{k_1},\gamma)\leq C_0(k_2-k_1)+C_6l+\max\{l, (k_2-k_1)/(100C_1)\},$$
which is bounded linearly in $l$.

The second case is if $k_2-k_1\geq 100 C_6C_1 l$. (Recall that we have to show that this does not happen.) In this case $\max\{l, (k_2-k_1)/(100C_1)\}=(k_2-k_1)/(100C_1)$. Let $x_{k_i}\in \gamma$ be so that $d_Y(x_{k_i},w_{k_i})\leq C_6l+ (k_2-k_1)/(100C_1)$, and let $\gamma'$ be the subpath of $\gamma$ connecting $x_{k_1}$ to $x_{k_2}$. We remark that
$$d_X(x_{k_i},w_{k_i})\leq d_Y(x_{k_i},w_{k_i})\leq C_6l+ \frac{k_2-k_1}{100C_1}\leq \frac{k_2-k_1}{4C_1}-C.$$
Hence, we have
\begin{align*}
 K= d_X(w_{k_1},w_{k_2})-d_X(x_{k_1},w_{k_1})-d_X(x_{k_2},w_{k_2})  -C\geq (k_2-k_1)/(2C_1).
\end{align*}
We also have
$$d_Y(x_{k_1},x_{k_2})\leq d_Y(w_{k_1},w_{k_2})+d_Y(x_{k_1},w_{k_1})+d_Y(x_{k_2},w_{k_2})\leq C_0(k_2-k_1) + \frac{k_2-k_1}{2C_1}-2C \leq 2C_0(k_2-k_1),$$
so that
$$l_Y(\gamma')\leq 2D^2 C_0(k_2-k_1)+D^3 < \frac{k_2-k_1}{2C_1} \rho(C_6l)\leq K \rho(C_6l).$$ 
Hence, by Proposition \ref{superlinproj} we have
$$k_2-k_1\geq \frac{l_Y(\alpha)}{C_0}\geq \frac{K\rho(C_6l)}{C_0}\geq \frac{k_2-k_1}{2C_0C_1} \rho(C_6l) >k_2-k_1,$$
a contradiction. So the second case cannot hold and the proof is complete.
\end{proof}

\section{Deviation inequalities in hyperbolic groups} \label{sec:hypgroups} 

In this section we study deviation inequalities in hyperbolic groups under very general conditions on driving measures. Perhaps surprisingly at first, we will see that driving measures $\mu$  with finite second moment satisfy the $p$-th moment deviation inequality for each $p<4$. As we will see in the proofs, this follows from the fact that, roughly speaking, the way that a sample path can deviate from a geodesic is by doing two large (almost) consecutive steps. Hence, once again roughly speaking, the probability that a point on a sample path is far from a geodesic connecting the endpoints of the path is comparable to that of having two large consecutive jumps.

Lemma \ref{bowtie} below is where we exploit the geometry of hyperbolic groups. In the given form, the lemma does not hold for acylindrically hyperbolic groups, and it is unclear whether a useful version of the lemma exists in that context. The proof is not the most efficient one for the case of measures with exponential tails, meaning that it does not give the exponential deviation inequality as in Theorem \ref{smalldevhier}.

\begin{theo}\label{devhyp}
 Let $G$ be a hyperbolic group endowed with the word metric $d_G$ and let $\mu_0$ be a measure with finite first moment on $G$ whose support generates a non-elementary semigroup of $G$. Then there exists a neighborhood $\mathcal N$ of $\mu_0$ and a constant $C\geq 1$ so that for every $\mu\in\mathcal N$ the random walk $(Z_n)$ with driving measure $\mu$ satisfies the following. For each $t\geq 0$, each $M\geq 1$ and each positive integers $k\leq n$ we have
 $$\matP^\mu\left[\sup_{[id,Z_n]} d_G(Z_k,[id,Z_n])\geq t\right]\leq Ce^{-M/C}+ M^2(\matP^\mu[d(id,X_1)\geq (t- C)/M])^2.$$
 Moreover, if $\mu$ has finite second moment then $\mu$ satisfies the $p$-th moment deviation inequality for each $p<4$.
\end{theo}

Fix the notation of the theorem from now on. The constants $C_i$ appearing below depend on the data of the theorem and are all uniform in a sufficiently small neighborhood of $\mu_0$.

\begin{lm}\label{doublecoset}
 There exists $C_0$ with the following property. For each $n\geq 1$ and each $g,h\in G$ we have
 $$\matP^\mu[d_G(id, g Z_{n}h)\leq n/C_0]\leq C_0 e^{-n/C_0}.$$
\end{lm}

\begin{proof}
First of all, we claim that there exists $K$ so that for each $a\in G$ we have $\matP^\mu[Z_n=a]\leq Ke^{-n/K}$.

Since non-elementary subgroups of $G$ are non-amenable, there exists $K'$ so that for each $a\in G$ we have $\matP^{\mu_0}[Z_n=a]\leq K'e^{-n/K'}$ \cite{kn:woess}. In order to pass from $\mu_0$ to $\mu$, we use an argument very similar to the proof of Lemma \ref{lem:uniform_in_mu}.

 Let $\mu$ be so that $\Nu(\mu,\mu_0)\leq\epsilon$, where $\epsilon$ will be chosen later.

Since $\mu(x)/\mu_0(x)\le 1+\epsilon$ for each $x\in G$, using the Girsanov formula we have
$$\esp^{\mu}[F(X_1,\dots,X_n)]\le (1+\epsilon)^n \esp^{\mu_0}[F(X_1,\dots,X_n)]$$
for any non-negative measurable function $F:G^n\to\mathbb R_+$.
Using this inequality with $F=\mathbbm{1}_A$ where $A$ is the event ``$Z_n=a$'' yields:
$$\matP^{\mu}[Z_n=a]\le K'(1+\epsilon)^n e^{-n/K'}.$$
It is then enough to choose $\epsilon$ small enough so that $K'(1+\epsilon)^n e^{-n/K'}$ decays exponentially, and the proof of the claim is complete.

 To conclude the proof of the lemma, consider the elements $\{a_i\}$ in the ball of radius $n/C_0$ in $G$. Then
 $$\matP^\mu[d_G(id, g Z_{n}h)\leq n/C_0]=\sum_i \matP^\mu[Z_n=g^{-1}a_ih^{-1}]\leq |B^G(id,n/C_0)|Ke^{-n/K},$$
 and the conclusion easily follows for $C_0$ large enough.
\end{proof}

We fix, for each $x,y\in G$, a geodesic $[x,y]$ connecting them.

\begin{lm}\label{fargeodesics}
 There exists $C_1$ with the following property. For each $n\geq 1$ we have
 $$\matP^\mu[\ d_G([id,Z_1],[Z_n,Z_{n+1}])\leq n/C_1\ ]\leq C_1 e^{-n/C_1}.$$
\end{lm}

\begin{proof}
For $C_0$ as in Lemma \ref{doublecoset} we have
\begin{align*}
 \matP^\mu[d_G([1,Z_1],&[Z_n,Z_{n+1}])\leq (n-1)/C_0]\\
  & \leq \sum_{h_1,h_2\in G} \sum_{x_i\in [id,h_i]} \matP^\mu[d_G(x_1, h_1 X_1^{-1}Z_{n}x_2)\leq n/C_0]\matP^\mu[X_1=h_1,X_{n+1}=h_2]\\
 & = \sum_{h_1,h_2\in G} \sum_{x_i\in [id,h_i]} \matP^\mu[d_G(x_1, h_1 Z_{n-1}x_2)\leq n/C_0]\matP^\mu[X_1=h_1,X_{n+1}=h_2]\\
 & \leq C_0 e^{-(n-1)/C_0} \sum_{h_1,h_2\in G} (d_G(id,h_1)+1)(d_G(id,h_2)+1)\matP^\mu[X_1=h_1,X_{n+1}=h_2]\\
 &= C_0 e^{-(n-1)/C_0} \matE^\mu[d_G(id,X_1)+1]^2,\\
\end{align*}
 as required, where to pass from the second to the third line notice that all terms in the sum $\sum_{x_i\in [id,h_i]}$ are equal, and $(d_G(id,h_1)+1)(d_G(id,h_2)+1)$ is just the number of terms.
\end{proof}

\begin{lm}\label{bowtie}
There exists $C_2$ with the following property. Let $\gamma$ be a geodesic in $G$ and let $(w_i)_{i=0,\dots, n}$ be a discrete path with endpoints on $\gamma$. Let $M\geq C_2$ be a positive integer. Then for each $k\in\{0,\dots, n\}$ one of the following holds.
\begin{enumerate}
 \item There exist $k_1< k \leq k_2$ with $|k_2-k_1|\leq M$ so that $d_G(w_{k_1},w_{k_1+1}), d_G(w_{k_2},w_{k_2+1})\geq (d_G(w_k,\gamma)- C_2)/M$.
 \item There exist $k_1< k \leq k_2$ with $|k_2-k_1|\geq M$ so that $d_G([w_{k_1},w_{k_1+1}], [w_{k_2},w_{k_2+1}])\leq (k_2-k_1)/C_1$.
 \item There exist $k_1<k \leq k_2$ with $|k_2-k_1|\geq M$ so that $\sum_{k_1\leq i< k_2} d_G(w_{i},w_{i+1})\geq e^{(k_2-k_1)/C_2}/C_2$.
\end{enumerate}
\end{lm}

\begin{proof}
We fix a choice $\pi_{\gamma}:G\to\gamma$ of closest point projection onto $\gamma$.

Denote $N^{\bowtie}(\gamma)$ be the closure of the set of all $x\in G$ so that $d_G(x,\pi_{\gamma}(x))\leq d_G(\pi_\gamma(w_k),\pi_{\gamma}(x))$.

Let $k_1<k$ be maximal (resp. $k_2\geq k$ be minimal) so that $[w_{k_1},w_{k_1+1}]\cap N^{\bowtie}(\gamma)\neq \emptyset$ (resp. $[w_{k_2},w_{k_2+1}]\cap N^{\bowtie}(\gamma)\neq \emptyset$).
 \par\medskip
{\bf Claim 1.} $d_G(w_{k_i},w_k)\geq d_G(w_k,\gamma)-100\delta$.
 \par\medskip
\emph{Proof of Claim 1.}
If $d_G(\pi_\gamma(w_k),\pi_\gamma(w_{k_i}))\leq 10\delta$ then 
$$d_G(w_k,\gamma)=d_G(w_k,\pi_\gamma(w_k))\leq d_G(w_k,w_{k_i})+d_G(w_{k_i},\pi_\gamma(w_{k_i}))+d_G(\pi_\gamma(w_{k_i}),\pi_\gamma(w_{k}))\leq d_G(w_k,w_{k_i})+20\delta.$$
If $d_G(\pi_\gamma(w_k),\pi_\gamma(w_{k_i}))\geq 10\delta$ then any geodesic from $w_{k}$ to $w_{k_i}$ passes $10\delta$-close to $\pi_\gamma(w_{k})$ and $\pi_\gamma(w_{k_i})$. Hence
$$d_G(w_{k},w_{k_i})\geq d_G(w_k,\pi_\gamma(w_k))+d_G(\pi_\gamma(w_k),\pi_\gamma(w_{k_i}))+d_G(\pi_\gamma(w_{k_i}),w_{k_i})-100\delta\geq d_G(w_k,\gamma)-100\delta.$$
\qed
 
In view of Claim 1, if we have $|k_1-k|\leq M$ and $|k_2-k|\leq M$ then 1) holds, possibly for different $k_1,k_2$. In fact, $\sum_{k'=k_1,\dots,k-1} d_G(w_{k'},w_{k'+1})\geq d_G(w_{k_1},w_k)\geq d_G(w_k,\gamma)-100\delta$, so that one of the terms of the sum is large, and similarly on the ``other side'' of $k$.

Hence, suppose that either $|k_1-k|> M$ or $|k_2-k|> M$, so that in particular $|k_2-k_1|\geq M$. Also, suppose that 2) does not hold, for the given $k_1,k_2$.

Let $\beta$ be the concatenation of a subpath of $[w_{k_1},w_{k_1+1}]$, $[w_i,w_{i+1}]$ for $i=k_1+1,\dots,k_2-1$ and a subpath of $[w_{k_2},w_{k_2+1}]$ with the property that $\beta$ intersects $N^{\bowtie}(\gamma)$ only at its endpoints $x,y$.

\par\medskip

{\bf Claim 2.} $l_G(\beta)\geq e^{d(x,y)/C_2}/C_2$.

\par\medskip

\emph{Proof of Claim 2.}
If we had $d_G(\pi_\gamma(x), \pi_\gamma(w_k)),d_G(\pi_\gamma(y), \pi_\gamma(w_k))\leq (k_2-k_1)/(5C_0)$ then we would have
$$d_G(x,y)\leq  d_G(x,\pi_\gamma(x))+d_G(\pi_\gamma(x),\pi_\gamma(y))+d_G(\pi_\gamma(y),y)\leq \frac{4}{5C_0} (k_2-k_1),$$
where we used the definition of $N^{\bowtie}(\gamma)$.  This contradicts the assumption that 2) does not hold.

Hence, let us say $d_G(\pi_\gamma(x), \pi_\gamma(w_k))\geq (k_2-k_1)/(5C_0)$, the other case being symmetric. The subpath $\beta'$ of $\beta$ from $x$ to $w_k$ avoids $B(\pi_\gamma(x), (k_2-k_1)/(10C_0))$, and $\pi_\gamma(x)$ lies $10\delta$-close to a geodesic from $x$ to $w_k$. Hence, the length of $\beta'$ is exponential in $k_2-k_1$, as required.
\qed

In view of Claim 2, condition 3) holds, and the proof of the lemma is complete.
\end{proof}

\begin{proof}[Proof of Theorem \ref{devhyp}]
Let $M\geq C_2$ be a positive integer (for $M\leq C_2$ the theorem trivially holds setting $C=C_2$). We claim that there exists $C_3$ so that with probability at least $1-C_3e^{-M/C_3}$ a sample path $(w_i)$ of $(Z_n)$ does not satisfy item 2) or 3) in Lemma \ref{bowtie}.

For given $k_1,k_2$, the probability that 2) holds is exponentially small in $k_2-k_1$ by Lemma \ref{fargeodesics}, and it is easily seen that the probability that 3) holds is also exponentially small in $k_2-k_1$. In fact, the inequality $\sum_{k_1\leq i< k_2} d_G(w_{i},w_{i+1})\geq e^{(k_2-k_1)/C_2}/C_2$ forces one of the summands to be exponentially large in $k_2-k_1$, but the probability of the existence of such a jump is exponentially small by Markov's inequality.

The claim now follows summing over all possible $k_1,k_2$, similarly to Lemma \ref{tightpath}.

In view of Lemma \ref{bowtie}, we then have
\begin{align*}
 \matP^\mu[d_G&(Z_k,[id,Z_n])\geq t]\\
&\leq C_3e^{-M/C_3}+ \sum_{\stack{k_1<k\leq k_2,}{|k_2-k_1|\leq M}} \matP^\mu[d_G(w_{k_1},w_{k_1+1})\geq (t- C_2)/M]\matP^\mu[d_G(w_{k_2},w_{k_2+1})\geq (t- C_2)/M] \\
& \leq C_3e^{-M/C_3}+ M^2(\matP^\mu[d(id,X_1)\geq (t- C_2)/M])^2,
\end{align*}
as required.

For the ``moreover'' part, we take $M$ of order $\log(t)$ in the expression above. The first term can then be made of order $t^{-4}$, while the second term, using Chebyshev, is of order $\log(t)^6/t^4$. Hence,
$$\matP^\mu[d_G(Z_k,[id,Z_n])\geq t]\leq C_4 \frac{\log(t)^6}{t^{4}}$$
for each large enough $t$ and a suitable constant $C_4$.
We can replace ``$d_G(Z_k,[id,Z_n])$'' in the expression above by ``$(id,Z_n)_{Z_k}$'' (up to modifying the constant), hence for each $0<\epsilon<1$, we have
$$\esp^\mu[((id,Z_n)_{Z_k})^{4-\epsilon}]=(4-\epsilon)\int_0^{\infty} t^{3-\epsilon}\matP^\mu[(id,Z_n)_{Z_k}\geq t],$$
which is finite in view of the estimate above.
\end{proof}

\section{Deviation for Green metrics}\label{sec:greendev} 

In this section we prove deviation inequalities for Green metrics.

\begin{theo}\label{smalldevgen}
Let $G$ be a finitely generated group acting acylindrically and non-elementarily on the geodesic hyperbolic space $X$.
Let $\mu_0$ be a symmetric measure on $G$ whose finite support generates $G$. Then $\mu_0$ has a neighborhood $\mathcal N$ with the following properties. There exists $C$ so that for each symmetric $\mu\in \mathcal N$, each $\mu'$ in $\mathcal N$, $0\leq k\leq n$ and $l\geq 1$ we have 
 $$\matP^{\mu'}\left[(id,Z_n)^{\calG_\mu}_{Z_k}\geq l\right]\leq Ce^{-l/C},$$
 where $(x,y)_w^{\calG_\mu}$ denotes the Gromov product in the Green metric $d^\mu_{\calG}$ with respect to $\mu$.
\end{theo}

The rest of this section is devoted to the proof of the theorem.
 
Fix the notation of the theorem from now on. When we write an inequality involving $\matP$, $d_\calG$ without explicit reference to the measure we mean that the statement holds for every $\mu\in \mathcal N$ and that the constants involved can be chosen uniformly for all $\mu\in\mathcal N$, where $\mathcal N$ is a small enough neighborhood of $\mu_0$. We fix the word metric $d_G$ on $G$ corresponding to the generating set $supp(\mu)$. We denote by $\gamma(g,h)$ any geodesic in $G$ from $g$ to $h$. Up to rescaling the metric of $X$, we can and will assume $d_X(g,h)\leq d_G(g,h)$ for all $g,h\in G$.

We denote by $C_i$ suitable constants depending on the data of the theorem, and for convenience we take $C_{i+1}\geq C_{i}$.

A $(T,S)$\emph{-linear progress point} $p\in \gamma(g,h)$ is a point that satisfies the following property. For each $q_1,q_2\in \gamma(g,h)$ with $d_G(p,q_i)\geq S$ and so that $q_1,p,q_2$ appear in this order along $\gamma(g,h)$, we have $d_G(q_1,q_2)\leq T d_X(q_1,q_2)$.

Denote by $\gamma(g,h)_{T,S}$ the collection of all $(T,S)$-linear progress points $p\in \gamma(g,h)$.

The theorem follows combining the two lemmas below. We remark that the measure $\mu'$ only plays a role in Lemma \ref{progresspoints}, while the measure $\mu$ only plays a role in Lemma \ref{closetoprogpoints}.

\begin{lm}\label{progresspoints}
There exists $T$ and $C_5$ so that for each $k$, $n\geq k$ and $l\geq 1$, 
$$\matP[d_G(Z_k,\gamma(id,Z_n)_{T,C_5l})\geq l]\leq C_5e^{-l/C_5}$$
\end{lm}

The idea is that points along a random path make linear progress in $d_X$ and stay $d_G$-close to $\gamma(id,Z_n)$, hence random points along $\gamma(id,Z_n)$ are of linear progress.

\begin{proof}
 From Theorem \ref{smalldevhier} and Theorem \ref{linprog} we know for each $k',k_1,k_2\leq n$, $l'\geq 1$:
 $$\matP[d_G(Z_{k'},\gamma(id,Z_n))\geq l']\leq C_1e^{-l'/C_1}$$
 and
$$\matP[d_X(Z_{k_1},Z_{k_2})\leq |k_1-k_2|/C_2]\leq C_2e^{-|k_1-k_2|/C_2}.$$

Also, as $\mu$ has finite support we have $\matP[d_G(Z_{k},Z_{k'})\leq C_3|k'-k|]=1$.

Let $\mathcal I$ be the set of all integers $i$ so that $k+i10C_2l\in\{0,\dots,n\}$. Summing over all $k',k_1,k_2\leq n$ of the form $k'=k+i10C_2l$ and with $l'=l+il$, we get that the probability that $ (a),(b),(c)$ hold for each $i,i_1,i_2\in\mathcal I$ with $i_1\leq 0\leq i_2$ is at least $1-C_4e^{-l/C_4}$, where

\begin{enumerate}[(a)]
 \item $d_X(Z_{k+i_1C_2l},Z_{k+ i_2C_2l})\geq 10(i_2-i_1)l$,
 \item $d_G(Z_{k+ iC_2l},\gamma(id,Z_n))\leq l+|i|l$,
 \item $d_G(Z_{k},Z_{k+ iC_2l})\leq |i|C_4l$.
\end{enumerate}

Hence, with probability at least $1-C_4e^{-l/C_4}$, along the geodesic from $id$ to the endpoint of a random walk $Z_n$, we have points $\{p_i\}$ so that
\begin{enumerate}
 \item $p_0$ is $l$-close to $Z_k$,
 \item $d_X(p_{i_1},p_{i_2})\geq 10(i_2-i_1)l-2l-(i_2-i_1)l\geq 7(i_2-i_1)l$ for $i_1\leq 0\leq i_2$ and $i_1\neq i_2$,
 \item $d_G(p_i,p_0)\leq |i|C_4l$.
\end{enumerate}

Such properties easily imply that $p_0$ is of $(T,C_5l)$-linear progress, as required.
\end{proof}

\begin{lm}\label{closetoprogpoints}
 Let $T\geq 1$. If $\mu\in\mathcal N$ is symmetric, then there exists $C_9$ so that if $p\in \gamma(g,h)$ is a $(T,S)$-linear progress point for some $S\geq 1$ then $d^\mu_\calG(g,p)+d^\mu_\calG(p,h)\leq d^\mu_\calG(g,h)+C_9S$.
\end{lm}

\begin{proof}
In this proof, by ``path'' we mean discrete path in the Cayley graph of $G$ with respect to the generating set $supp(\mu)$. The weight $W(\alpha)$ of a path $\alpha$ of length $n$ is the probability that a random walk of length $n$ driven by $\mu$ follows the path $\alpha$. Similarly, the weight $W(\calP)$ of a set of paths $\calP$ is the sum of the weights of the paths in the set.
Notice that the Green function between two points $g,h$ is $G^\mu(g^{-1}h)=W(\calP)=\sum_{\alpha\in \calP(g,h)}W(\alpha)$, where $\calP(g,h)$ is the set of all paths connecting $g,h$. 
Recall from the definitions in paragraph \ref{par:Green}  that $d^\mu_\calG(g,h)=-\log G^\mu(g^{-1}h)+\log G^\mu(id)$. Thus it will be sufficient to prove that, for points $g$, $h$ and $p$ as in the Lemma, we have 
$$G^\mu(g^{-1}p)G^\mu(p^{-1}h)\geq e^{-C_9S}G^\mu(id)G^\mu(g^{-1}h)\,.$$ 
Adjusting the value of $C_9$, we see that it suffices to prove that 
$$G^\mu(g^{-1}p)G^\mu(p^{-1}h)\geq e^{-C_9S}G^\mu(g^{-1}h)\,.$$ 

We denote $\gamma=\gamma(g,h)$ for convenience.

In the first part of the proof we show that a path avoiding a $d_G$-ball around $p\in\gamma$ has a long subpath with certain properties.

\emph{Choice of constants.} We first choose constants $\epsilon, \theta, N$ depending on the measure $\mu$ only.

There exists $\epsilon=\epsilon(\mu)>0$ so that the weight of any path of length $n$ is at least $\epsilon^n$. Moreover, we claim that there is a positive $\theta<1$ so that for each $n\geq 1/(1-\theta)$, the set of all paths of length at least $n$ between any two given points has weight at most $\theta^n$. In fact, there exists $\hat\theta=\rho_\mu<1$ so that the probability $\P^\mu[Z_m=a^{-1}b]$  is at most $\hat\theta^{m}$, see (\ref{eq:specrad}). This means that the set of all paths of length exactly $m$ between two given points has weight at most $\hat\theta^m$. Hence, the set of all paths of length at least $n$ between any two given points is at most $\sum_{m\geq n}\hat\theta^m=\hat\theta^n/(1-\hat\theta)\leq \theta^n$, for $\theta$ sufficiently close to $1$ and $n$ sufficiently large. Finally, let $N> \max\{2,1/(1-\theta)\}$ be so that $\theta^{N}/\epsilon\leq 1/2$.
 
Fix now $C_6$ so that $\rho(t)\geq 2NT$ for all $t\geq C_6$, where $\rho$ is as in Proposition \ref{superlinproj} with $Y=G$ (later we will also use the constant $C$ from that proposition). For convenience, we require that $C_6S$ is an integer. We then choose a  constant $C_7$ much larger than $C_6$; more specifically here is the list of all the inequalities we will need, in the form in which they will be used:

\begin{itemize}
\item $C_7>3 C_6$, and

for each $A\geq 2C_7S-4C_6S$ we have:
\item  $A-2C_6S\geq \frac{A}{2} + 2C_6S+ C$,
\item $A+2C_6S\leq 2 A$, and

 \item $\epsilon^{4C_6S}2^{-A}\leq 1$.
\end{itemize}

 \emph{Paths avoiding balls.} Let $\alpha$ be a path from $g$ to $h$ that avoids $B^G(p,C_7S)$. Then we claim that we can find a subpath $\beta$ of $\alpha$, pictured in Figure \ref{green_detour}, with the following properties. ``To the left'' and ``to the right'' refer to the natural order along $\gamma$. 
 
 \begin{enumerate}
 \item $\beta$ does not intersect the open $C_6S$-neighborhood of $\gamma$ in $G$.
 \item The endpoints $g',h'$ of $\beta$ are at $d_G$-distance $C_6S$ from $\gamma$.
 \item For some $g'',h''\in \gamma$ closest to $g',h'$ (in particular, with $d_G(g',g'')=C_6S$ and $d_G(h',h'')=C_6S$) we have that $g''$ is to the left of $p$ and $h''$ is to the right of $p$.
\end{enumerate}

In fact, we can take as $\beta$ the subpath of $\alpha$ joining
\begin{itemize}
 \item the last point $g'$ along $\alpha$ that is $C_6S$-close to a point in $\gamma$ to the left of $p$, and
 \item the first point $h'$ after $g'$ along $\alpha$ that is $C_6S$-close to a point to the right of $p$.
\end{itemize}

Let us check the required properties of $\beta$. Notice that since $C_7$ is large compared to $C_6$, no point $x$ along $\alpha$ can be simultaneously $C_6S$-close to points $x',x''$ on opposite sides of $p$ (otherwise we would have $d_G(p,x')\leq d_G(x',x'')\leq 2C_6S$ and $d_G(p,x)\leq 3C_6S<C_7S$). In particular, the endpoints of the path $\beta$ defined above do not lie in the open $C_6S$-neighborhood of $\gamma$ in $G$, and neither do other points along $\beta$, by definition. Also, $g'$ cannot be at distance less than $C_6S$ from $\gamma$, because otherwise the next point along $\alpha$ would still be within $C_6S$ of some point on $\gamma$ to the left of $p$. Similarly, $h'$ also lies at distance $C_6S$ of $\gamma$. Finally, the last property of $\beta$ holds by definition.

We now show that $\beta$ needs to be long compared to $d_G(g',h')$. Notice that we have $d_G(g'',h'')\geq d_G(g',h')-2C_6S$ since $d_G(g',g'')\leq C_6S$ and $d_G(h',h'')\leq C_6S$. Also:
\begin{align*}
d_G(g',h') & \geq d_G(g'',h'')-2C_6S=d_G(g'',p)+d_G(p,h'')-2C_6S\\
& \geq d_G(g',p)+d_G(p,h')-4C_6S \\
&  \geq 2C_7S-4C_6S.
\end{align*}
Hence, since $C_7$ is large enough compared to $C_6$, we get
$$d_X(g'',h'')\geq \frac{d_G(g'',h'')}{T}\geq \frac{d_G(g',h')-2C_6S}{T}\geq \frac{d_G(g',h')}{2T} + 2C_6S+ C.$$
Hence, using Proposition \ref{superlinproj} we get
\begin{align*}
 \max\{l_G(\beta), d_G(g'',h'')\} & \geq (d_X(g'',h'')- 2C_6S-C) \rho(C_6S) \\
 & \geq \frac{d_G(g',h')}{2T} \rho(C_6S) \\
 & \geq N d_G(g',h').\\
\end{align*}

But since $C_7$ is large enough compared to $C_6$, we have $d_G(g'',h'')\leq d_G(g',h')+2C_6S \leq 2d_G(g',h')< N d_G(g',h')$, so we get
$$l_G(\beta)\geq N d_G(g',h').\ \ \ (*)$$

\par\medskip

{\bf Claim.} Let $P_a$ be the collections of all paths from $g$ to $h$ that avoid the ball of radius $C_7S$ around $p$, and let $P_t$ be the collection of those that intersect it. Then $W(P_a)\leq W(P_t)$.

\par\medskip

\emph{Proof of Claim.} We will construct a map $\psi:P_a\to P_t$ with the property that the preimage of any $\alpha'\in P_t$ has weight at most $W(\alpha')$. This suffices to establish the claim.


Let $\alpha\in P_a$, and let $\beta$ be a subpath as described above, which has $d_G$-length at least $N d_G(g',h')$ by $(*)$. 

Let $\hat{\beta}$ be the concatenation of a geodesic from $g'$ to $g''$, the subgeodesic of $\gamma$ from $g''$ to $h''$, and a geodesic from $h''$ to $h'$. Also, let $\psi(\alpha)$ be the concatenation of the initial subpath $\alpha_1$ of $\alpha$ with final point $g'$, $\hat{\beta}$ and the final subpath $\alpha_2$ of $\alpha$ starting at $h'$. It is clear that $\psi(\alpha)\in P_t$, and that in fact $\psi(\alpha)$ contains a subgeodesic of $\gamma$ that contains $p$.

\begin{figure}[h]
\centering
 \includegraphics[scale=0.7]{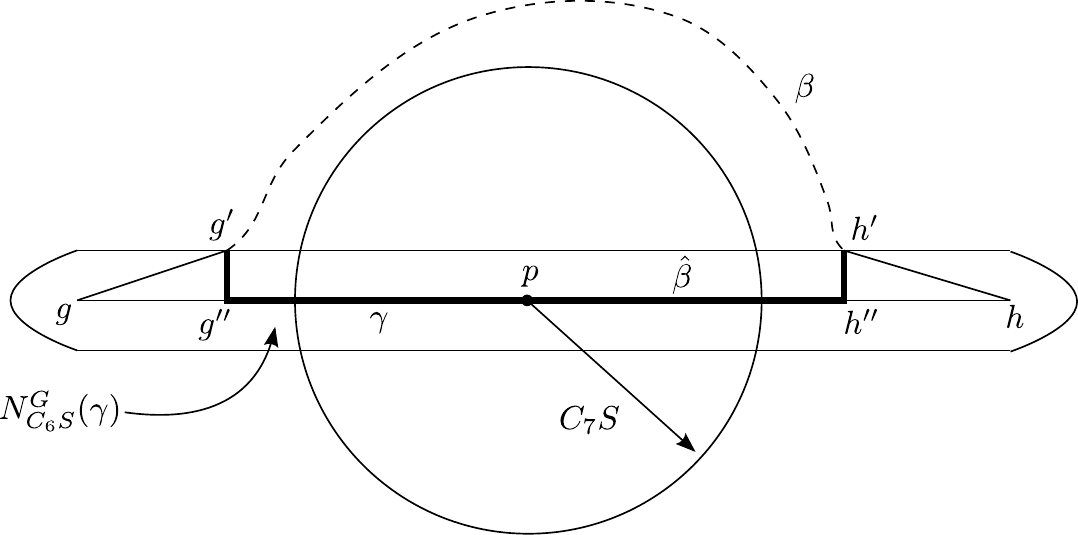}
 \caption{We can replace the subpath $\beta$ of $\alpha$ with the much shorter path $\hat\beta$. This operation increases the weight.}\label{green_detour}
\end{figure}

The map $\psi$ is not 1-1. However, given any $\alpha'\in P_t$ and $\alpha\in\psi^{-1}(\alpha')$, we have that $\alpha$ is obtained from $\alpha'$ by replacing a subpath $\hat{\beta}$ by another path, say with endpoints $g',h'$. The points $g',h'$ can be identified in the following way. Consider the maximal common subgeodesic of $\alpha'$ and $\gamma$ containing $p$. The endpoints need to be the points $g'',h''$ described above, so that $g'$ and $h'$ can be found by moving along $\alpha'$ an extra $C_6S$ away from $p$.

Notice that $W(\alpha)= W(\alpha_1)W(\alpha_2)W(\beta),$ and similarly for $\alpha'$ with $\hat{\beta}$ replacing $\beta$. Also, $W(\hat\beta)\geq \epsilon^{d_G(g',h')+4C_6S}$, and the weight of all paths of length at least $N d_G(g',h')$ from $g'$ to $h'$ is at most $\theta^{N d_G(g',h')}$. Hence, the weight of the preimage of $\alpha'$ is at most
$$W(\alpha_1)W(\alpha_2)\theta^{N d_G(g',h')}\leq W(\alpha')\theta^{N d_G(g',h')}/\epsilon^{d_G(g',h')+4C_6S}\leq W(\alpha')\epsilon^{4C_6S}2^{-d_G(g',h')}.$$
Recalling that $d_G(g',h')\geq  2C_7S-4C_6S$, and since $C_7$ is large enough compared to $C_6$ we have $\epsilon^{4C_6S}2^{-d_G(g',h')}\leq 1$, as required.\qed

\par\smallskip

Expanding the definition of the Green metric, one sees that it suffices to show the following. Let $P(a,b)$ be the collection of all paths from $a$ to $b$. Then
$$W(P(g,h))\leq C_8 W(P(g,p)) W(P(p,h)).$$

We now claim that $ W(P(g,p))W(P(p,h))\geq \epsilon^{2C_7S}W(P_t)$. In order to prove the claim, it suffices to describe an injective map $\psi_g\times\psi_h:P_t\to P(g,p)\times P(p,h)$ so that for each $\alpha\in P_t$ we have $W(\psi_g(\alpha))W(\psi_h(\alpha))\geq \epsilon^{2C_7S} W(\alpha)$. Once this is done, we get the desired conclusion from the inequality $$W(P(g,p))W(P(p,h))=\sum_{\alpha_1\in P(g,p),\alpha_2\in P(p,h)} W(\alpha_1)W(\alpha_2) \geq \sum_{\alpha\in P_t}W(\psi_g(\alpha))W(\psi_h(\alpha)).$$

Starting from a path $\alpha\in P_t$, we can form a path $\psi_g(\alpha)$ by concatenating the initial subpath of $\alpha$ to the first point in $\alpha\cap B^G(p,C_7S)$ and a path $\beta_\alpha$ of length $\leq C_7S$ to $p$. The path $\psi_h(\alpha)$ is the concatenation of the inverse of $\beta_\alpha$ and the suitable final subpath of $\alpha$. The map $\psi_g\times\psi_h$ is easily seen to be injective, and the required weight condition follows from the fact that the weight of $\beta_\alpha$ is at least $\epsilon^{C_7S}$.

We are now ready to conclude the proof of the lemma: 
\begin{align*}
 W(P(g,h)) & = W(P_a\cup P_t)\leq 2W(P_t) \\
 & \leq 2\epsilon^{-2C_7S}W(P(g,p))W(P(p,h)),\\
\end{align*}
as required.
\end{proof}


\part{Conclusions}

\section{Statements of the main results}\label{sec:statements}
For the convenience of the reader and for future reference, in this section we state the theorems that can be obtained combining results that rely on deviation inequalities, proven in the first half of the paper, with results about getting deviation inequalities from the second part of the paper.

First of all, we collect the results about the regularity of the rate of escape. Acylindrical and non-elementary actions are defined in Section \ref{prelim}, while the notion of being acylindrically intermediate is given in Definition \ref{def:acylinter}. (Recall that if the finitely generated group $G$ acts acylindrically on the geodesic hyperbolic space $X$, then any Cayley graph of $G$ and $X$ itself are acylindrically intermediate for $(G,X)$, see Proposition \ref{prop:examplesacylinter} for more examples). Recall that we fixed the convention that, when we have a fixed action of a group $G$ on some metric space $Y$, we automatically fix a basepoint $y\in Y$ and denote $d_Y$ the metric on $G$ defined by $d_Y(g,h)=d_Y(gy,hy)$ for each $g,h\in G$. Finally, recall that, for $G$ a group and $B\subseteq G$, we defined $\p(B)$ to be the set of all measures on $G$ supported on $B$ and that, for $d$ a metric on $G$ and $\mu$ a measure on $G$ with finite first moment, we defined the rate of escape as $\ell(\mu;d):=\lim_{n\ra\infty}\frac 1 n \sum_{x\in G} d(id,x)\mu^n(x)$. We defined a distance on $\p(B)$ in paragraph \ref{ssec:distance}. 

The following theorem is obtained combining Theorem \ref{smalldevhier} with Theorems \ref{theo:rescp} and \ref{theo:diffrate}.

\begin{theo}
 Let $G$ be a finitely generated group acting acylindrically on the geodesic hyperbolic space $X$, and let $Y$ be acylindrically intermediate for $(G,X)$. Let $\mu$ be a measure on $G$ with exponential tail whose support $B$ generates a non-elementary semigroup of $G$. Then
 \begin{enumerate}
  \item there exists a neighborhood of $\mu$ in $\p(B)$, say $\cal N$, such that  
the function $\mu_0\ra \ell(\mu_0;d_Y)$ is Lipschitz continuous on $\cal N$.
\item the function $\mu_0\ra \ell(\mu_0;d_Y)$ is differentiable at $\mu_0=\mu$ in the following sense: 
Let $(\mu_t, t\in[0,1])$ be a curve in $\p(B)$ such that $\mu_0=\mu$ and, for all $a\in B$, the function 
$t\ra \log \mu_t(a)$ has a derivative at $t=0$, say $\nu(a)$. We assume that $\nu$ is bounded on $B$ and also that 
$\sup_{t\in[0,1]}\sup_{a\in B} \vert \frac 1 t \log\frac{\mu_t(a)}{\mu_0(a)} -\nu(a)\vert<\infty$. Then the limit of $\frac 1 t (\ell(\mu_t;d_Y)-\ell(\mu;d_Y))$ as $t$ tends to $0$ exists. 
Furthermore, this limit coincides with the covariance
\beqn\sigma(\nu,\mu;d_Y):=\lim_n\frac 1 n \esp^\mu[d_Y(id,Z_n)\big(\sum_{j=1}^n \nu(X_j)\big)]\,.\eeqn
 \end{enumerate}
\end{theo}

Here is instead the statement about the asymptotic entropy, obtained combining Theorem \ref{smalldevgen} with Theorems \ref{theo:entropsym} and \ref{theo:diffentrop}. Recall that, whenever $G$ is a group and $\mu$ is a measure on $G$, we define the entropy $H(\mu):=\sum_{x\in G} (-\log \mu(x))\,\mu(x)\,$ and the asymptotic entropy $h(\mu):=\lim_{n\ra\infty}\frac 1 n H(\mu^n)$ (when $H(\mu)$ is finite). We denote $\p_s(B)$ the set of all symmetric measures supported on $B$.

\begin{theo}
 Let $G$ be a finitely generated group acting acylindrically and non-elementarily on the geodesic hyperbolic space $X$.
Let $B$ be a finite symmetric set that generates $G$, and let $\mu$ be a symmetric probability measure with support $B$. Then
\begin{enumerate}
 \item there exists a neighborhood of $\mu$ in $\p_s(B)$, say $\cal N$, such that the function $\mu_0\to h(\mu_0)$ is Lipschitz continuous on $\cal N$,
 \item the function $\mu_0\ra h(\mu_0)$ is differentiable at $\mu_0$ in the following sense: 
Let $(\mu_t, t\in[0,1])$ be a curve in $\p(B)$ such that $\mu_0=\mu$ and, for all $a\in B$, the function 
$t\ra \log \mu_t(a)$ has a derivative at $t=0$, say $\nu(a)$. We assume that $\nu$ is bounded on $B$ and also that 
$\sup_{t\in[0,1]}\sup_{a\in B} \vert \frac 1 t \log\frac{\mu_t(a)}{\mu_0(a)} -\nu(a)\vert<\infty$. Then the limit of $\frac 1 t (h\mu_t)-h(\mu))$ as $t$ tends to $0$ exists. 
Furthermore, this limit coincides with the covariance
\beqn\sigma_{\cal G}(\nu,\mu):=\lim_n\frac 1 n \esp^\mu[d_{\cal G}^\mu(id,Z_n)\big(\sum_{j=1}^n \nu(X_j)\big)]\,.\eeqn
\end{enumerate}
\end{theo}

We now proceed with the bound on (higher) moments of measures, obtained combining Theorem \ref{smalldevhier} with Theorem \ref{theo:othermoments}.

\begin{theo}
 Let $G$ be a finitely generated group acting acylindrically on the geodesic hyperbolic space $X$, and let $Y$ be acylindrically intermediate for $(G,X)$. Let $\mu$ be a measure on $G$ with exponential tail whose support generates a non-elementary semigroup. Then for all $p>1$ there exists a constant $c=c(\mu,p)$ such that for all $n\geq 0$ we have
$$\esp^\mu[\vert d_Y(id,Z_n)-\esp^\mu[d_Y(id,Z_n)]\vert^p]\le c n^{p/2}\,.$$
\end{theo}

Finally, we conclude with the Central Limit Theorem. Items (1) and (3) are obtained combining Theorem \ref{smalldevhier} with Theorems \ref{theo:existvariance} and \ref{theo:clt}. Item (2) follows from Theorem \ref{theo:downvariance} in view of the fact that, in the notation set below, we have $\ell(\mu;d_X)>0$ by Theorem \ref{linprog}, and whence a fortiori we also have $\ell(\mu;d_Y)>0$.

\begin{theo}\label{theo:CLT_conclusions}[Central Limit Theorem]
 Let $G$ be a finitely generated group acting acylindrically on the geodesic hyperbolic space $X$, and let $Y$ be acylindrically intermediate for $(G,X)$. Let $\mu$ be a measure on $G$ with exponential tail whose support generates a non-elementary semigroup of $G$. Then the following hold.
 \begin{enumerate}
  \item $\frac 1 n \var^\mu(d_Y(id,Z_n))$ has a limit as $n$ tends to $\infty$, which we denote by $\sigma^2$.
  \item $\sigma^2>0$.
 \item The law of $\frac 1{\sqrt{n}} (d_Y(id,Z_n)-\ell(\mu;d_Y)n)$ under $\P^\mu$ weakly converges to the Gaussian law with zero mean and variance $\sigma^2$.
 \end{enumerate}
\end{theo}

\bibliographystyle{alpha}
\bibliography{biblio}

\end{document}